\documentclass[times]{amsart}
\usepackage{bbold}
\usepackage{amssymb}
\usepackage{amsmath}
\usepackage{graphicx,enumerate,amssymb}
\graphicspath{{./}{figures/}}
\usepackage{algorithm,algpseudocode}     
\usepackage{geometry}
\newcommand{\RR}{\mathbb R}
\newcommand{\Dm}{D_{\scriptscriptstyle{-}}}
\newcommand{\Dp}{D_{\scriptscriptstyle{+}}}

\newtheorem{proposition}{Proposition}

\begin{document}

\title{Modeling and Optimal Control of an Octopus Tentacle}

\author[S. Cacace]{Simone Cacace}
\address[S. Cacace]
{Dipartimento di Matematica e Fisica,
	 Universit\`{a} degli Studi Roma Tre, Largo S. Murialdo, 1,
	00154  Roma, Italy, cacace@mat.uniroma3.it}
\author[A. C. Lai]{Anna  Chiara Lai}
\address[A. C. Lai]{ Dipartimento di Scienze di Base
	e Applicate per l'Ingegneria,
	Sapienza Universit\`{a} di  Roma,
	Via A. Scarpa, 16,
	00161  Roma, Italy, anna.lai@sbai.uniroma1.it}

\author[P. Loreti]{Paola Loreti}
\address[P. Loreti]{Dipartimento di Scienze di Base
	e Applicate per l'Ingegneria,
	Sapienza Universit\`{a} di  Roma,
	Via A. Scarpa, 16,
	00161  Roma, Italy, \textsf{paola.loreti@sbai.uniroma1.it}}

\begin{abstract}
We present a control model for an octopus tentacle, based on the dynamics of an inextensible string with curvature constraints and curvature controls. 
We derive the equations of motion together with an appropriate set of boundary conditions, and we characterize the corresponding equilibria.  
The model results in a system of fourth-order evolutive nonlinear controlled PDEs, generalizing the classic Euler's \emph{dynamic elastica} equation, 
that we approximate and solve numerically introducing a consistent finite difference scheme. 
We proceed investigating a reachability optimal control problem associated to our tentacle model. 
We first focus on the stationary case, by establishing a relation with the celebrated Dubins' car problem. 
Moreover, we propose an augmented Lagrangian method for its numerical solution. 
Finally, we address the evolutive case obtaining first order optimality conditions, then we numerically solve the optimality system by means 
of an adjoint-based gradient descent method.
\end{abstract}
\maketitle

\section{Introduction}
Octopus tentacles are extremely complex and fashinating biological structures,  
whose study is motivated by both natural science interests as well as 
applications to robotics, in particular, in the framework of soft manipulators \cite{nature}. 
In this paper we are interested in an optimal control problem associated to the dynamics of an octopus tentacle. The aim 
is to model the voluntary muscle 
contractions of the tentacle and to derive optimality conditions for a reachability problem. 
To this end, we propose a continuous model, based on the dynamics of an inextensible string 
with curvature constraints and curvature controls. More precisely, our model in an extension of the \emph{Euler's dynamic elastica} equation 
(see equation \eqref{dee} below) in the following sense: we add to the exact inextensibility constraint and the bending momentum, 
a term representing a curvature constraint and a control term. The curvature constraint prevents the tentacle from bending over a fixed treshold; while the control term, 
which we consider the main novelty of this model, forces pointwise the curvature of the tentacle.  
After characterizing explicitly the equilibria of the tentacle system as functions of the controls, 
both in terms of shape configuration and tension, we introduce a finite difference scheme for its numerical approximation, and 
we present some simulations validating the results. 
Then, we address the following optimal control problem: 
steer the tentacle tip to a target point minimizing the activation of the tentacle muscles. 
More precisely, we want to optimize the distance of the tentacle tip from a fixed target point and a quadratic cost associated to the controls. 
This is done in two different setting, a stationary case and a dynamic case. 
In the stationary case, we obtain a characterization of the reachable set for the tentacle tip, by establishing a relation with a 
classical motion planning problem first posed by Markov in 1889 \cite{markov}, also known as \emph{Dubins car problem}. 
Moreover, we provide an augmented Lagrangian method for the numerical solution of the optimal control problem and present some simulations.
In the dynamic case, the objective functional accounts for the whole evolution of the tentacle on a time interval $[0,T]$, 
plus a term representing the kinetic energy at the final time $T$. The aim is then also to stop the tentacle as soon as possible. 
To this end we introduce an adjoint state and derive first order optimality conditions. The resulting optimality system is then solved 
numerically via an adjoint-based gradient descent method and some simulations are presented.
We remark that, to the best of our knowledge, the present theoretic approach, in particular the dynamic case based on the optimal control of PDEs 
(see e.g. \cite{optimalpdebook}), appears to be new.

We now outline the main features of our model, while postponing a more formal justification to Section \ref{Sec:Model}. 
The unknowns are a curve $q(s,t)\in\RR^2$ describing the shape of the tentacle, and the associated inextensibility multiplier $\sigma(s,t)\in \RR$, 
playing the role of the tension of the tentacle. We remark that the variable $s\in[0,1]$ is an arclength coordinate parametrizing the curve $q$, 
and we denote by $q_s$, $q_{ss}$, $q_{tt}$ partial derivatives in space and time respectively.
The tentacle model includes:
\begin{enumerate}
\item[-] the inextensibility constraint, given by the equation $|q_s|^2=1$ and encompassed in the dynamics via a scalar Lagrange multiplier $\sigma$;
\item[-] the bending momentum, i.e. an intrinsic resistance of the tentacle to bending. It is described by an elastic potential of the form $\varepsilon(s)|q_{ss}|^2$, 
where the function $\varepsilon$ represents a non uniform bending stiffness; 
\item[-] a curvature constraint, given by the inequality $|q_{ss}|\le \omega(s)$, for some threshold positive function $\omega$. 
It is imposed using an unilateral potential of the form $\nu(s)\left(|q_{ss}|^2-\omega^2(s)\right)_+^2$, where $\nu$ is a non uniform elastic constant 
and $(\cdot)_+$ denotes the positive part of its argument;
\item[-] a control term forcing the curvature of the tentacle, given by the equality constraint $q_s \times q_{ss}=\omega(s) u(s,t)$, 
for a control map $u$ valued in $[-1,1]$. It is imposed using an elastic potential of the form 
$\mu(s)\left(\omega(s) u(s,t)-q_s \times q_{ss}\right)^2$, where $\mu$ is the corresponding non uniform elastic constant and 
$q_s \times q_{ss}$ represents the signed curvature (see Section \ref{Sec:Model} below for the definition of the vector product $\times$ in $\RR^2$).
\end{enumerate}
\noindent The tentacle dynamics is then described by the following system of fourth order, evolutive, non-linear, controlled PDEs: for $(s,t)\in(0,1)\times(0,T)$
\begin{equation}\label{sysintro}
 \left\{
  \begin{array}{l}
   \rho q_{tt}=\left(\sigma q_s-H q_{ss}^\bot\right)_s -\left(G q_{ss}+H q_{s}^\bot\right)_{ss}\\
   |q_s|^2=1 
  \end{array}
 \right.
\end{equation}
where 
$$G[q,\nu,\varepsilon,\omega](s,t):=\varepsilon(s)+\nu(s)\left(|q_{ss}(s,t)|^2-\omega^2(s)\right)_+,$$
$$H[q,\mu,u,\omega](s,t):=\mu(s)\left(\omega(s) u(s,t)-q_s(s,t) \times q_{ss}(s,t)\right)\,,$$
the function $\rho(s)$ represents the non uniform mass distribution of the tentacle, 
and the symbol $v^\bot$ denotes the counterclockwise orthogonal vector to $v$. The system is completed with appropriate boundary and initial conditions.\\
To put the tentacle model into perspective, we now briefly discuss its relation with some well-known equations in the literature, depending on specific 
choices of the parameters, summarized in the following table: \\  
\begin{center}
 \begin{tabular}{l|c|c|c|c}
  &Inextensibility&Bending&Curvature&Curvature\\		
  &&momentum& constraints&controls
\\
		\hline
		Euler-Bernoulli beam&$\times$&\checkmark&$\times$&$\times$\\
		Whip equation&\checkmark&$\times$&$\times$&$\times$\\
		Dynamic elastica&\checkmark&\checkmark&$\times$&$\times$\\
		Tentacle equation \eqref{sysintro}&\checkmark&\checkmark&\checkmark&\checkmark\\
	\end{tabular}
\end{center}
\medskip
\noindent{\em Euler-Bernoulli beam}. 
This model describes the vibration of a beam, depending on the given boundary conditions. A typical choice is the cantilevered case, 
in which one endpoint of the beam is clamped and the other one is free. To retrieve it, we neglect the inextensibility constraint, 
we set $\nu(s)=\mu(s)\equiv 0$ and $\varepsilon(s)=const$, so that \eqref{sysintro} reduces to
\begin{equation}\label{ebb2}
 \rho q_{tt}=-\varepsilon q_{ssss}\,.
\end{equation}
Assuming analytical initial data and uniform mass distribution $\rho$, Equation \eqref{ebb2} admits a unique classical (vector) solution given by
\begin{align*}q(s,t)=\sum_{k\in \mathbb N}a_ke^{-i\varphi_k t}\left(\cosh(\beta_k s)-\cos(\beta_k s)+\gamma(\beta_k)(\sinh(\beta_ks)-\sin(\beta_k s))\right)+s q_s(t,0)\,,
\end{align*}
where the coefficients $a_k\in \mathbb C^2$ depend only on the initial data, the sequence of real numbers $\{\beta_k\}$ is composed by the positive solutions 
of the equation 
$$\cosh(\beta)\cos(\beta)+1=0\,,$$
while $\varphi_k$ and $\gamma(\beta_k)$ are given, for each $k$, respectively by
$$\varphi_k=\beta_k^2\sqrt{\frac{\varepsilon}{\rho}}\,,\qquad \gamma(\beta_k)=\frac{\cosh \beta_k+\cos \beta_k}{\sinh \beta_k+\sin \beta_k}\,.$$

\medskip
\noindent{\em Whip equation}. Assuming $\nu(s)=\mu(s)=\varepsilon(s)\equiv 0$, 
Equation \eqref{sysintro} reduces to a second-order differential equation modeling the motion of an inextensible string
\begin{equation}\label{whip}
\left\{
  \begin{array}{l}
   \rho q_{tt}=\left(\sigma q_s\right)_s\\
   |q_s|^2=1 
  \end{array}
 \right.
\end{equation}
Typical boundary conditions prescribe one fixed endpoint and zero tension at the other endpoint. 
The problem results in a nonlinear system with a boundary layer at the free end, where the motion degenerates into a first order equation in space. 
To the best of our knowledge, one of the most complete results for this system is given in \cite{preston}, proving existence and uniqueness of solutions in suitable weighted Sobolev spaces. 

\medskip
\noindent{\em Euler's dynamic elastica}. Assuming $\nu(s)=\mu(s)\equiv 0$ and $\varepsilon(s)=const$, 
we obtain the following equation, also known as the nonlinear Euler-Bernoulli beam equation 
\begin{equation}\label{dee}
\left\{
  \begin{array}{l}
   \rho q_{tt}=\left(\sigma q_s\right)_s -\varepsilon q_{ssss}\\
   |q_s|^2=1 
  \end{array}
 \right.
\end{equation}
Classical boundary conditions prescribe the position and the tangent of both endpoints. In this case, the problem is well-posed and admits a smooth solution, 
representing the curve of fixed length that connects the two endpoints minimizing the integral of the squared curvature.

The problem of modeling and control soft inextensible bodies, including octopus tentacles, 
has been attacked with several approaches, including tools
coming from calculus of variations, control theory, motion planning for robotics. This resulted in a wide scientific literature which is 
almost impossible to exhaustively summarize. We then conclude this introduction by reporting only the papers that mostly inspired our work.

For general references on the modeling of octopus tentacles we mention the series of papers \cite{octopusbiomechanics,octopusbiomechanics2}, 
where a rich discrete control model is proposed: note that the main difference with the present paper is that we discuss a simpler but fully continuous
model. 
For the modeling framework, we refer to the survey \cite{stringschainsropes} for an introduction on the derivation of 
equations for inextensible strings (with no curvature constraints) and their discrete counterpart, the chains.
The papers \cite{ropedynamics,nonlinearhanging} treat the same topic, stressing the numerical aspects of the problem.
The paper \cite{vurpoimportante} provides a nice investigation of boundary conditions for the Euler's dynamic elastica 
and presents some numerical simulations --- also, see the paper \cite{discreteelastica} for a discussion of the discrete case.
The papers \cite{finger, hand} contain a discussion  in  a control and number theoretic setting of the discrete, stationary case with an exponential decay in the mass distribution. The paper \cite{fibonacci} includes more general, self-similar mass distributions; we finally refer to  \cite{ems} for a first investigation of the continuous counterpart of these models. 

The problem related to curvature constraints and controls is widely investigated in the framework of soft robotics: the problem is in general to 
prescribe the curvature of a flexible device in a distributed fashion, in order to perform locomotion or grasping, see \cite{science,biomimeticrobotoctopus}. 
Among many others, we refer to \cite{kinematicssoft} for the study of the kinematics of a soft-robot prototype. 
In \cite{octopusrobotprescribed,dynamicoctopusrobot}, after deriving the dynamics of a discrete tentacle-like actuator, 
a control strategy is prescribed a priori, and its efficiency is either numerically confirmed or measured by actual prototypes, respectively. 
An interesting approach based on machine learning can be found in \cite{learning}. 

One our main results consists in showing that equilibria of \eqref{sysintro} solve an ordinary differential equation recurring in motion planning problems. 
For an overview on the aspects that are more related to our approach, we refer to the paper \cite{dubinsapproximation} and the references therein. 
In order to characterize the reachable set of the equilibria, we rely on the results in \cite{cockayne}, see also \cite{reachablecarmotion}. 
Finally, we remark that the stationary case studied in this paper is very close to the framework of \emph{Euler's elastica}, i.e., the study of static equations 
for elastic rods -- see \cite{elasticahystory} for an hystorical excursus on the problem, and the papers 
\cite{elasticastability, dynamicelastica, elasticapulvirenti} and the references therein for an overview.  \\

The paper is organized as follows. Section \ref{Sec:Model} is devoted to our octopus tentacle model, 
including the derivation and the numerical approximation of its equations of motion, and also the characterization its equilibria. 
In Section \ref{Sec:Static}, we present and numerically solve the optimal control problem in the stationary case. Moreover, 
we characterize the reachable set of the tentacle tip. Finally, in Section \ref{Sec:Dynamic}, 
we address the optimal control problem in the dynamic case. We derive first order optimality conditions, and we numerically solve the corresponding optimality system.   
For each section, some numerical simulations complete the presentation. 

\section{A control model for an octopus tentacle}\label{Sec:Model}
This section is devoted to our octopus tentacle model. Starting from a discrete particle system describing a non uniform $N$-pendulum, 
we encode all the geometrical constraints into a discrete Lagrangian for the system, then we obtain the continuous 
model as a formal limit for the number of particles $N$ going to infinity. We proceed by applying the least action principle to obtain the equation of motions, 
and we study its equilibrium configurations depending on the given control map. Finally, we discretize the equations by means of a finite difference scheme 
and we present some simulations for validating the model. 

Let us fix some notations. The scalar product between two vectors $q,w\in \RR^2$ is denoted 
by $q\cdot w$, whereas the Euclidean norm is denoted by $|q|=\sqrt{q\cdot q}$. Moreover, we define the counterclockwise orthogonal vector
$$q^\bot:=\begin{pmatrix}
              0&1\\
              -1&0
             \end{pmatrix}q$$
and, with a little abuse of notation, we denote by $q\times w$ the vector product in $\RR^2$, namely the scalar quantity 
$q\times w:=q\cdot w^\bot$. 
We use the symbol $(\cdot)_+$ to denote the positive part function: $(x)_+=x$ if $x>0$ and $(x)_+=0$ otherwise. We also use the symbol $0$ to denote the null vector in $\RR^2$.

\subsection{From a discrete particle system to a continuous tentacle}\label{sdiscrete}
Our first step of modeling is to consider the tentacle as a three-dimensional body with an axial 
symmetry, with a fixed endpoint and characterized by longitudinal 
inextensibility: the evolution of the curve on the plane representing the 
symmetry axis is the object of our investigation. Roughly speaking, we neglect the torsion of the curve and reduce the problem in a two-dimensional setting, 
approximating the whole tentacle with a planar inextensible string. 
We proceed by introducing a discretization, which is, as a matter of fact, 
the cornerstone of both the desired continuous model and the numerical scheme for its simulation. 
We sample the shape of the tentacle by $N$ particles (or joints) with positions $q_k\in\RR^2$ and masses $m_k>0$, for $k=1,...,N$. 
We pose an additional particle $q_0$ fixed at the origin, representing the anchor point of the tentacle.
Moreover, for formal computations, we also add two ghost particles $q_{-1}$ and $q_{N+1}$ at the endpoints: 
the first particle is aligned to the vertical axis, i.e. $q_{-1}=q_0+\ell e_2$, 
where $\ell>0$ is meant to represent the length of two consecutive joints and to go to zero as $N\to\infty$, while $e_2$ denote the unit vector $(0,1)$; 
the last particle is aligned with the two preceding joints, i.e. $q_{N+1}=q_N+(q_N-q_{N-1})$. 
We introduce the dependency of the joints on time and, with a little abuse of notation, we denote the $(N+3)$-tuple of vectors $(q_{-1}(t),q_0(t),...,q_{N+1}(t))$ by $q(t)$, 
or simply $q$ depending on the context. Similarly, we denote the time derivative by $\dot q$.

We now impose the inextensibility constraint, assuming that the distance between consecutive joints is exactly $\ell$, i.e., for $k=1,...,N$
$$
|q_{k}-q_{k-1}|=\ell\,.
$$
We remark that, in most cephalopods species, a tentacle is thicker when close to the body and 
thinner when reaching the free endpoint. This morphological aspect can be easily encoded by assuming that the sequence of masses $m_k$ is decreasing in $k$. 
Note that, at this stage, the considered particle system can be viewed as just a non uniform $N$-pendulum, or a chain.

To deal with the ability of the tentacle to bend (as if it were a beam), we introduce some additional constraints involving its curvature. 
We assume that, for $k=1,...,N$,
the angle formed by the three consecutive joints $q_{k-1}, q_k, q_{k+1}$ is bounded by a maximum angle $\alpha_k$. 
This relation can be easily expressed by the following inequality constraint:
$$(q_{k+1}-q_k)\cdot(q_{k}-q_{k-1})\ge \ell^2\,\cos(\alpha_k)\,.$$
On the other hand, we take into account a bending momentum, namely we assume that the tentacle has an intrinsic resistance to bending. The 
position at rest corresponds to null relative angles, as described by the following equality constraint:
$$(q_{k+1}-q_k)\times(q_{k}-q_{k-1})=0\,.$$
Finally, we endow our model with a control term, pointwise prescribing the curvature of the tentacle. More precisely, we can force the exact angle between 
the joints $q_{k-1}, q_k, q_{k+1}$ by means of the following equality constraint: 
$$(q_{k+1}-q_k)\times(q_{k}-q_{k-1})=\ell^2\,\sin(\alpha_k u_k)\,,$$
where $u_k\in[-1,1]$ is the control associated to $k$-th joint. 
In this way, we take into account the fact that an octopus tentacle is composed by independent muscular tissues. 
The local contraction of a muscle, which is the way an octopus controls its limbs, prescribes indeed the local curvature. 

The next step is to form a suitable Lagrangian for our discrete model. To this end, we impose the three curvature constraints discussed above via penalization, 
namely we define respectively the following quadratic potentials: for $k=1,...,N$
$$G_k(q):=\nu_k\left(\cos(\alpha_k)-\frac{1}{\ell^2}(q_{k+1}-q_{k})\cdot(q_{k}-q_{k-1})\right)_+^2\,,$$
$$B_k(q):=\varepsilon_k\Big( (q_{k+1}-q_{k})\times(q_{k}-q_{k-1})\Big)^2\,,$$
$$H_k(q):=\mu_k\left(\sin(\alpha_ku_k)-\frac{1}{\ell^2}(q_{k+1}-q_{k})\times(q_{k}-q_{k-1})\right)^2\,,$$
where the penalty parameters $\nu_k, \mu_k, \varepsilon_k$ play the role of elastic constants. 
This choice is motivated by the fact that the tentacles are soft elastic structures, allowing for small violations of their constraints, 
at least on their radial components. We can incorporate in the elastic constants the morphological property that thin joints bend easier than thick joints. In 
this respect, we can assume that the sequences $\nu_k, \mu_k, \varepsilon_k$ are decreasing in $k$, whereas the sequence of angles $\alpha_k$ is increasing in $k$.

On the other hand, we impose the inextensibility constraint exactly, by defining for $k=1,...,N$
$$F_k(q):=\sigma_k\left(|q_{k}-q_{k-1}|^2-\ell^2\right)\,,$$
where $\sigma_k$ is a Langrange multiplier representing the tension exerted along the tentacle. 
We must say that, actually, longitudinal extension of the tentacles is also observed in many species (e.g. in the \emph{octopus vulgaris}),  
so one may wonder about the choice of an exact constraint. 
This is essentially motivated by three aspects. First, we recall that the muscolar system of an octopus is formed by volume preserving modules called \emph{hydrostats}: 
when a hydrostat is contracted, it shortens. The variability of the length of the tentacle is mainly due to this compensation mechanism, 
rather  than mere elastic forces. We guess that a more accurate model should include a curvature depending local inextensibility constraint. 
Therefore, at this stage of our investigation, we opted for a constant exact constraint which may be a first step towards this more accurate mathematical model. 
The second reason is given by the fact that octopus models are massively applied to soft robotics, and a large class of these devices are indeed inextensible. 
Finally, as shown below, this choice allows to slightly simplify the expression of the continuous Lagrangian of the system and, consequently, the corresponding equations of motion. 

Now we are ready to introduce the Lagrangian associated to the considered particle system:
\begin{equation}\label{discretelagrangian}
\mathcal L_N(q,\dot q,\sigma):=\sum_{k=1}^N\left( 
\frac12 m_k |\dot q_k|^2
-\frac{1}{2\ell} F_k(q)
-\frac{1}{\ell^3} G_k(q)
-\frac{1}{2\ell^5} B_k(q)
-\frac{1}{2\ell}H_k(q)
\right),
\end{equation}
where the first term represents the kinetic energy, while the other terms account for the constraints, suitably rescaled in view of the limit $N\to\infty$ we are going to perform. 
To this end, we assume that the total length of the chain is equal to $1$, and we set $\ell=1/N$. Moreover, we take the masses and the maximum bending angles 
of the form $m_k=\ell\rho_k$ and $\alpha_k=\ell\omega_k$ respectively, for given mass distributions $\rho_k$ and curvatures $\omega_k$. 
Finally, we consider smooth functions $\nu, \mu, \varepsilon, \rho, \omega:[0,1]\to\RR^+$ and, for $T>0$, smooth functions 
$q:[0,1]\times[0,T]\to\RR^2$, $\sigma:[0,1]\times[0,T]\to\RR$ and $u:[0,1]\times[0,T]\to[-1,1]$ such that, for all $N$, $k=1,...,N$ and $t\in[0,T]$
$$
\nu_k=\nu(k\ell),\quad \mu_k=\mu(k\ell),\quad \varepsilon_k=\varepsilon(k\ell),\quad \rho_k=\rho(k\ell),\quad \omega_k=\omega(k\ell),
$$
$$
q_k(t)=q(k\ell,t),\quad \sigma_k(t)=\sigma(k\ell,t),\quad u_k(t)=u(k\ell,t)\,.
$$
We remark that the space variable, denoted by $s\in[0,1]$, represents an arclength for the curve parametrized at time $t$ by $q(s,t)$. 
We employ standard notations $q_s$, $q_{ss}$ for partial derivatives in space, and $q_t$, $q_{tt}$ for partial derivatives in time. 
Moreover, we set $q^k_s:=q_s(\ell k)$, $q^k_{ss}:=q_{ss}(\ell k)$ and consider following Taylor expansions:
$$q_k-q_{k-1}=\ell q^k_s+o(\ell)=\ell q^k_s-\frac12\ell^2q^k_{ss}+o(\ell^2)\,,$$
$$q_{k+1}-q_{k}=\ell q^k_s+\frac12\ell^2q^k_{ss}+o(\ell^2)\,,$$
$$q_{k+1}-2q_k+q_{k-1}=\ell^2 q^k_{ss}+o(\ell^2)\,.$$
Then, as $N\to \infty$ (hence $\ell\to 0$), we have
\begin{align*}\frac{1}{\ell} F_k(q)&=\frac{1}{\ell}\sigma_k\left(|q_{k}-q_{k-1}|^2-\ell^2\right)=\frac{1}{\ell} 
\sigma_k\left( \ell^2|q^k_s|^2 +o(\ell^2)-\ell^2\right)\\
&=\ell\sigma_k(|q^k_s|^2 - 1) +o(\ell).\end{align*}
Since the inextensibility constraint is exact, we can use the relation $|q_{k}-q_{k-1}|=\ell$, for all $k=1,...,N$, to simplify the curvature potential given by $G_k$. Indeed, we get
\begin{align*}
\frac{1}{\ell^3} G_k(q)&=\frac{1}{\ell^3} \nu_k\left(\cos(\ell\omega_k)-\frac{1}{\ell^2}(q_{k+1}-q_{k})\cdot(q_{k}-q_{k-1})\right)_+^2\\
&=\frac{1}{\ell^3} \nu_k\left(1-\frac12\ell^2\omega_k^2+o(\ell^2)+\frac{1}{2\ell^2}\left( |q_{k+1}-2q_{k}+q_{k-1}|^2\right) -1 \right)_+^2\\
&=\frac{1}{\ell^3} \nu_k\left(-\frac12\ell^2\omega_k^2+o(\ell^2)+\frac{1}{2\ell^2}\left( \ell^4 |q^k_{ss}|^2+o(\ell^4)\right) \right)_+^2\\
&=\frac14\ell\nu_k\left(|q^k_{ss}|^2-\omega_k^2\right)_+^2+o(\ell)\,.
\end{align*}
Moreover, recalling that $q^k_s\times q^k_s=0$, $q^k_{ss}\times q^k_{ss}=0$ and $q^k_{s}\times q^k_{ss}=-q^k_{ss}\times q^k_{s}$ by definition of the vector product, we obtain 
\begin{align*}
\frac{1}{\ell^5} B_k(q)&=
\frac{1}{\ell^5}\varepsilon_k\Big( (q_{k+1}-q_{k})\times(q_{k}-q_{k-1})\Big)^2\\
&=\frac{1}{\ell^5}\varepsilon_k\Big((\ell q^k_s+\frac{\ell^2}{2}q^k_{ss}+o(\ell^2))\times(\ell q^k_s-\frac{\ell^2}{2}q^k_{ss}+o(\ell^2) )\Big)^2\\
&=\frac{1}{\ell^5}\varepsilon_k\Big( \frac12\ell^3 q^k_{ss}\times q^k_{s} -\frac12\ell^3 q^k_{s}\times q^k_{ss} +o(\ell^3) \Big)^2\\
&=\ell\varepsilon_k\left(q^k_s\times q^k_{ss}\right)^2+o(\ell)\,.
\end{align*}
Similarly,
\begin{align*}
\frac{1}{\ell}H_k(q)&=
\frac{1}{\ell}\mu_k\left(\sin(\ell\omega_ku_k)-\frac{1}{\ell^2}(q_{k+1}-q_{k})\times(q_{k}-q_{k-1})\right)^2\\
&=\frac{1}{\ell}\mu_k\Big(\ell\omega_ku_k-\ell q^k_{s}\times q^k_{ss} +o(\ell) \Big)^2\\
&=\ell\mu_k\left(\omega_ku_k-q^k_s\times q^k_{ss}\right)^2+o(\ell)\,.
\end{align*}
Putting together all the terms, summing on $k=1,...,N$ and taking the formal limit as $N\to\infty$, we obtain that $\mathcal{L}_N$ in \eqref{discretelagrangian} converges to 
\begin{align*}
\mathcal{L}_\infty(q,q_t,\sigma):=\int_0^1\left(\frac12\rho|q_t|^2\right.-\frac12 \sigma(|q_s|^2 - 1)-\frac14\nu\left(|q_{ss}|^2-\omega^2\right)_+^2
\left.-\frac12\varepsilon\left(q_s\times q_{ss}\right)^2
-\frac12\mu\left(\omega u-q_s\times q_{ss}\right)^2
\right)ds\,.\end{align*}
Note that $|q_s|^2=1$ is the continuous counterpart of the inextensibility constraint. Again, since this constraint is imposed exactly via the Lagrange multiplier $\sigma$, 
we can employ it to simplify the bending momentum term in $\mathcal{L}_\infty$. Indeed, differentiating $|q_s|^2=1$ with respect to $s$, we get $q_s\cdot q_{ss}=0$, 
namely $q_s$ and $q_{ss}$ are orthogonal. This implies that $\left(q_s\times q_{ss}\right)^2=|q_s|^2|q_{ss}|^2=|q_{ss}|^2$, and we can finally obtain the following equivalent 
Lagrangian for the continuous tentacle model:
\begin{equation}\label{continuouslagrangian}
\begin{split}\mathcal{L}(q,q_t,\sigma):=\int_0^1\left(\frac12\rho|q_t|^2\right.-\frac12 \sigma(|q_s|^2 - 1)-\frac14\nu\left(|q_{ss}|^2-\omega^2\right)_+^2
\left.-\frac12\varepsilon|q_{ss}|^2
-\frac12\mu\left(\omega u-q_s\times q_{ss}\right)^2
\right)ds\,.\end{split}
\end{equation}
\subsection{Equations of motion}\label{smotion} Here we derive the equations of motion for the continuous octopus tentacle model, 
by applying the classical least action principle to the Lagrangian \eqref{continuouslagrangian}.
To this end, we first assign initial data, i.e. we impose to the tentacle initial position and velocity, by means of smooth functions $q^0, v^0:[0,1]\to\RR^2$:
\begin{equation}\label{initialdata}
q(s,0)=q^0(s)\,,\qquad q_t(s,0)=v^0(s)\,,\qquad\mbox{for }s\in[0,1]\,.
\end{equation}
Moreover, we impose two boundary conditions at one endpoint of the tentacle. Following the discrete model, we defined the first two particles of the approximating chain as 
$q_{-1}=q_0+\ell e_2$ and $q_0=0$. In the first relation we recognize a backward finite difference discretization of the derivative $q_s(0)$, 
while the second relation simply prescribes the anchor point of the tentacle. Hence, we set 
\begin{equation}\label{leftboundary}
q(0,t)=0\,,\qquad q_s(0,t)=-e_2\,,\qquad\mbox{for }t\in[0,T]\,.
\end{equation}
We now define the action 
\begin{equation}\label{action}
\mathcal{S}(q,\sigma)=\int_0^T\mathcal{L}(q,q_t,\sigma)dt\,.
\end{equation}
Note that the unknowns of the problem are both the curve $q$ 
and the tension $\sigma$. According to the least action principle, for $T>0$ fixed, the dynamics of the tentacle in the interval $[0,T]$ and between the states 
$(q(s,0),\sigma(s,0))$ and $(q(s,T),\sigma(s,T))$ is the one for which $\mathcal{S}(q,\sigma)$ is stationary, namely we require
$$
\langle \delta_q\mathcal{S}(q,\sigma),w\rangle:=\lim_{\alpha\to 0}\frac{1}{\alpha}\left(\mathcal{S}(q+\alpha w,\sigma)-\mathcal{S}(q,\sigma)\right)=0\,,
$$
$$
\langle \delta_\sigma\mathcal{S}(q,\sigma),\chi\rangle:=\lim_{\alpha\to 0}\frac{1}{\alpha}\left(\mathcal{S}(q,\sigma+\alpha \chi)-\mathcal{S}(q,\sigma)\right)=0\,,
$$
for all admissible test functions $w:[0,1]\times[0,T]\to\RR^2$ and $\chi:[0,1]\times[0,T]\to\RR$. We remark that the notion of admissible test is related to the boundary conditions in space and time. More precisely, 
the pair $(w,\chi)$ is chosen so that the corresponding variation of $(q,\sigma)$ keep the same boundary conditions. In particular, for $s\in(0,1)$, we impose
$w(s,0)=w(s,T)=0$, $\chi(s,0)=\chi(s,T)=0$, and, in view of \eqref{initialdata}, also $w_t(s,0)=0$. Moreover, in view of \eqref{leftboundary}, we impose 
$w(0,t)=w_s(0,t)=0$ for $t\in(0,T)$. Note that no condition on $(w,\chi)$ is required for $s=1$, since it corresponds to the free end of the tentacle. 
Nevertheless, as we show below, additional boundary conditions will appear at $s=1$ for both $q$ and $\sigma$, as a consequence of the stationarity of the action.

To proceed, we need additional assumptions on some elastic constants appearing in the Lagrangian \eqref{continuouslagrangian}. For some $\varepsilon_0>0$, we require
\begin{equation}\label{epsmu}
\begin{array}{ll}
\varepsilon(s)\geq \varepsilon_0\,,&\mbox{for }s\in[0,1]\,,\\\\
\mu(1,t)=\mu_s(1,t)=0\,,\qquad&\mbox{for }t\in[0,T]\,.
\end{array}
\end{equation}
The first assumption prevents a degeneracy of the equations of motion, in particular at the free end, since $\varepsilon$ has a regularizing effect as discussed below. 
The second assumption is more technical, and allows one to considerably reduce the expression for the boundary conditions. 
In order to simplify the presentation, we also define
\begin{equation}\label{GH}
G[q,\nu,\varepsilon,\omega](s,t):=\varepsilon(s)+\nu(s)\left(|q_{ss}(s,t)|^2-\omega^2(s)\right)_+,
\end{equation}
$$H[q,\mu,u,\omega](s,t):=\mu(s)\left(\omega(s) u(s,t)-q_s(s,t) \times q_{ss}(s,t)\right)\,.$$

Let us start by computing the variation of $\mathcal S$ with respect to $\sigma$ in direction $\chi$, and impose stationarity: we immediately get
$$
\langle\delta_\sigma \mathcal S, \chi \rangle=\int_0^T\int_0^1 (|q_s|^2-1)\chi\,ds\,dt=0\,,
$$
for every admissible test $\chi$. This implies the inextensibility constraint, i.e. 
$$|q_s|^2=1\qquad \mbox{almost everywhere in }(0,1)\times(0,T)\,.$$
We now compute the variation of $\mathcal S$ with respect to $q$ in direction $w$: we have
$$
\langle\delta_q \mathcal S,w\rangle=\int_0^T\int_0^1\Big\{\rho q_t \cdot w_t 
-\sigma q_s \cdot w_s  -G q_{ss}\cdot w_{ss} -H(-w_s\times q_{ss}-q_s\times w_{ss})\Big\}ds\,dt\,.
$$
Integrating by parts each term, we get 
\begin{align*}
\int_0^T\int_0^1\rho q_t \cdot w_t\,ds\,dt=&\int_0^1 \left[\rho q_t \cdot 
w\right]_0^Tds-\int_0^T\int_0^1 \rho q_{tt} \cdot w\,ds\,dt\,,\\
-\int_0^T\int_0^1 \sigma q_s \cdot w_s\,ds\,dt=&-\int_0^T \left[\sigma q_s\cdot 
w\right]_0^1dt+\int_0^T\int_0^1 \left(\sigma q_s\right)_s \cdot w\,ds\,dt\,,\\
-\int_0^T\int_0^1 G q_{ss}\cdot w_{ss}\,ds\,dt=&-\int_0^T \left[ G q_{ss}\cdot w_s\right]_0^1~dt
+\int_0^T \left[ \left(G q_{ss}\right)_s\cdot w\right]_0^1~dt-\int_0^T\int_0^1 \left( G q_{ss}\right)_{ss}\cdot w\,ds\,dt
\end{align*}
and, recalling that $q\times w=q\cdot w^\bot=-q^\bot\cdot w$,
\begin{align*}
-\int_0^T\int_0^1&H (-w_s\times q_{ss}-q_s\times w_{ss})ds\,dt=\int_0^T\int_0^1 (H 
q^\bot_{ss}\cdot w_s - H q^\bot_s\cdot w_{ss})ds\,dt\\
=&\int_0^T \left[H q^\bot_{ss}\cdot w\right]_0^1dt - \int_0^T \left[H 
q^\bot_s\cdot w_s\right]_0^1dt
+ \int_0^T \left[\left(Hq^\bot_s\right)_s\cdot 
w\right]_0^1dt \\
&- \int_0^T\int_0^1 \left(H q^\bot_{ss}\right)_s\cdot w\,ds\,dt- \int_0^T\int_0^1 
\left(Hq^\bot_s\right)_{ss}\cdot w\,ds\,dt\,.
\end{align*}
Summing up, we obtain
\begin{align*}
\langle\delta_q \mathcal S,w\rangle=&\int_0^T\int_0^1 \left\{-\rho q_{tt}+\left(\sigma q_s\right)_s - \left(G q_{ss}\right)_{ss} 
 - \left(H q^\bot_{ss}\right)_s - \left(Hq^\bot_s\right)_{ss}\right\}
                          \cdot w\,ds\,dt\\
&+\int_0^1 \Big[\rho q_t \cdot w\Big]_0^Tds\\
&+\int_0^T \Big[\left\{ -\sigma q_s+\left( G q_{ss}\right)_s
+H q^\bot_{ss}+\left(Hq^\bot_s\right)_s\right\}\cdot w-\left\{ G q_{ss}+H q^\bot_s\right\}\cdot 
w_s\Big]_0^1~dt\,.
\end{align*}
Using the admissibility conditions $w(s,0)=w(s,T)=w(0,t)=w_s(0,t)=0$ and imposing stationarity, we finally get 
\begin{align*}
 \langle\delta_q \mathcal S,w\rangle=&\int_0^T\int_0^1 \left\{-\rho q_{tt}                   +\left(\sigma q_s\right)_s 
                          - \left( G q_{ss}\right)_{ss} 
                          - \left(H q^\bot_{ss}\right)_s 
                          - \left(Hq^\bot_s\right)_{ss}
                          \right\}\cdot w\,ds\,dt
\\&+\int_0^T \Big( \left\{ -\sigma q_s+\left( G q_{ss}\right)_s+H 
q^\bot_{ss}+\left(Hq^\bot_s\right)_s\right\}\cdot w (1,t) - \left\{ G q_{ss}+H q^\bot_s\right\}\cdot w_s (1,t)\Big)dt=0\,.
\end{align*}
Now, choosing $w$ with compact support in $(0,1)\times(0,T)$, we have in particular $w(1,t)=w_s(1,t)=0$, so that 
the first term in the previous expression should vanish. By the arbitrariness of $w$ we obtain, almost everywhere, the equations of motion
$$\rho q_{tt}=\left(\sigma q_s-H q_{ss}^\bot\right)_s -\left(G q_{ss}+H q_{s}^\bot\right)_{ss}\,.$$
On the other hand, choosing $w$ such that $w(1,t)=0$ and $w_s(1,t)$ is arbitrary, we get the following boundary condition at $s=1$:
$$G q_{ss}+H q^\bot_s=0\,.$$
Note that $H$ vanishes due to the assumption $\mu(1,t)=0$. Moreover, $G$ is strictly positive due to the assumption $\varepsilon(s)\ge\varepsilon_0>0$. Then, we obtain
$$q_{ss}(1,t)=0\qquad \mbox{for }t\in(0,T)\,.$$
Similarly, choosing $w$ such that $w_s(1,t)=0$ and $w(1,t)$ is arbitrary, we also get 
$$-\sigma q_s+G_s q_{ss}+G q_{sss}+H q^\bot_{ss}+H_sq^\bot_s+Hq^\bot_{ss}=0\,,$$
where we expanded, for convenience, the derivatives of the two products. Using again the assumption $\mu(1,t)=0$, and also $\mu_s(1,t)=0$, we observe that both 
$H$ and $H_s$ vanish at $s=1$. Moreover, we already know that $q_{ss}(1,t)=0$, hence by definition of $G$ we get
\begin{equation}\label{zeroshearstress}
-\sigma q_s+\varepsilon q_{sss}=0
\end{equation}
and, dot multiplying by $q_s$, also
\begin{equation}\label{zeroshearstressscalar}
-\sigma |q_s|^2+\varepsilon q_s\cdot q_{sss}=0\,.
\end{equation}
We now employ the inextensibility constraint $|q_s|^2=1$ obtained above. Differentiating twice with respect to $s$, 
we obtain $q_s\cdot q_{ss}=0$ and $q_s\cdot q_{sss}=-|q_{ss}|^2$ respectively. Prolonging these relations by continuity up to $s=1$, 
and plugging them into \eqref{zeroshearstressscalar}, we get
$$
-\sigma-\varepsilon |q_{ss}|^2=0\,,
$$
which implies $\sigma(1,t)=0$ for $t\in(0,T)$, using again $q_{ss}(1,t)=0$. 
Coming back to \eqref{zeroshearstress}, by the assumption $\varepsilon(s)>0$, we conclude that
$$q_{sss}(1,t)=0\qquad \mbox{for }t\in(0,T)\,.$$
Putting together all the initial and boundary conditions with the necessary conditions derived above, we end up with the following strong formulation of the 
system of motion for the tentacle model:
\begin{equation}\label{tentaclemotion}
\left\{\begin{array}{ll}
                  \rho q_{tt}=\left(\sigma q_s-H q_{ss}^\bot\right)_s -\left(G q_{ss}+H q_{s}^\bot\right)_{ss}&\mbox{in }(0,1)\times(0,T)\\
                  |q_s|^2=1 &\mbox{in }(0,1)\times(0,T)\\
                  q(s,0)=q^0(s) & s\in(0,1)\\
                  q_t(s,0)=v^0(s) & s\in(0,1)\\
                  q(0,t)=0 & t\in(0,T)\\
                  q_s(0,t)=-e_2 & t\in(0,T)\\
                  q_{ss}(1,t)=0 & t\in(0,T)\\
                  q_{sss}(1,t)=0 &t\in(0,T)\\
                  \sigma(1,t)=0 &t\in(0,T)
               \end{array}
\right.
\end{equation}
We remark that the relations at the free endpoint are known respectively as the zero bending momentum, zero shear stress and zero tension boundary conditions. 
They provide, with the anchor point and the fixed tangent at $s=0$, the so called {\em cantilevered} boundary conditions for the classical 
nonlinear Euler-Bernoulli beam discussed in the Introduction. We also observe that the equation is of order four in space. The positive term $\varepsilon$ in $G$ plays 
the role of a regularization and prevents the equation to degenerate to the third order or, worse, to the second order if also $H=0$ somewhere or, worse than worse, 
to the first order at $s=1$ where $\sigma=0$.

To conclude this section, let us spend few words about dissipation for the tentacle model. We can introduce two kinds of dissipation mechanisms: the first one is 
an environmental viscous friction proportional to the velocity, modeling the fact that the tentacle can be immersed in a fluid; 
the second one is an internal viscous friction, associated to the muscles 
contractions and hence proportional to the change in time of the tentacle curvature. It is very well known that such dissipative forces do not fit the standard variational framework of Lagrangian mechanics, they 
can enter in the equations of motion only formally, introducing a suitable Rayleigh dissipation functional of the form  $\int_0^1\mathcal{R}(q_t(s))ds$, 
and defining, ad hoc, the variation of the action \eqref{action} (note that the variation of $\mathcal{R}$ is performed with respect to $q_t$ and not $q$!):
$$
\langle \delta_q\mathcal{S},w\rangle:=\int_0^T\left(\langle \delta_q \mathcal{L},w\rangle-\langle \delta_{q_t} \mathcal{R},w\rangle\right)dt\,.
$$
Here we choose 
$$
\mathcal{R}(q_t)=\int_0^1\left(\beta(s)|q_t|^2+\gamma(s)|q_{sst}|^2\right)ds\,,
$$
where $\beta$ and $\gamma$ are positive functions playing the role of viscous friction coefficients, typically proportional to the local surface area. 
Since our tentacle model is thick at the base and thin at the tip, we can assume, as for the bending parameters $\nu, \varepsilon, \mu$, that also $\beta$ and $\gamma$ are decreasing in $s$. 
Imposing stationarity of the action above, using integration by parts and the boundary conditions in \eqref{tentaclemotion}, 
we easily obtain the corresponding frictional forces in the right hand side of the equations of motion: 
\begin{equation}\label{tentaclemotionfriction}
\rho q_{tt}=\left(\sigma q_s-H q_{ss}^\bot\right)_s -\left(G q_{ss}+H q_{s}^\bot\right)_{ss}-\beta q_t-\gamma q_{sssst}\,.
\end{equation}
This dissipative version of the model will be employed in the following numerical simulations, in order to observe the equilibrium configurations of the system.  
\subsection{Equilibria}\label{sstationary}
Here we characterize the equilibria of \eqref{tentaclemotion}, namely we provide, up to an ordinary integration in space, an explicit solution $(q,\sigma)$ to the stationary problem
\begin{equation}\label{tentacleequilibrium}
\left\{\begin{array}{ll}
                  \left(\sigma q_s-H q_{ss}^\bot\right)_s -\left(G q_{ss}+H q_{s}^\bot\right)_{ss}=0&\mbox{in }(0,1)\\
                  |q_s|^2=1 &\mbox{in }(0,1)\\
                  q(0)=0 \\
                  q_s(0)=-e_2 \\
                  q_{ss}(1)=0 \\
                  q_{sss}(1)=0 \\
                  \sigma(1)=0 
               \end{array}
\right.
\end{equation}
in terms of a given stationary control map $u:[0,1]\to[-1,1]$. Indeed, we prove the following proposition.
\begin{proposition}\label{pequilibria}
Let $u\in C^2[0,1]$ and assume $\mu(1)=\mu_s(1)=0$, $\varepsilon(s)>0$. Then the stationary problem \eqref{tentacleequilibrium} 
admits a unique solution $(q,\sigma)\in C^4([0,1])\times C^1([0,1])$, which is given, for $s\in[0,1]$, by 
\begin{equation}\label{control2state}
q(s)=\int_0^s\Big(\sin(\theta(\xi)),-\cos(\theta(\xi))\Big)d\xi\,,\qquad \sigma(s)=\varepsilon(s) \Big(\bar\omega(s) u(s)\Big)^2\,,
\end{equation}
where 
$$\theta(\xi):=\int_0^\xi \bar\omega(z) u(z)\,dz\qquad\mbox{and}\qquad \bar\omega(s) :=\frac{\mu(s) \omega(s)}{\mu(s)+\varepsilon(s)}\,.$$ 
\end{proposition}
\begin{proof}
First of all, we observe that a solution to \eqref{tentacleequilibrium} is, by construction, a minimum point of the potential 
\begin{equation}\label{potential}
\mathcal{V}(q)=\int_0^1\left(\frac14\nu\left(|q_{ss}|^2-\omega^2\right)_+^2
+\frac12\varepsilon|q_{ss}|^2
+\frac12\mu\left(\omega u-q_s\times q_{ss}\right)^2
\right)ds\,.
\end{equation}
Differentiating the inextensibility constraint $|q_s|^2=1$, we obtain the orthogonality condition $q_s\cdot q_{ss}=0$ and hence, for every $s\in(0,1)$, there exists 
$\alpha\in\RR$ such that $q_{ss}=\alpha q_s^\bot$, since $q_s$ and $q_s^\bot$ define a local base for $\RR^2$. Moreover, we have $|q_{ss}|=|\alpha|$ and 
$q_s\times q_{ss}=\alpha$. It follows that the minimization of $\mathcal{V}$ in \eqref{potential} can be performed pointwise, namely minimizing, for every $s\in(0,1)$, 
the following function of $\alpha$: 
$$
f(\alpha)=\frac14\nu\left(\alpha^2-\omega^2\right)_+^2+\frac12\varepsilon\alpha^2+\frac12\mu\left(\omega u-\alpha\right)^2\,. 
$$
We easily obtain that
$$
f^\prime(\alpha)=\alpha\nu\left(\alpha^2-\omega^2\right)_+  +\varepsilon\alpha-\mu\left(\omega u-\alpha\right)=0
$$
if and only if
$$
\alpha=\frac{\mu\omega u}{\mu+\varepsilon+\nu\left(\alpha^2-\omega^2\right)_+}\,.
$$
Recalling that $\varepsilon>0$ and $u\in[-1,1]$, we have $|\alpha|<\omega$ and then $\left(\alpha^2-\omega^2\right)_+=0$. We conclude that the unique minimum point of $f$ is
$$
\alpha=\frac{\mu\omega u}{\mu+\varepsilon}=\bar\omega u\,.
$$
Hence, the stationary problem \eqref{tentacleequilibrium} is equivalent to the following second order problem: find $q$ such that
\begin{equation}\label{reducedstationary}
\quad\left\{\begin{array}{ll}
                  q_{ss}=\bar \omega u q_s^\bot&\mbox{in }(0,1)\\
                  |q_s|^2=1 &\mbox{in }(0,1)\\
                  q(0)=0\\q_{s}(0)=-e_2
                 \end{array}
\right. 	\end{equation}
Note that, due to the assumption $\mu(1)=\mu_s(1)=0$, the conditions $q_{ss}(1)=q_{sss}(1)=0$, although irrelevant, are now embedded in the first equation of \eqref{reducedstationary}. 
We proceed by introducing polar coordinates such that 
$$q_s(s)=\Big(\sin(\theta(s)),-\cos(\theta(s))\Big)\,,$$
where $\theta:[0,1]\to\RR$ is unknown. It follows that $|q_s|^2=1$, while $q_s(0)=-e_2$ is equivalent to $\theta(0)=0$. Moreover, 
differentiating in $s$, we obtain 
$$q_{ss}(s)=\theta_s(s)\Big(\cos(\theta(s)),\sin(\theta(s))\Big)=\theta_s(s)q_s^\bot(s)\,,$$
then we end up with the following Cauchy problem in $\theta$:
$$
\left\{\begin{array}{ll}
                  \theta_s=\bar\omega u &\mbox{in }(0,1)\\
                  \theta(0)=0\\
               \end{array}
\right.
$$
The unique solution is given by
$$
\theta(s)=\int_0^s \bar\omega(z) u(z)\,dz\,,
$$
and, integrating $q_s$, we conclude
$$
q(s)=\int_0^s\Big(\sin(\theta(\xi)),-\cos(\theta(\xi))\Big)d\xi\,,
$$
where, for $s=0$, we also recover the condition on the anchor point $q(0)=0$. 

We now recall the differential equation in the stationary problem \eqref{tentacleequilibrium}, 
$$
 \left(\sigma q_s-H q_{ss}^\bot\right)_s -\left(G q_{ss}+H q_{s}^\bot\right)_{ss}=0\,.
$$
Integrating in $s$ and expanding the second term, we get
$$
 \sigma q_s-2H q_{ss}^\bot -H_s q_{s}^\bot -G q_{sss}-G_s q_{ss}=C\,,
$$
for some constant vector $C\in\RR^2$. By continuity, for $s\to 1$, using the boundary conditions $\sigma(1)=H(1)=H_s(1)=0$ and $q_{ss}(1)=q_{sss}(1)=0$, 
we easily get $C=0$. 
Moreover, dot multiplying by $q_s$ and using $|q_s|^2=1$, $q_{s}\cdot q_{ss}=0$ and $q_s\cdot q_{sss}=-|q_{ss}|^2$, we get
$$
 \sigma -2H q_s\times q_{ss}+G |q_{ss}|^2=0\,.
$$
Finally, using the definitions of $G, H, \bar\omega$ and substituting $q_s\times q_{ss}=\bar\omega u$, $|q_{ss}|^2=(\bar\omega u)^2$, we conclude
$$
\sigma =2\mu(\omega u- \bar\omega u) \bar\omega u -\varepsilon (\bar\omega u)^2=
2\mu\left(\frac{\mu+\varepsilon}{\mu}- 1\right) (\bar\omega u)^2 -\varepsilon (\bar\omega u)^2=\varepsilon (\bar\omega u)^2\,,
$$
and this completes the proof. 
\end{proof}

\subsection{Discretization}\label{Sec:Discretization}
In this section we give some details on the numerical approximation and solution of the system of motion \eqref{tentaclemotion}, 
in particular we numerically validate the model and the characterization of its equilibria provided by Proposition \ref{pequilibria}. 
We introduce a uniform discretization of the space-time $[0,1]\times[0,T]$, namely we consider integers $N,M$ and define grid nodes $s_k=k\Delta s$ for $k=0,...,N$ and 
$t_n=n\Delta t$ for $n=0,...,M$, where the discretization steps are given respectively by $\Delta s=1/N$ and $\Delta t=T/M$. For a generic scalar or vector function 
$\chi$ defined on $[0,1]\times[0,T]$ we adopt the standard notation $\chi_k^n\approx\chi(s_k,t_n)$ to identify the corresponding approximation. 
In particular, we suitably define the approximations of all the data functions $\rho, \varepsilon, \nu, \omega, \mu, u, q^0, v^0$ appearing in \eqref{tentaclemotion}.

Space discretization for the derivatives of the unknowns $q,\sigma$ is performed by employing standard finite difference operators. 
Denoting by $\Dm$,$\Dp$ respectively first order backward and forward differences, and by $D^2_c$ second order central differences, we introduce the following 
approximation of the accelerations in the 
equations of motion:
$$
a(q_k^n,\sigma_k^n):=\displaystyle\frac{1}{\rho_k}\left(\Dp\left(\sigma_k^n \Dm q_k^n-H_k^n D^2_cq_{k}^{n\bot}\right) -D^2_c\left(G_k^n D^2_cq_{k}^n+H_k^n \Dm q_{k}^{n\bot}\right)\right)
$$
with 
$$G_k^n=\varepsilon_k+\nu_k\left(|D^2_c q_k^n|^2-\omega_k^2\right)_+,\qquad H_k^n=\mu_k\left(\omega_k u_k^n-\Dm q_k^n \times D^2_c q_k^n\right)\,.$$
Note that the tangent $q_s$ is approximated by backward differences, whereas other first order derivatives are approximated by forward differences. 
This choice clearly gives an overall approximation of order one in space for the equations of motion and naturally comes from the discrete model. 
More precisely, the approximations $a(q_k^n,\sigma_k^n)$ coincide with the accelerations of the particle system discussed in Section \ref{sdiscrete}: 
they can be derived from the least action principle directly applied to the discrete Lagrangian $\mathcal{L}_N$ in \eqref{discretelagrangian}. 

Let us take a look at the boundary conditions. Introducing two ghost nodes $s_{-1}=-\Delta s$ and $s_{N+1}=(N+1)\Delta s$ at the ends of the interval $[0,1]$, we set  
$$
\left\{\begin{array}{ll}
                  q_{N+1}^n-2q_N^n+q_{N-1}^n=0 & \mbox{zero bending moment}\\
                   q_{N+1}^n-3q_N^n+3q_{N-1}^n-q_{N-2}^n=0  & \mbox{zero shear stress}\\
                    \sigma_N^n=0  & \mbox{zero tension}\\
                                q_0^n=0 & \mbox{anchor point}\\
                  q_0^n-q_{-1}^n=-e_2\Delta s & \mbox{fixed tangent}
               \end{array}
\right.
$$
We remark that, from the first two (linear) equations, we readily get the values of $q_N^n$ and $q_{N+1}^n$ in terms of the internal nodes, 
while from the last equation we obtain 
a fixed value for the node $q_{-1}^n$.

Finally, we proceed with the time integration, by employing a classical Velocity Verlet scheme. 
This choice is mainly motivated by the symplectic nature of the scheme, in particular 
its ability to preserve energy regardless of the duration of the simulation. 
For $k=1,...,N-1$ and $n=1,...,M-1$, starting from $q_k^0=q^0(s_k)$ and $v_k^0=v^0(s_k)$, we discretize the equations of motion and the inextensibility constraints as 
\begin{equation}\label{verlet}
\left\{\begin{array}{l}
q_k^{n+1}=q_k^n+v_k^n\Delta t+\displaystyle\frac12 a(q_k^n,\sigma_k^n) \Delta t^2 \\
|\Dm q_k^{n+1}|^2=1
\end{array}
\right.
\end{equation}
and we update the velocities by setting
$$v_k^{n+1}=v_k^n+\displaystyle\frac12\Big(a(q_k^n,\sigma_k^n)+a(q_k^{n+1},\sigma_k^n)\Big)\Delta t\,.$$
In the dissipative case, the frictional forces appearing in \eqref{tentaclemotionfriction} can be included in $a(q_k^n,\sigma_k^n)$ using the following approximations:
$$
-\beta(s_k) q_t(s_k,t_n)\approx -\beta_k v_k^n\,,\qquad -\gamma(s_k) q_{sssst}(s_k,t_n)\approx -\gamma_k D^2_c D^2_c v_k^n\,.
$$
Note that, for every time step $n=0,...,M-1$, the system \eqref{verlet} consists in $3(N-1)$ equations in the unknowns $(q_k^{n+1},\sigma_k^n)_{k=1,...,N-1}$. 
Due to the definition of $a(q_k^n,\sigma_k^n)$, the first $2(N-1)$ equations are linear, whereas the last $N-1$ are nonlinear, and a solution can be iteratively computed 
by means of a standard Newton method.\\

\noindent Now, let us fix the parameters for the simulation. We set $\Delta s=2\times 10^{-2}$ (corresponding to $N=50$ joints), $\Delta t=10^{-4}$ and
$$\rho(s)=\exp(-s),\qquad \varepsilon(s)=10^{-3}(1-0.9s),\qquad \nu(s)=10^{-3}(1-0.09s),\qquad \omega(s)=2\pi(1+s^2),$$
$$\mu(s)=(1-s) \exp(-0.1s^2/(1-s^2)),\qquad \beta(s)=4-s,\qquad \gamma(s)=10^{-6}(4-s)\,.$$
We observe that $\varepsilon$ satisfies the assumption $\varepsilon(1)=10^{-4}>0$, whereas 
$\mu$ is just a small perturbation of a linear function satisfying the assumption $\mu(1)=\mu_s(1)=0$. 
We choose the fully extended initial profile $q^0(s)=-e_2\,s$, the initial velocity $v^0(s)\equiv 0$ and the control map $u(s,t)\equiv 1$, corresponding to a 
full contraction of the tentacle, constant in time. Figure \ref{tentacledynamics} shows the evolution computed by the numerical scheme in the time interval $[0,12]$. 
In each frame we represent the tentacle using a tube of width proportional to $\rho$, to give a qualitative idea of the non uniform mass distribution. 
Moreover, we show the different configurations in gray scales, from light gray to black as the time increases. 
\begin{figure}[!h]
\centering
\begin{tabular}{cccc}
 \includegraphics[scale=0.45]{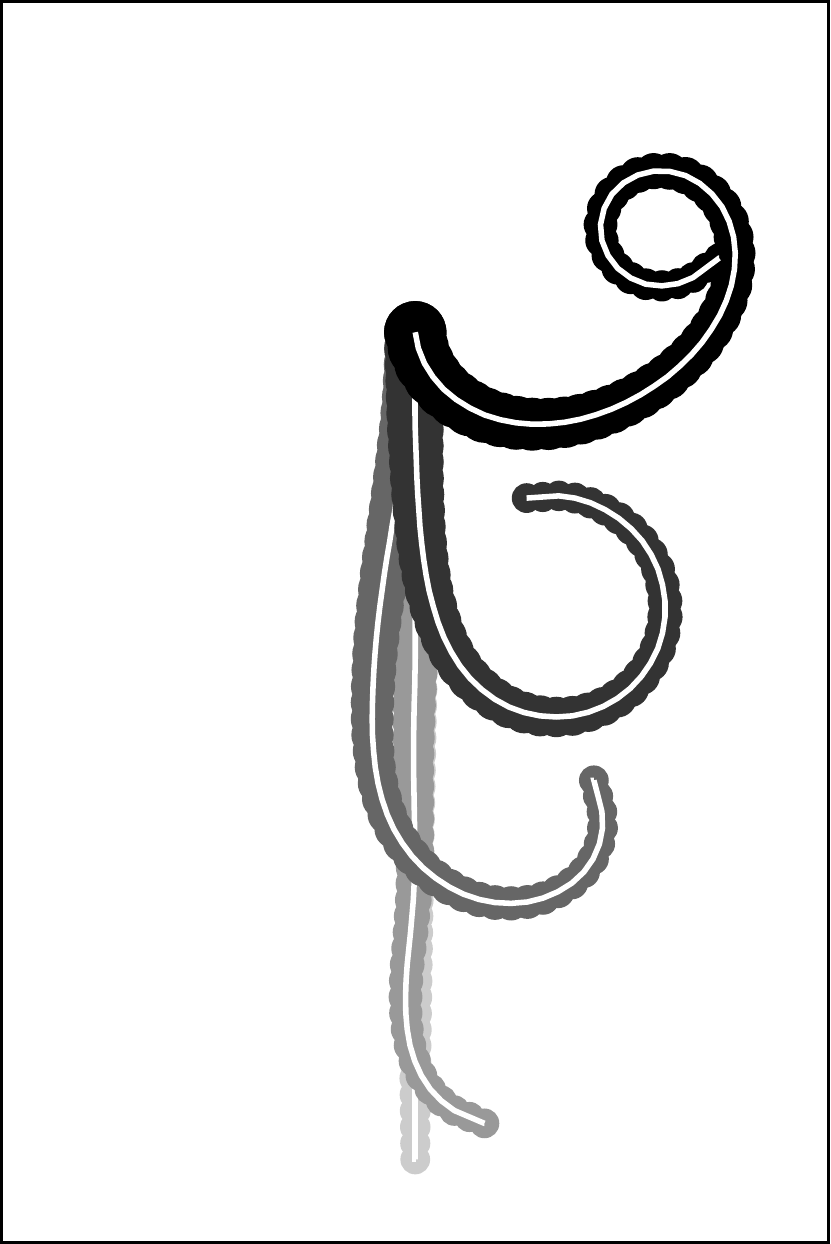}&
 \includegraphics[scale=0.45]{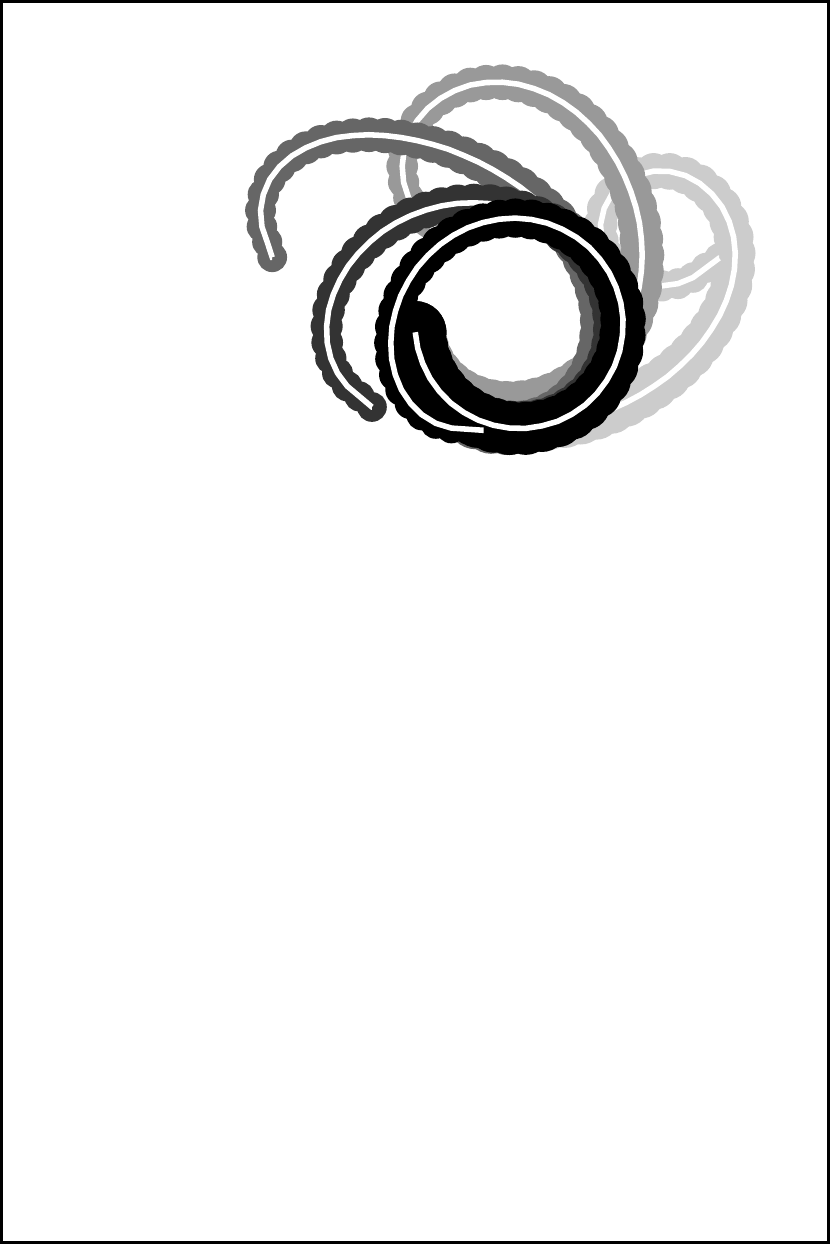}&
 \includegraphics[scale=0.45]{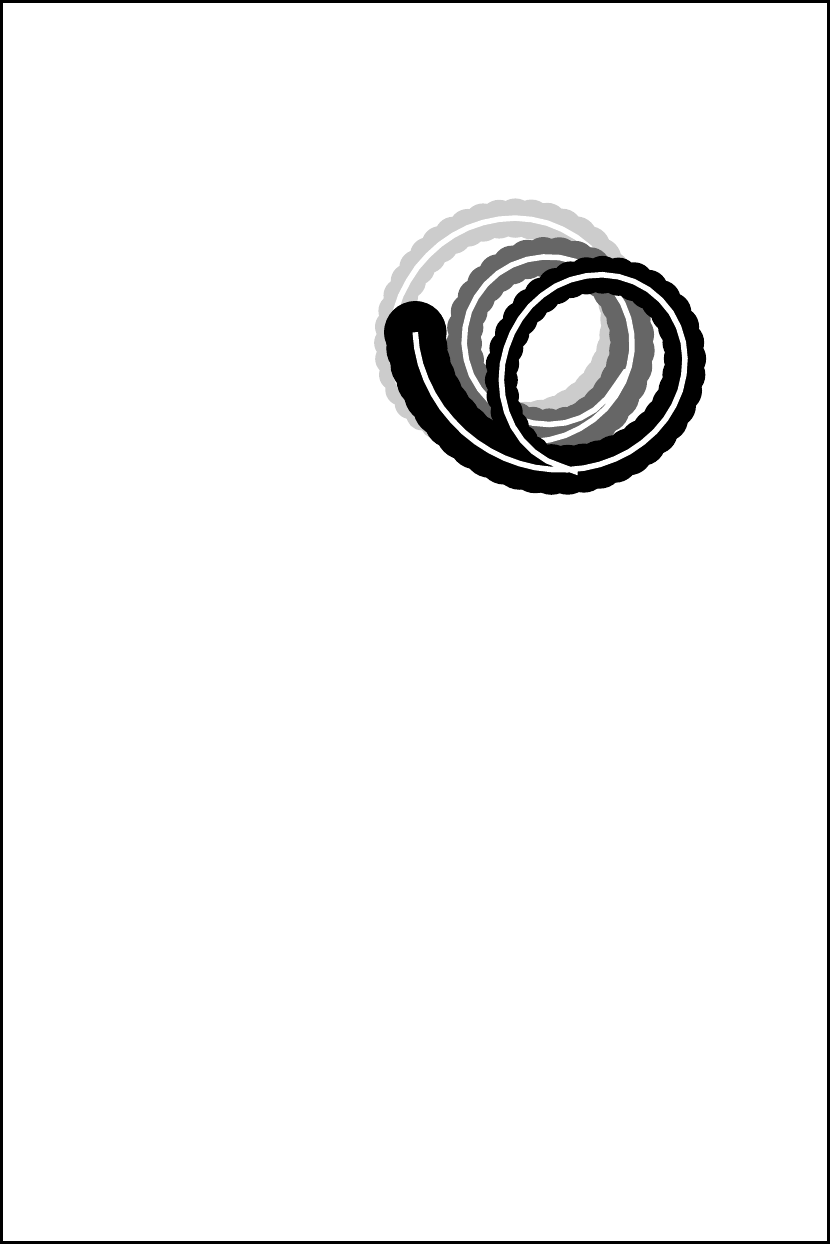}&
 \includegraphics[scale=0.45]{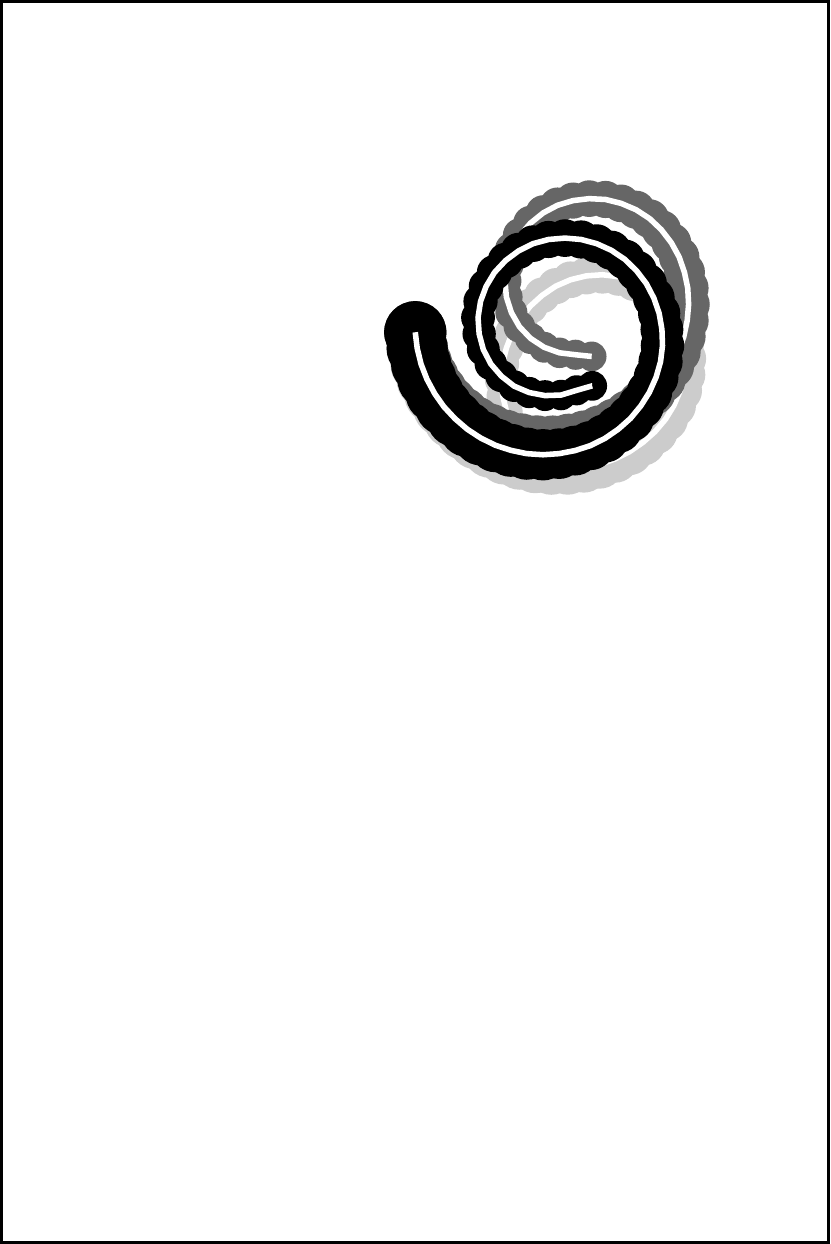}\\
 $0\le t \le 0.35$ &$0.35\le t \le 0.8$ &$0.8\le t \le 1.5$ &$1.5\le t \le 12$ 
\end{tabular}
\caption{tentacle dynamics for a full contraction control}\label{tentacledynamics}
\end{figure}

\noindent The effect of the curvature constraints is clearly visible, 
while the presence of frictional forces guarantees convergence to the equilibrium configuration given in Proposition \ref{pequilibria}. 
This is confirmed in Figure \ref{errors-curvature-tension}, showing the asymptotic behavior of 
$\mathcal{E}_{q}(n)=\||D^2_c q^n|-\bar\omega\|_\infty$ and $\mathcal{E}_{\sigma}(n)=\|\sigma^n -\varepsilon\bar\omega^2\|_\infty$, respectively 
the error between the curvature at time $t_n=n\Delta t$ and the equilibrium curvature, and the error between the tension at time $t_n$ and the equilibrium tension.
\begin{figure}[!h]
\centering
\begin{tabular}{cc}
 \includegraphics[width=0.3\textwidth]{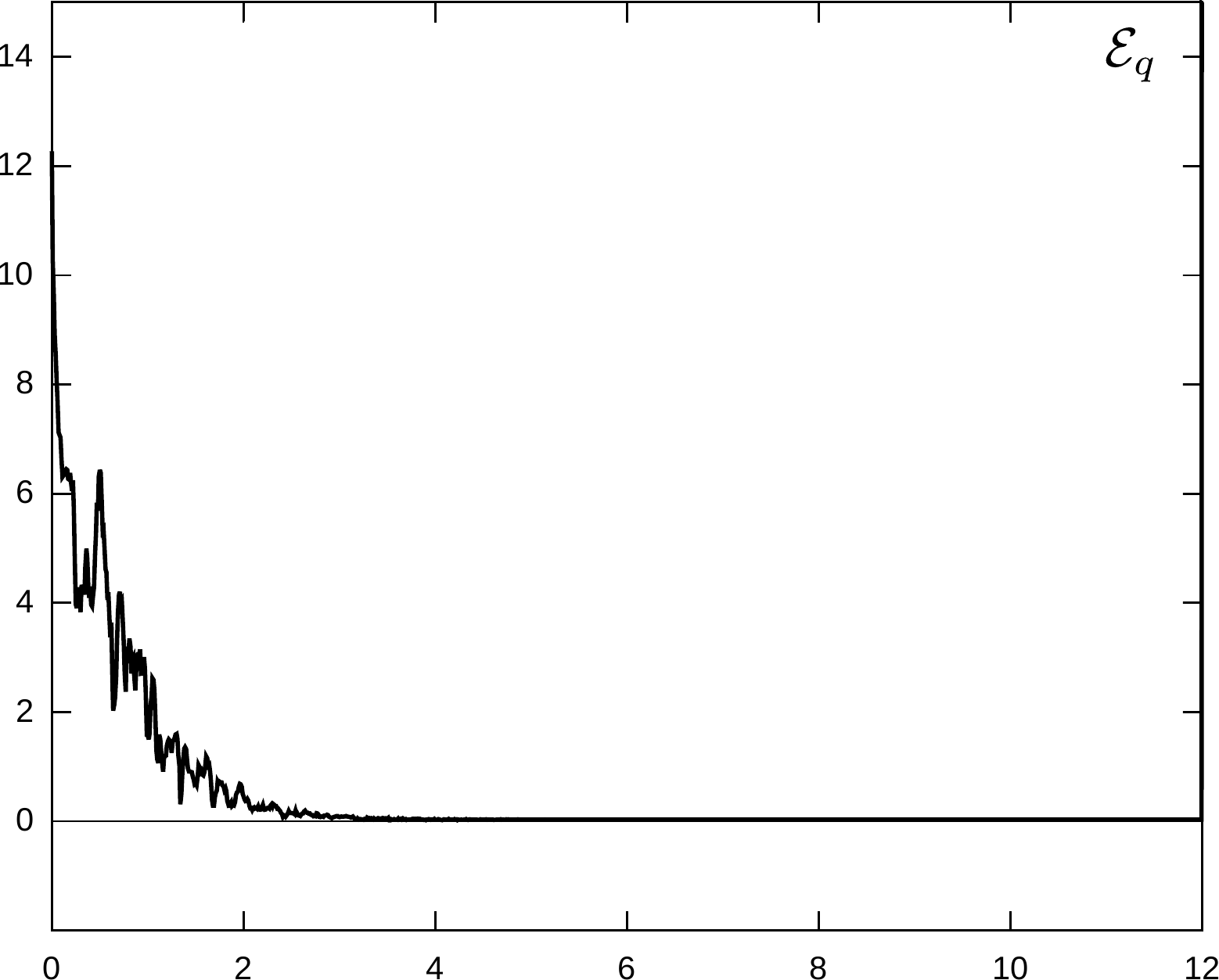}&
 \includegraphics[width=0.3\textwidth]{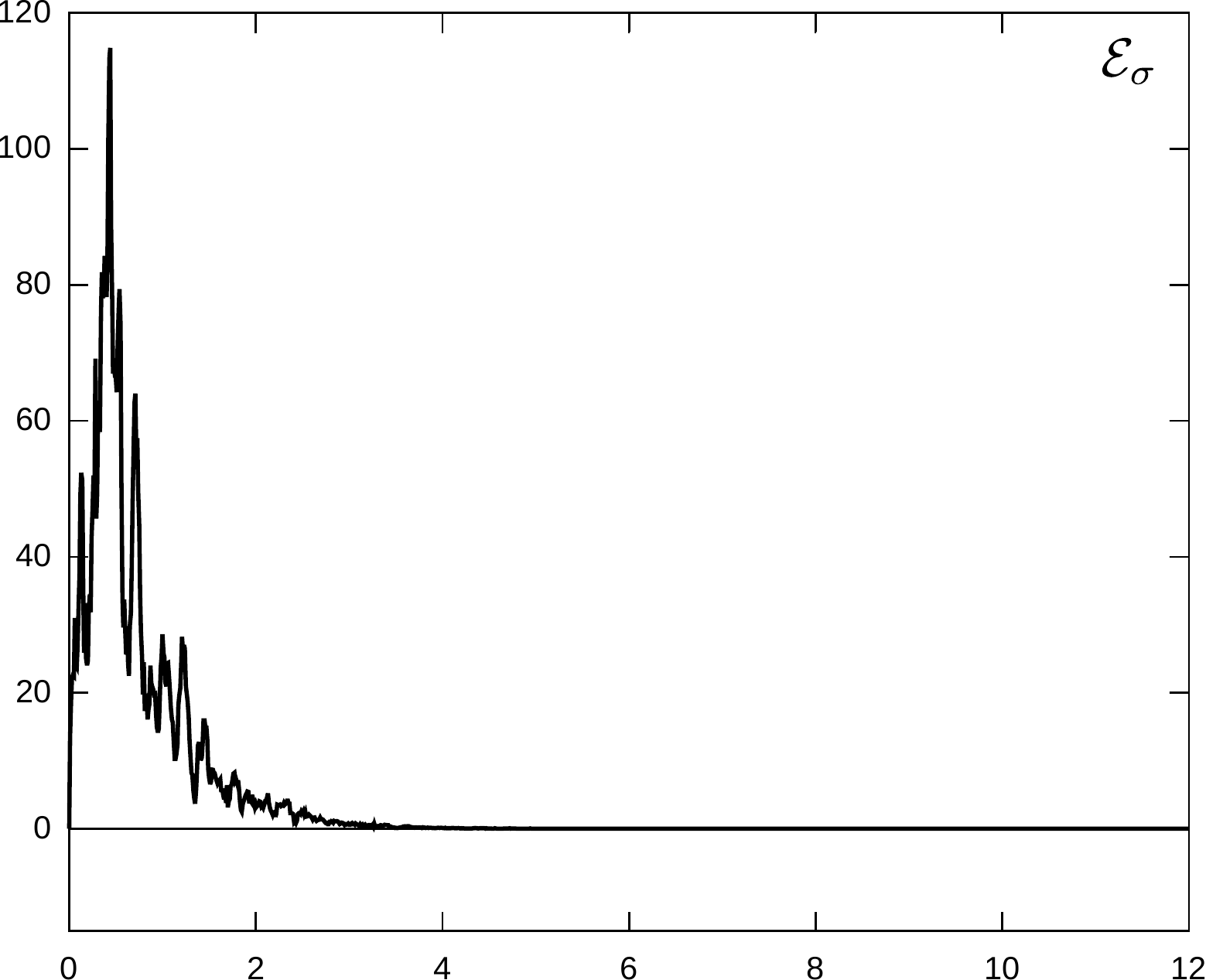}\\
 $(a)$& $(b)$
\end{tabular}
\caption{Asymptotic behavior of curvature error $\mathcal{E}_{q}$ (a) and tension error $\mathcal{E}_{\sigma}$ (b)}\label{errors-curvature-tension}
\end{figure}

\section{The stationary optimal control problem}\label{Sec:Static}
In this section we address the static optimal control problem discussed in the Introduction. 
The goal is to touch a point with the tentacle tip and minimum effort. More precisely, 
given $q^\ast\in\RR^2$ and $\tau>0$, we consider the following optimization problem:
\begin{equation}\label{touchfunctional}
\min\left\{\frac12\int_0^1 u^2 ds+\frac{1}{2\tau}|q(1)-q^\ast|^2\right\}\quad\mbox{subject to}\quad\left\{\begin{array}{ll}
                  q_{ss}=\bar \omega u q_s^\bot &\mbox{in }(0,1)\\
                  |q_s|^2=1 &\mbox{in }(0,1)\\
                  |u|\leq 1 &\mbox{in }(0,1)\\
                  q(0)=0\\q_{s}(0)=-e_2
                 \end{array}
\right. 
\end{equation}
where the first term in the functional quantifies the activation of the tentacle muscles in terms of the control, 
while the second term penalizes the distance of the tentacle tip from the target point $q^\ast$. We observe that, 
by virtue of Proposition \ref{pequilibria}, we have an explicit characterization of the constraints appearing in  \eqref{touchfunctional} in terms of the control. 
In the context of optimal control, formula \eqref{control2state} provides the so called {\em control to state} map, 
and it can be employed to completely remove $q$ from the optimization problem above, yielding the following optimization in the control only:
$$\min_{|u|\leq 1}\left\{\frac12\int_0^1 u^2 ds+\frac{1}{2\tau}\left|\int_0^1\left(\sin\left(\int_0^s \bar\omega(\xi) u(\xi)\,d\xi\right),
-\cos\left(\int_0^s \bar\omega(\xi) u(\xi)\,d\xi\right)\right)ds-q^\ast\right|^2\right\}\,.$$
Note that the inextensibility constraint is now hidden in the integral formulation but, from a numerical point of view, we found that this problem is quite involved, mainly due to the fact that, in polar coordinates, 
the tentacle tip $q(1)$ is obtained by integration on the whole curve. This results in a non local dependency of the functional on the control, 
so that the corresponding Euler-Lagrange equation is in turn an integral equation in $u$. Then, we decided to follow the opposite approach, namely remove 
the dependency on the control, using the relation $|q_{ss}|=\bar \omega |u|$. This gives the following equivalent formulation of the problem: 
\begin{equation}\label{staticfunctional}
\begin{split}\min\left\{\frac12\int_0^1 \frac{1}{\bar\omega^2}|q_{ss}|^2 ds+\frac{1}{2\tau}|q(1)-q^\ast|^2\right\}
\mbox{ subject to }\left\{\begin{array}{ll}
        |q_s|^2=1 &\mbox{in }(0,1)\\
        |q_{ss}|\leq \bar \omega&\mbox{in }(0,1)\\
        q(0)=0\\q_{s}(0)=-e_2
        \end{array}
\right.\end{split}
\end{equation}
The constraints are still there, and they should be enforced by suitable multipliers. Nevertheless, the problem is more tractable, as we will show below 
introducing an augmented Lagrangian method for its numerical solution. 
But first, we present an investigation on the reachability of target points $q^*$, i.e., 
we study the set of $q^*$ for which the equation $q(1)=q^*$ admits a solution, under the inextensibility and curvature constraints. In particular, we establish 
a connection with the celebrated \emph{Dubins car} problem, a classical example studied in motion planning \cite{markov}.
\subsection{Reachability}
Let us start by making an important assumption on the model, i.e., 
that the elastic constants $\mu, \varepsilon$ and the curvature bound $\omega$ are balanced in the following way:
$$
\bar\omega(s)=\frac{\mu(s)\omega(s)}{\mu(s)+\varepsilon(s)}\equiv \bar\omega_0\quad\mbox{for }s\in[0,1)\quad\mbox{and}\quad\lim_{s\to 1}\bar\omega(s)=\bar\omega_0\,,$$
for some $\bar\omega_0>0$. Incidentally, notice that this implies that $\omega$ is unbounded near $1$, since $\mu(1)=0$. 
With this assumption standing, we recover the same results of Proposition \ref{pequilibria}, by restricting our set of admissible controls to those satisfying $u(1)=u_s(1)=0$. More precisely  we define the \emph{set of admissible controls} as follows
 $$\mathcal A:=\{ u:[0,1]\to\RR\mid u\in C^2([0,1]),\,|u|\leq 1,\, u(1)=u_s(1)=0\}$$
 and we rewrite the problem \eqref{reducedstationary} in the following way:
\begin{equation}\label{polar}
\left\{
\begin{array}{ll}
q_s=(\sin(\theta),-\cos(\theta))&\mbox{in }(0,1)\\
\theta_s= \bar\omega_0 u&\mbox{in }(0,1)\\
q(0)=0\\
\theta(0)=0
\end{array}
\right.
\end{equation}
which is exactly the classical Dubins car control system \cite{markov}. Then, we define the \emph{reachable set} as
$$R:=\{q(1)\mid q=q[u] \text{ is a $C^1$ solution of \eqref{polar}}, u\in\mathcal A\}.$$
Moreover, we consider the wider control set
$$\mathcal A':=\{u:[0,1]\to\RR \mid u \text{ piecewise continuous}, |u|\leq 1\}$$
and let
$$R':=\{q[u]\mid q[u]  \text{ is a $C^1$ solution of } \eqref{polar}, u\in \mathcal A'\}.$$
 In \cite{cockayne} the set $R'$ is characterized via the following control set:
$$\mathcal A'':=\{u:[0,1]\to\{0,\pm 1\} \mid u \text{ has a finite number of discontinuities}\}.$$
In order to properly cite this result, we adopt the following notations:  $q\in CL$ (respectively $q\in CC$) if $q=q[u]$ is a $C^1$ solution of \eqref{polar} with  
\begin{equation}\label{uCL}
u(s)=\begin{cases}
\bar u &s\in[0,\bar s]\\
0 \text{ (respectively, }-\bar u)&s\in (\bar s,1]
\end{cases}
\end{equation} 
for some $\bar s\in [0,1]$ and $\bar u\in\{\pm 1\}$. In other words, $q[u]\in CL$ if it is composed by a circular arc with radius $1/\bar\omega_0$ 
and by a line segment, while $q[u]\in CC$ if it is composed by two tangent circular arcs with radius $1/\bar\omega_0$. 
We are now in position to recall Cockayne and Hall's result:
\begin{proposition}\cite{cockayne}\label{cockaine}
	For every $\bar\omega_0>0$, if $q^*\in R'$, then there exists  $u\in \mathcal A''$ such that the Carath\'eodory solution $q[u]$ of \eqref{polar} satisfies 
	$q[u](1)=q^*$. Moreover, $R'$ is a compact set and $q[u](1)$ belongs to its boundary only if $u$ is a piecewise constant function of the form \eqref{uCL}, i.e., $\partial R'\subseteq \{q(1)\mid q\in CC\cup CL\}$.
	\end{proposition}
\noindent Using this result, we can prove the following proposition.
\begin{proposition}\label{th-reach} For every $\bar\omega_0>0$
	$$R=\text{int}(R'),$$
	where $\text{int}(\cdot)$ denotes the interior of a set. In particular, 
	$$\partial R\subseteq\{q(1)\mid q\in CC\cup CL\}.$$
\end{proposition}
\begin{proof}
First of all, we remark that, by definition, $R\subset R'$. By Proposition \ref{cockaine},  
the boundary of $R'$ is reachable only by piecewise constant controls, then we have $R\subseteq \text{int}( R')$. To prove the other inclusion, let $q^*\in \text{int}(R')$: again by Proposition \ref{cockaine} there exists a control $u\in \mathcal A''$ be such that $q[u](1)=q^*$. 
Now, note that $\mathcal A$ is dense in $\mathcal A''$ with respect to the $L^2$ norm, we then may consider an approximating control sequence $u_n\in \mathcal A$ such that 
$$\int_0^1 |u_n-u|^2 \leq \frac{1}{n}\quad \forall n\in\mathbb N.$$
Setting $\theta_n(s)=\int_0^s\bar \omega_0 u_n(\xi)d\xi$ and $\theta(s)=\int_0^s \bar \omega_0 u(\xi)d\xi$ we also get the estimate
$$|\theta_n(s)-\theta(s)|^2=\bar \omega^2_0\left|\int_0^s (u_n(\xi)-u(\xi))d\xi\right|^2\leq\bar \omega^2_0\int_0^1|u_n-u|^2\leq \frac{\bar \omega^2_0}{n}\quad \forall s\in[0,1].$$
 The role of the function $\theta$ (and similarly of $\theta_n$) becomes clearer by remarking that from \eqref{polar}, we get
 $$q_{ss}[u]=\bar\omega_0 u(\cos(\theta),\sin(\theta)) \quad \text{a.e. in }[0,1]\,.$$
By the assumption $|u|,|u_n|\leq 1$, it follows that
 \begin{align*}
q_{ss}[u_n]\cdot q_{ss}[u]&=\bar\omega_0^2 u_nu\cos(\theta-\theta_n)\geq \bar\omega_0^2 u_nu(1-\frac{1}{2}(\theta-\theta_n)^2)
\geq \bar\omega_0^2 u_nu\left(1-\frac{\bar \omega_0^2}{2n}\right)\geq \bar\omega_0^2 u_nu-\frac{\bar\omega_0^4}{2n}\,.
 \end{align*}
Now, recalling that $q[u](0)=q[u_n](0)=0$ and $q_s[u](0)=q_s[u_n](0)=-e_2$, we have
\begin{align*}
|q[u_n](1)-q^*|^2=&\left|\int_0^1\int_0^s \left(q_{ss}[u_n](\xi)-q_{ss}[u](\xi)\right)d\xi ds\right|^2\\\leq& \int_0^1\int_0^1|q_{ss}[u_n]-q_{ss}[u]|^2=\int_0^1\int_0^1|q_{ss}[u_n]|^2+|q_{ss}[u]|^2-2q_{ss}[u_n]\cdot q_{ss}[u]\\ 
\leq &\bar\omega_0^2 \int_0^1\int_0^1\left(( u_n- u)^2+\frac{\bar\omega_0^2}{n}\right) \leq \frac{\bar\omega_0^2+\bar\omega_0^4}{n}\to 0\qquad \text{as }\, n\to+\infty.
\end{align*}
We deduce from the arbitrariness of $q^*$ that $R'\subseteq cl(R)$ and consequently $int(R')\subseteq R$. The proof is complete. 
\end{proof}
We conclude this section by employing Proposition \ref{th-reach} to compute the boundary of the reachable set $R$. To this end, we choose a uniform grid on $[0,1]$ and 
we discretize the integral representation of the tentacle tip  
$$q[u](1)=\int_0^1\left(\sin\left(\int_0^s \bar\omega_0 u(\xi)\,d\xi\right),
-\cos\left(\int_0^s \bar\omega_0 u(\xi)\,d\xi\right)\right)ds$$
by means of a rectangular quadrature rule. Then we let $u$ exhaustively attain all the extremal configurations of type \eqref{uCL} projected on the grid.  
Figure \ref{reachset} shows the results, depending on the constant parameter $\bar\omega_0$. 

\begin{figure}[!h]
\centering
\begin{tabular}{ccccc}
 \includegraphics[scale=0.225]{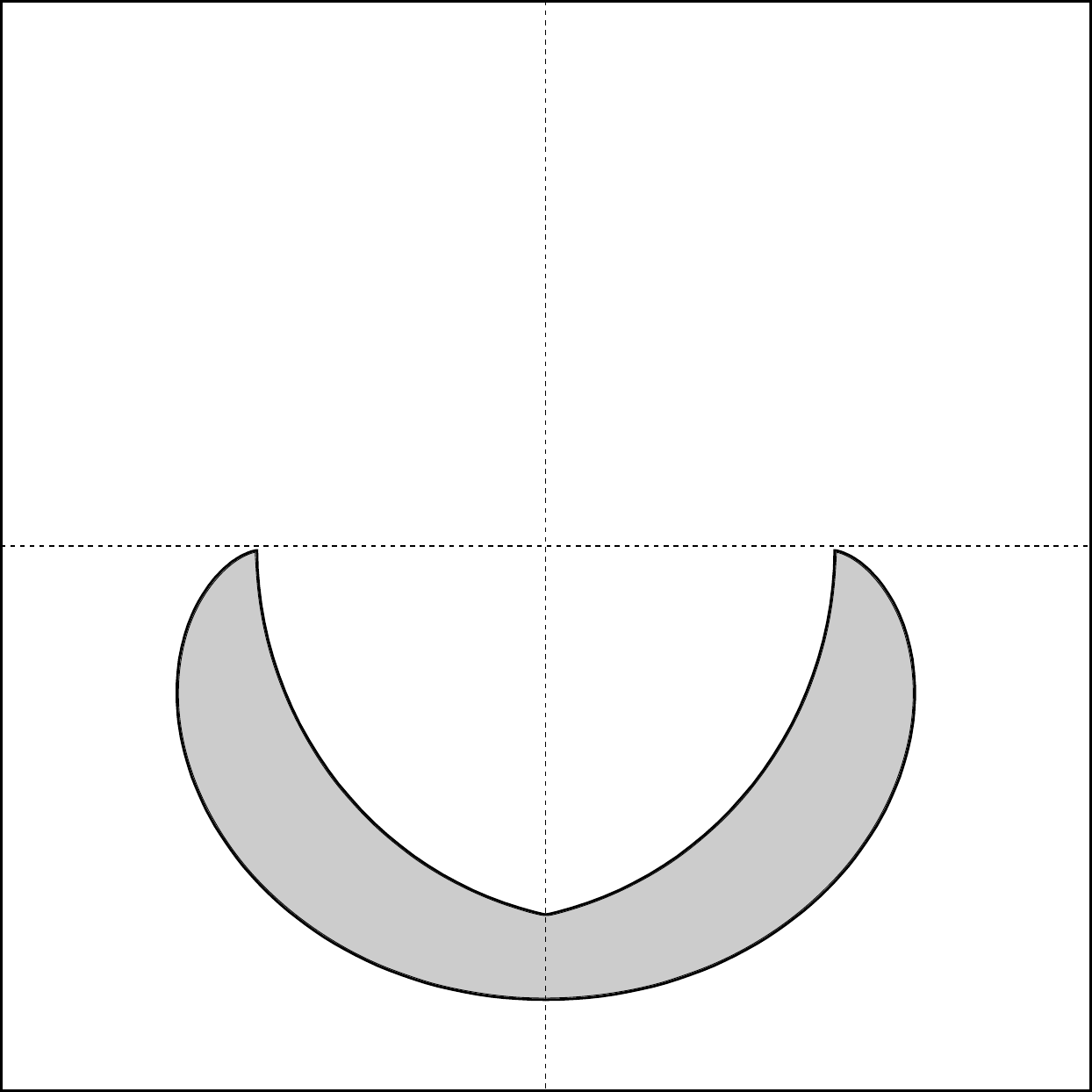}&
 \includegraphics[scale=0.225]{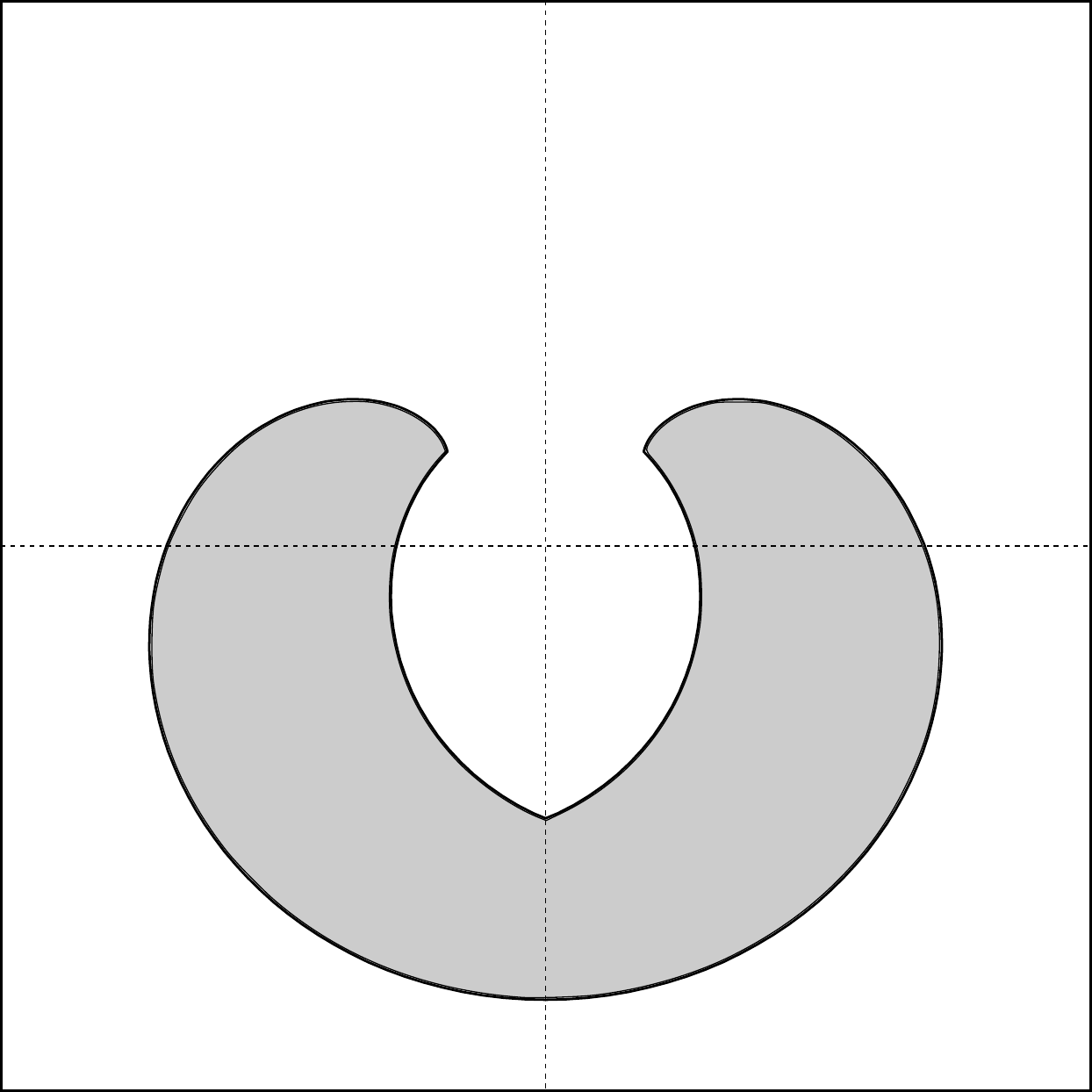}&
 \includegraphics[scale=0.225]{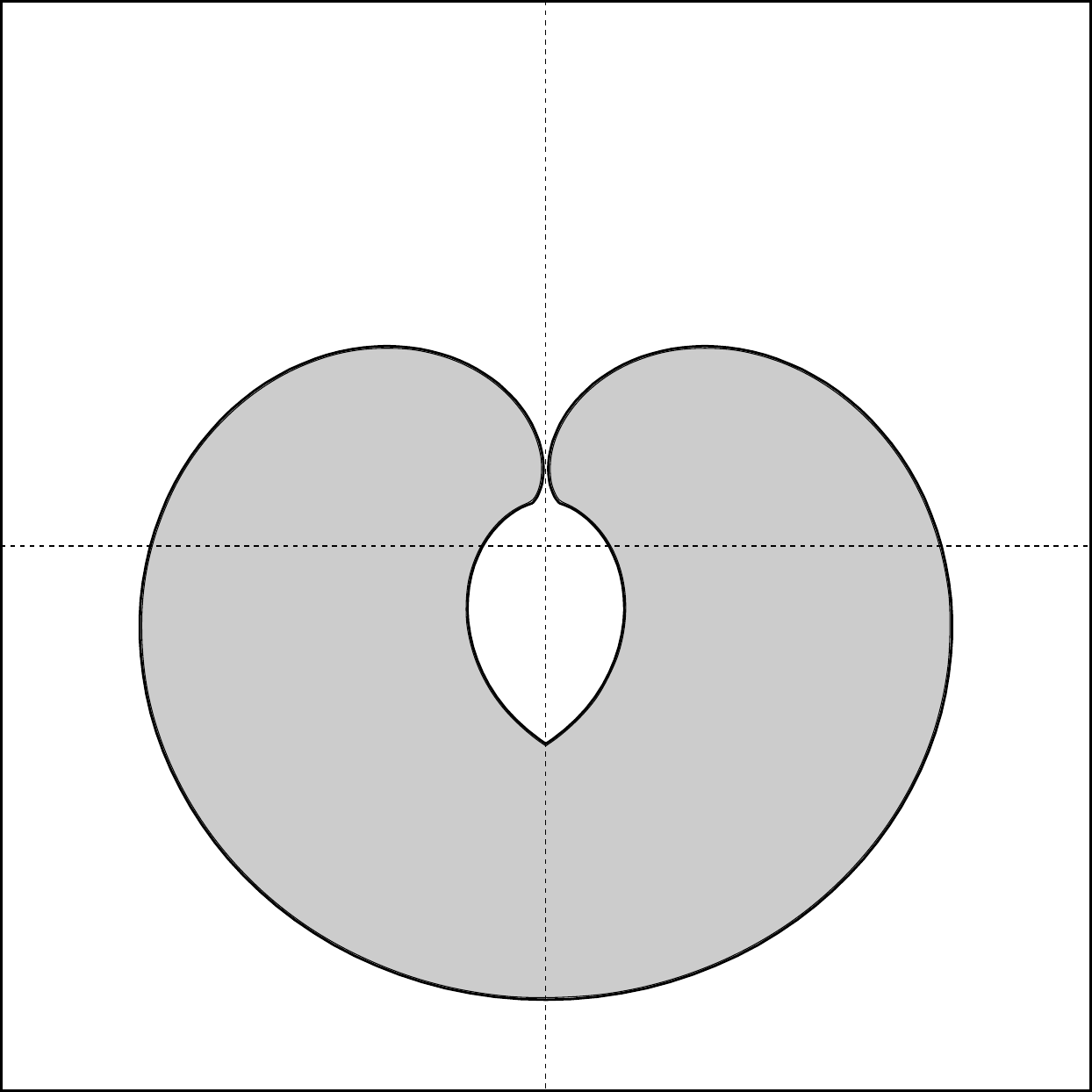}&
 \includegraphics[scale=0.225]{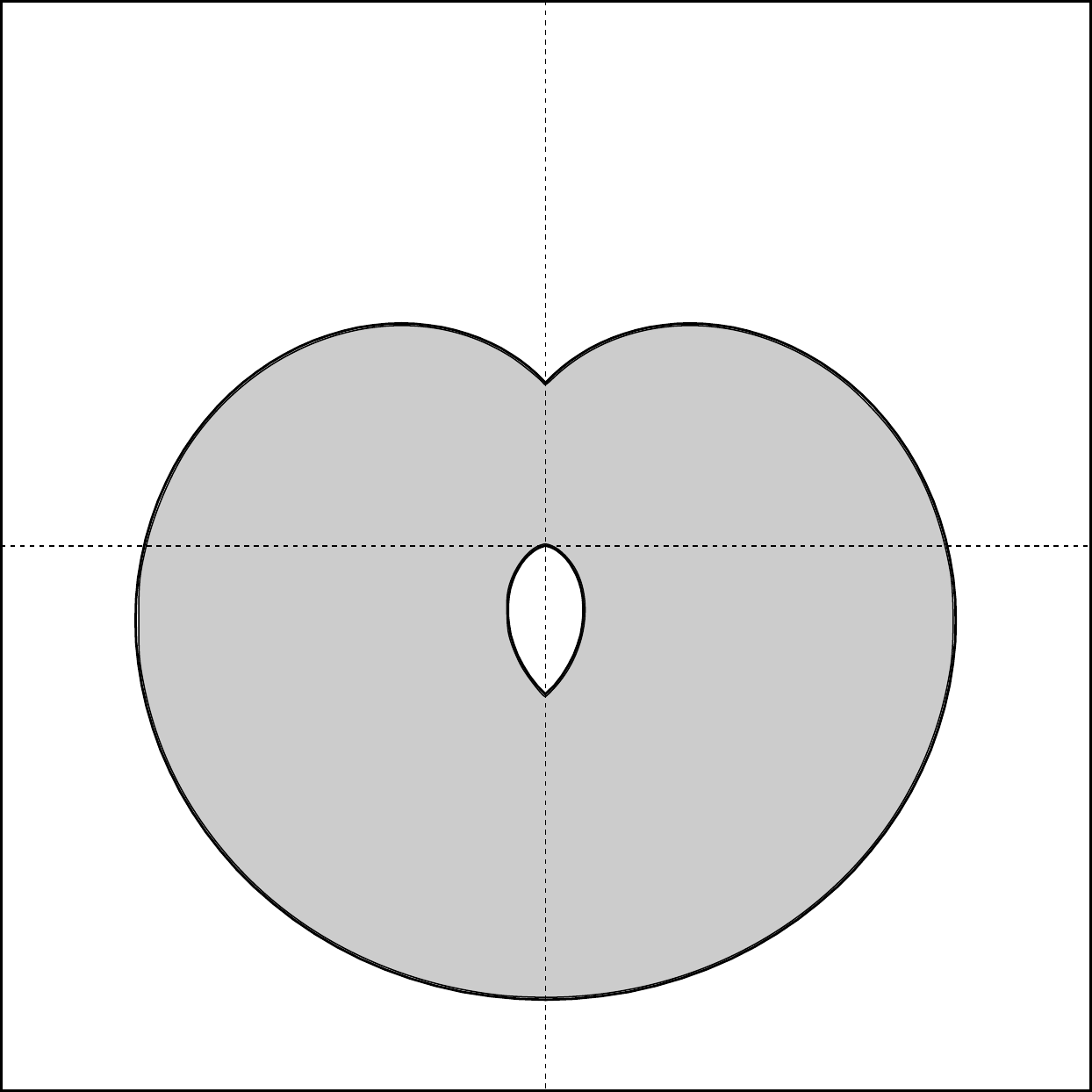}&
 \includegraphics[scale=0.225]{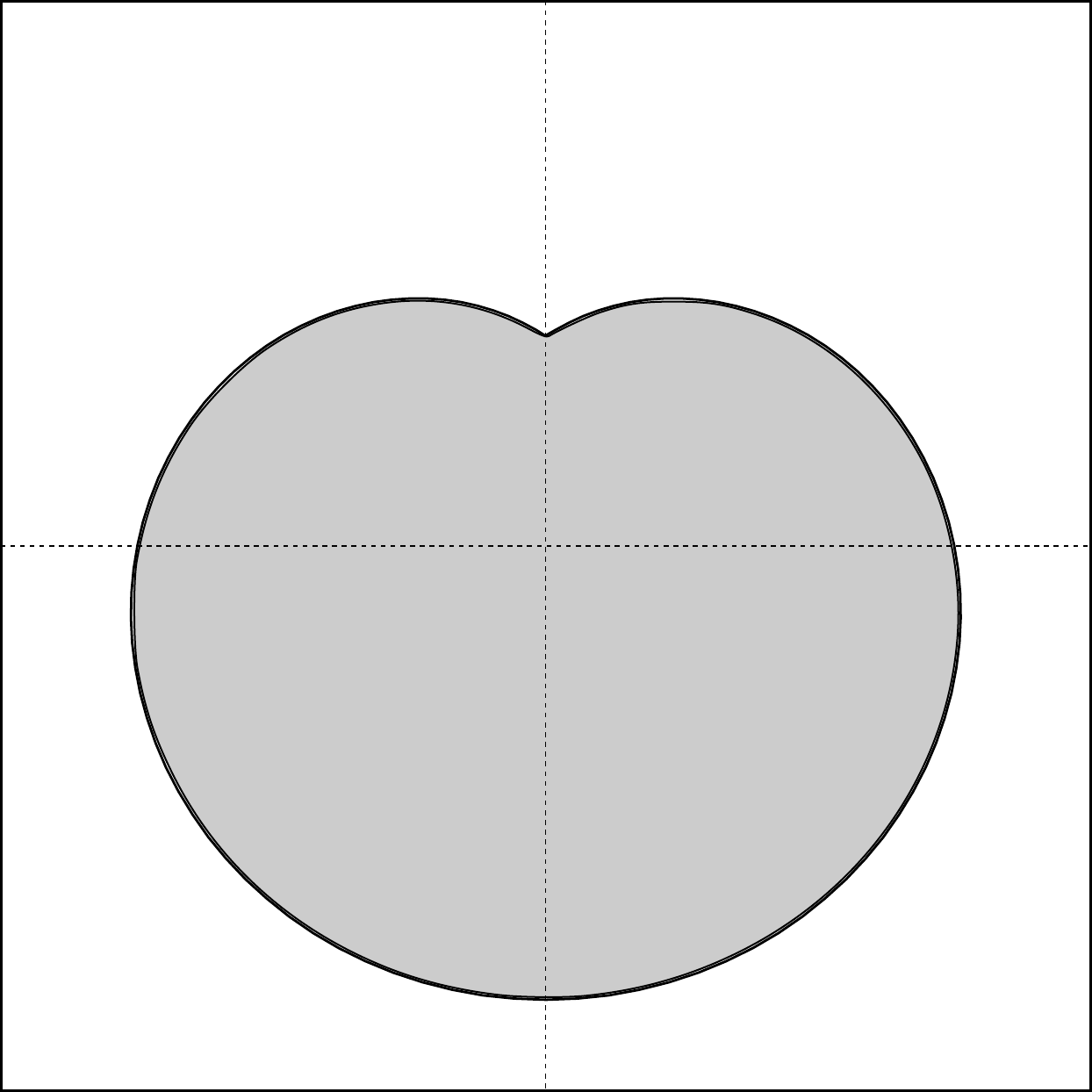}\\
 $\bar \omega_0= \pi$ &
 $\bar \omega_0= \frac32\pi$ &
 $\bar \omega_0= \frac32\pi+1$ &
 $\bar \omega_0= 2\pi$ &
 $\bar \omega_0= \frac94\pi$
\end{tabular}
\caption{reachable set depending on $\bar\omega_0$}\label{reachset}
\end{figure}
\noindent Starting from the extremal value $\bar \omega_0=0$, for which $(0,-1)$ is clearly the unique reachable point, 
the set $\mathcal{R}$ grows up to form a cardiod, enclosing a region of unreachable points. As $\bar \omega_0$ increases, the hole disappears, while for 
$\bar \omega_0\to\infty$, $\mathcal{R}$ converges to the full circle of radius $1$, which is exactly the reachable set of an inextensible string without curvature constraints.\\
Numerical evidence shows that Proposition \ref{th-reach} holds also in the case of non constant curvature $\bar \omega$, as for the simulation presented in 
the next section. Nevertheless, a complete proof of such result, based on Pontryagin's maximum principle, is still under investigation.
\subsection{An augmented Lagrangian method}
Here we propose a numerical method for solving the stationary optimal control problem \eqref{staticfunctional}, and we present some simulations. 
First of all, we relax the inequality constraint on the curvature in \eqref{staticfunctional}, by introducing a so called {\em slack variable}, namely we define 
$z:=\bar\omega^2-|q_{ss}|^2$, so that $|q_{ss}|\le \bar \omega$ if and only if $z\ge0$.
Note that we still have an inequality constraint in $z$, but it is much simpler to treat, as we show below. 
We introduce the following augmented Lagrangian
$$\mathcal{L}(q,\sigma,z,\lambda)=\mathcal{J}(q)+\frac12\hspace{-2pt}\int_0^1 \hspace{-5pt}\sigma(|q_s|^2-1)ds+\frac12\hspace{-2pt}\int_0^1\hspace{-5pt}\lambda(|q_{ss}|^2-\bar\omega^2+z)ds+\frac{1}{4\rho_\lambda}\hspace{-2pt}\int_0^1\hspace{-5pt}(|q_{ss}|^2-\bar\omega^2+z)^2ds\,,
$$
where
\begin{equation}\label{staticfun}
\mathcal{J}(q)=\frac12\hspace{-2pt}\int_0^1 \hspace{-5pt}\frac{1}{\bar\omega^2}|q_{ss}|^2 ds+\frac{1}{2\tau}|q(1)-q^\ast|^2\,,
\end{equation}
$\sigma$ is again an exact Lagrange multiplier for the inextensibility constraint, while 
$\lambda$ and $\rho_\lambda>0$ are respectively the multiplier and penalty parameter related to the relaxed constraint $|q_{ss}|^2-\bar\omega^2+z=0$. 
We apply the classical method of multipliers \cite{multipliers-H,multipliers-P} to compute a solution of our minimization problem, namely we iterate on $i\ge 0$ up to convergence
$$
\left\{\begin{array}{l}(q^{(i+1)},\sigma^{(i+1)},z^{(i+1)})=\arg\hskip-2pt\displaystyle\min_{\hskip-13ptq,\,\sigma,\,\,\,z\ge 0}\mathcal{L}(q,\sigma,z,\lambda^{(i)})\\\\
\lambda^{(i+1)}=\lambda^{(i)}+\displaystyle\frac{1}{\rho_\lambda}(|q_{ss}^{(i+1)}|^2-\bar\omega^2+z^{(i+1)})
\end{array}
\right.$$
where the optimization sub-problem for fixed $\lambda^{(i)}$ is solved as follows, by imposing first order optimality conditions. 

Taking the variation of $\mathcal{L}$ with respect to $\sigma$ in direction $\chi$, we immediately recover the inextensibility constraint. Indeed,
$$\langle\delta_\sigma\mathcal{L},\chi\rangle=\int_0^1 (|q_s|^2-1)\chi ds=0\quad\forall\,\chi\quad
\Longrightarrow\quad |q_s|^2=1 \quad\mbox{a.e. in $(0,1)$}.
$$
On the other hand, taking the variation of $\mathcal{L}$ with respect to $z$ in direction $v$ and imposing the constraint $z\ge 0$, we get the variational inequality
$$
\langle\delta_z\mathcal{L},v\rangle=\frac12\int_0^1\left(\lambda^{(i)}+\displaystyle\frac{1}{\rho_\lambda}(|q_{ss}|^2-\bar\omega^2+z)\right)(v-z)ds\ge 0\qquad \forall\,v\ge 0\,,
$$
from which we readily obtain the optimal value 
$$
z=\max\left\{-\lambda^{(i)}\rho_\lambda-|q_{ss}|^2+\bar\omega^2,0\right\}\quad\mbox{a.e. in $(0,1)$}.
$$
Now, taking the variation of $\mathcal{L}$ with respect to $q$ in direction $w$ and substituting the above expression for the optimal $z$, we get
$$
\langle\delta_q\mathcal{L},w\rangle=\hskip-5pt\int_0^1\hskip-5pt\left( \Lambda(q_{ss},\lambda^{(i)})q_{ss}\cdot w_{ss} 
+\sigma q_s\cdot w_s\right)\,ds+\frac{1}{\tau}(q(1)-q^\ast)\cdot w(1)=0\qquad \forall\,w
$$
with
$$
\Lambda(q_{ss},\lambda^{(i)})= \frac{1}{\bar\omega^2}+\max\left\{\lambda^{(i)}+\frac{1}{\rho_\lambda}(|q_{ss}|^2-\bar\omega^2),0\right\}\,.
$$
Integrating by parts and imposing boundary conditions, we end up with the following strong formulation of the optimality conditions:
\begin{equation}\label{staticcontrol}
\left\{\begin{array}{ll}
                  \left(\Lambda(q_{ss},\lambda^{(i)})q_{ss}\right)_{ss} -\left(\sigma q_s\right)_s=0 & \mbox{in }(0,1)\\
                  |q_s|^2=1 & \mbox{in }(0,1)\\
                  q(0)=0,\quad q_{s}(0)=-e_2\\q_{ss}(1)=0,\quad q_{sss}(1)=0\\
                  \sigma(1)q_s(1)+\frac{1}{\tau}(q(1)-q^\ast)=0
                 \end{array}
\right.
\end{equation}

We remark that, apart the nonlinear term in $\Lambda$ related to the curvature constraint, the first equation is exactly the stationary Euler's elastica equation, 
completed with cantilevered boundary conditions. 
On the other hand, the last equation quantifies the balance between the tension of the tentacle and the boundary force resulting from 
the potential energy $\frac{1}{2\tau}|q(1)-q^\ast|^2$ associated to the target point.
More explicitly, dot-multiplying this equation by $q_s$ and using the inextensibility constraint, we obtain
$$
\sigma(1)=-\frac{1}{\tau}(q(1)-q^\ast)\cdot q_s(1)\,,
$$
i.e., the tension at the free end is exactly opposite to the tangent projection of the boundary force.

We discretize the optimality system \eqref{staticcontrol} by means of a finite difference scheme as in Section \ref{Sec:Discretization}, obtaining
\begin{equation}\label{staticcontroldiscrete}
\left\{\begin{array}{ll}
                  D^2_c\left(\Lambda(D^2_c q_{k},\lambda^{(i)})D^2_c q_{k}\right) -\Dp\left(\sigma_k \Dm q_k\right)=0 & k=1,...,N-1\\
                  |\Dm q_k|^2=1 & k=1,...,N-1\\
                  q_0=0,\quad q_{-1}=q_0+e_2\Delta s\\
                  q_{N+1}-2q_N+q_{N-1}=0\\
                  q_{N+1}-3q_N+3q_{N-1}-q_{N-2}=0\\
                  \sigma_N \Dm q_N+\frac{1}{\tau}(q_N-q^\ast)=0
                 \end{array}
\right.
\end{equation}
This is again a nonlinear system in the pair of unknowns $(q_k,\sigma_k)_{k=1,...,N}$, 
and it can be easily solved by means of a quasi-Newton method. More precisely, at each iteration, the Newton step for \eqref{staticcontroldiscrete} is computed 
by freezing the nonlinearity in the term $\Lambda(D^2_c q_{k},\lambda^{(i)})$ at the previous iteration, 
namely by replacing the full Jacobian of $D^2_c\left(\Lambda(D^2_c q_{k},\lambda^{(i)})D^2_c\right)$ with 
$\Lambda(D^2_c q_{k},\lambda^{(i)})D^2_c D^2_c$ only. This is a reasonable approximation as we approach a solution,  
indeed, by definition, $\Lambda$ does not depend on $q$ whenever the curvature constraint is satisfied. \\
Finally, to evaluate convergence, we discretize the functional $\mathcal{J}$ in \eqref{staticfun} by means of a simple rectangular quadrature rule:
\begin{equation}\label{staticfundis}
\mathcal{J^\sharp}(q)=\frac12\Delta s\sum_{k=0}^N\frac{1}{\bar\omega^2_k}|D^2_c q_k|^2+\frac{1}{2\tau}|q_N-q^\ast|^2\,.
\end{equation}
Summing up, we have the following augmented Lagrangian algorithm: 
\begin{algorithm}
\begin{algorithmic}[1]
\State Assign $q^{(0)}, \lambda^{(0)}$, a target point $q^\ast$, penalties $\tau, \rho_\lambda>0$ and a tolerance $tol>0$. Compute $\mathcal{J^\sharp}(q^{(0)})$ and set $i=0$
\Repeat
\State Compute the solution $(q^{(i+1)},\sigma^{(i+1)})$ of \eqref{staticcontroldiscrete}
\State Update the multipliers
$\lambda_k^{(i+1)}=\max\left\{\lambda_k^{(i)}+\displaystyle\frac{1}{\rho_\lambda}(|D^2_c q_{k}^{(i+1)}|^2-\bar\omega^2),0\right\}$ for $k=1,...,N-1\,.$
\State Compute $\mathcal{J^\sharp}(q^{(i+1)})$ and set $i=i+1$\medskip
\Until {$\left|\mathcal{J^\sharp}(q^{(i)})-\mathcal{J^\sharp}(q^{(i-1)})\right| < tol$}
\end{algorithmic}
\caption{}\label{ALG1}
\end{algorithm}

\noindent Note that this algorithm consists in two nested loops, an external loop for the update of the multipliers, and an internal loop for the computation of the solution. In 
practice, we realized that a single loop is enough to obtain and speed up the convergence, by performing both the update and the Newton step simultaneously. \\
We choose the parameters for the simulations as in Section \ref{Sec:Discretization}. Moreover, we set $q^{(0)}=-e_2 s$, $\lambda^{(0)}=0$, $\tau=10^{-4}$, $\rho_\lambda=10^2$ and $tol=10^{-8}$.  
In Figure \ref{Fig:static-control}, we show the corresponding reachable set, and some optimal solutions computed by Algorithm \ref{ALG1} for different target points. 
For each configuration 
we also report the graph of the signed curvature $\kappa=\Dm q \times D^2_c q$ (thick line), compared with the bounds $[-\bar\omega,\bar \omega]$ (thin lines). 
In particular, we observe that the reachable set is a cardiod also in the case of non constant curvature $\bar \omega$. Moreover, 
we recognize a {\em Dubins path} in the top-right solution, formed by a $CL$-like curve as in \eqref{uCL}, with a circular part of 
variable maximal radius $1/\bar\omega$ plus a straight segment. 
\begin{figure}[!h]
\centering
\includegraphics[width=.875\textwidth]{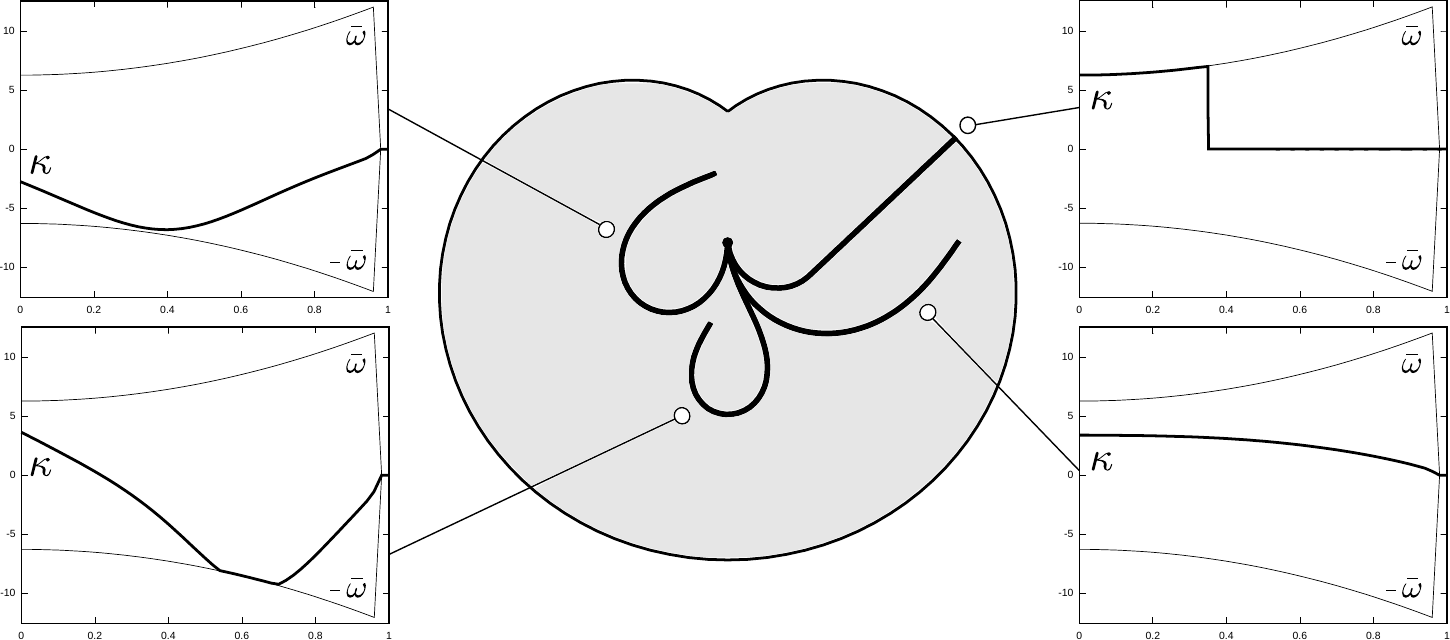}
\caption{solutions of the stationary optimal control problem for different target points}\label{Fig:static-control}
\end{figure}
 
\section{The dynamic optimal control problem}\label{Sec:Dynamic}
In this section we tackle the dynamic optimal control problem discussed in the Introduction. More precisely, 
given $q^\ast\in\RR^2$ and $T, \tau>0$, we want to minimize the functional
\begin{equation}\label{dynamicfunc}
\mathcal{J}(q,u)=\frac{1}{2}\int_0^T\hskip-5pt\int_0^1\hskip-5pt u^2 ds\,dt+\frac{1}{2\tau}\int_0^T\hskip-5pt|q(1,t)-q^\ast|^2 dt
+\frac{1}{2}\int_0^1\hskip-5pt\rho(s)|q_t(s,T)|^2 ds\,,
\end{equation}
among all the controls $u:[0,1]\times[0,T]\to[-1,1]$, and subject to the tentacle dynamics \eqref{tentaclemotion}. 
We observe that, compared to the functional of the stationary control problem \eqref{touchfunctional}, 
the first two terms in $\mathcal{J}$ account respectively for the activation of the tentacle muscles and the tip-target distance during {\em all} all evolution, 
whereas the last term corresponds to the kinetic energy at the final time $T$. 
The problem is then to reach and keep the tip close to the target point with minimum effort, but also to stop the whole tentacle as soon as possible.  
After deriving first order optimality conditions, we propose and implement an adjoint-based gradient descent algorithm for its numerical solution.  
\subsection{First order optimality conditions}
We introduce the adjoint state $(\bar q,\bar \sigma)$, with $\bar q:[0,1]\times[0,T]\to\RR^2$, $\bar\sigma:[0,1]\times[0,T]\to\RR$, and we form the following Lagrangian
\begin{align*}
\bar{\mathcal L}(q,\sigma,u,\bar q,\bar\sigma):=&\mathcal J(q,u) +  \int_0^{T}\int_0^1\bar 
q\cdot\left\{\rho q_{tt}-(\sigma q_s-H q_{ss}^\bot)_{s}+(Gq_{ss}+Hq_{s}^\bot)_{ss}\right\}\,ds\,dt+\frac{1}{2}\int_0^{T}\int_0^1\bar \sigma(|q_s|^2-1)\,ds\,dt\,.
\end{align*}
By construction, the stationarity of $\bar{\mathcal L}$ with respect to the adjoint variables immediately gives back 
the equations of motion and the inextensibility constraint. 
We then focus on the stationarity of $\bar{\mathcal L}$ with respect to $q$, $\sigma$ and $u$, 
which involves long and technical computations. 
Here we report the key steps, in particular the derivation of the boundary conditions for the adjoint system.
To this end, we recall, for the reader's convenience, the definitions of $G$ and $H$ given in \eqref{GH},
$$G[q,\nu,\varepsilon,\omega](s,t):=\varepsilon(s)+\nu(s)\left(|q_{ss}(s,t)|^2-\omega^2(s)\right)_+,$$
$$H[q,\mu,u,\omega](s,t):=\mu(s)\left(\omega(s) u(s,t)-q_s(s,t) \times q_{ss}(s,t)\right)\,.$$
and we denote the corresponding linearizations by
$$\bar G[q,\bar q,\nu,\omega](s,t)=g[q,\nu,\omega](s,t)q_{ss}(s,t)\cdot \bar q_{ss}(s,t)\,,$$
$$\bar H[q,\bar q,\mu])(s,t)=-\mu(s)\left(\bar q_s(s,t) \times q_{ss}(s,t)+q_s(s,t) \times \bar q_{ss}(s,t)\right)\,,$$
where $g[q,\nu,\omega](s,t)=2\nu(s)\mathbf{1}(|q_{ss}(s,t)|^2-\omega^2(s))$ and $\mathbf{1}(\cdot)$ stands for the Heaviside function, i.e. $\mathbf{1}(x)=1$ for $x\ge 0$ and 
$\mathbf{1}(x)=0$ otherwise.  
As usual, we assume enough regularity to perform the computations, and we employ the given boundary conditions to define the admissible tests $(w,\chi)$ 
associated to $(q,\sigma)$. Moreover, we recall the assumptions \eqref{epsmu} on the functions $\varepsilon, \mu$ and some useful properties of $G, H$. 
We summarize all these relations in the following list: for $(s,t)\in[0,1]\times[0,T]$
\begin{equation}\label{boundarygen}
\begin{array}{lll}
q(s,0)=q^0(s)& \Rightarrow & w(s,0)=0\\
q_t(s,0)=v^0(s)&\Rightarrow & w_t(s,0)=0\\
q(0,t)=0&\Rightarrow & w(0,t)=0\\
q_s(0,t)=-e_2&\Rightarrow & w_s(0,t)=0\\
q_{ss}(1,t)=0&\Rightarrow & w_{ss}(1,t)=0\\ 
q_{sss}(1,t)=0&\Rightarrow & w_{sss}(1,t)=0\\ 
\sigma(1,t)=0&\Rightarrow & \chi(1,t)=0\\ 
\varepsilon(s)\geq \varepsilon_0>0 &\Rightarrow& G(1,t)=\varepsilon(1)>0\\
 \mu(1,t)=\mu_s(1,t)=0 &\Rightarrow&  H(1,t)=H_s(1,t)=0\\
\end{array}
\end{equation}
\noindent Let us start by computing the variation of $\bar{\mathcal{L}}$ with respect to $\sigma$ in direction $\chi$:
$$
\langle\delta_\sigma\bar{\mathcal L},\chi\rangle=-\int_0^{T}\int_0^1\bar 
q\cdot(\chi q_s)_s\,ds\,dt= -\int_0^{T}[ \bar q\cdot q_s\chi]_0^1\,dt+\int_0^{T}\int_0^1 \bar q_s\cdot q_s \chi\,ds\,dt\,.
$$
Using \eqref{boundarygen} and imposing stationarity, by the arbitrariness of $\chi$, we get for almost every $(s,t)\in(0,1)\times (0,T)$
\begin{equation}\label{adjointconstr}
 \bar q(0,t)\cdot q_s(0,t)=0\qquad\mbox{and}\qquad\bar q_s(s,t)\cdot q_s(s,t)=0\,,
\end{equation}
that we prolong by continuity to all $(s,t)\in[0,1]\times[0,T]$.  \\

\noindent We now compute the variation of $\bar{\mathcal{L}}$ with respect to $q$ in direction $w$:
\begin{align*}
\langle\delta_q\bar{\mathcal L},w\rangle &=\int_0^T(q(1,t)-q^\ast)\cdot w(1,t)\,dt\\&+\int_0^1\rho q_t(s,T)\cdot w_t(s,T) ds\\
&+  \int_0^{T}\int_0^1\rho \bar q\cdot w_{tt}\,ds\,dt\\&
-\int_0^{T}\int_0^1\bar q\cdot (\sigma w_s+ \bar H[q,w] q_{ss}^\bot+H w_{ss}^\bot)_{s} \,ds\,dt\\
&+ \int_0^{T}\int_0^1 \bar q\cdot ( \bar G[q,w] q_{ss} + Gw_{ss}+\bar H[q,w] q_{s}^\bot+ H w_{s}^\bot)_{ss}\,ds\,dt\\
&+\int_0^{T}\int_0^1\bar\sigma q_s\cdot w_s\,ds\,dt\,,
\end{align*}
\noindent Using again \eqref{boundarygen}, after a very long integration by parts, we obtain
\begin{equation}\label{variation}
\end{equation}
\begin{align*}\langle \delta_q \bar{\mathcal L},w\rangle&=\int_0^T\int_0^1 \left\{\rho\bar q_{tt}
-\left(\sigma \bar q_s - H \bar q_{ss}^\bot + \bar \sigma q_s -\bar H q_{ss}^\bot \right)_s+\left(G \bar q_{ss} + H \bar q_{s}^\bot + \bar G q_{ss} +\bar H q_{s}^\bot \right)_{ss}\right\}\cdot w\,ds\,dt\\
&-\int_0^1 \rho\bar q_t\cdot w(s,T)\,ds\\&+\int_0^1 \rho(q_t+\bar q)\cdot w_t(s,T)\,ds\\
&+\int_0^T \varepsilon\bar q_{ss}\cdot w_s(1,t)\,dt\\&+\int_0^T \left\{\bar \sigma q_s-(G\bar q_{ss})_s+\frac{1}{\tau}(q-q^*)\right\}\cdot w(1,t)\,dt\\
&-\int_0^T\left\{g(\bar q\cdot q_{ss})q_{ss}+G\bar q+\mu(\bar q\cdot q_s^\bot)q_s^\bot\right\}\cdot w_{sss}(0,t)\,dt\\
&+\int_0^T\left\{2H \bar q^\bot-G_s\bar q-\bar q\cdot(2\mu q_{ss}^\bot+\mu_s q_s^\bot)q_s^\bot - \bar q\cdot(g q_{ss})_sq_{ss}
\right.\\
&\hskip.9cm \left.-g(\bar q\cdot q_{ss})q_{sss}+g(\bar q_s\cdot q_{ss})q_{ss}+G\bar q_s+\mu(\bar q_s\cdot q_s^\bot)q_s^\bot\right\}\cdot w_{ss}(0,t)\,dt\,.
\end{align*}
Imposing stationarity, by the arbitrariness of $w$, it follows that all the integrands in the expression above should vanish 
for almost every $(s,t)\in(0,1)\times (0,T)$. In order, we get the adjoint equation
$$
\rho\bar q_{tt}=\left(\sigma \bar q_s - H \bar q_{ss}^\bot + \bar \sigma q_s -\bar H q_{ss}^\bot \right)_s-
\left(G \bar q_{ss} + H \bar q_{s}^\bot + \bar G q_{ss} +\bar H q_{s}^\bot \right)_{ss}\,,
$$
and the final conditions 
$$
\bar q(s,T)=-q_t(s,T)\,,\qquad \bar q_t(s,T)=0\,.
$$
Moreover, since $\varepsilon(1)>0$ by \eqref{boundarygen}, we deduce
$$
\bar q_{ss}(1,t)=0\,,
$$
and then we also get 
\begin{equation}\label{qsss}
\bar \sigma(1,t) q_s(1,t)-\varepsilon(1)\bar q_{sss}(1,t)+\frac{1}{\tau}(q(1,t)-q^*)=0\,.
\end{equation}
Now consider the relation $\bar q_s\cdot q_s=0$ in \eqref{adjointconstr}. 
Differentiating twice in $s$, we obtain $q_s\cdot \bar q_{sss}+2q_{ss}\cdot \bar q_{ss}+q_{sss}\cdot \bar q_s=0$, and 
using the boundary conditions $q_{ss}(1,t)=q_{sss}(1,t)=0$, it follows that $q_s(1,t)\cdot \bar q_{sss}(1,t)=0$. 
Dot-multiplying \eqref{qsss} by $q_s$, we then conclude
$$\bar\sigma(1,t)=-\frac{1}{\tau}(q(1,t)-q^*)\cdot q_s(1,t)\,.$$
Substituting in \eqref{qsss} and projecting $(q-q^\ast)$ on the base $\{q_s,q_s^\bot\}$, we also get  
\begin{equation*}
\bar q_{sss}(1,t)=\frac{1}{\tau\varepsilon(1)}\left((q-q^*)\cdot q_s^\bot\right) q_s^\bot (1,t)\,. 
\end{equation*}
We now proceed with the stationarity of the last two integrands in \eqref{variation}, related to the boundary conditions at $s=0$. 
The first one gives
\begin{equation}\label{wsss}
g(\bar q\cdot q_{ss})q_{ss}+G\bar q+\mu(\bar q\cdot q_s^\bot)q_s^\bot=0\,.
\end{equation}
Using the orthogonality relation $\bar q(0,t)\cdot q_s(0,t)=0$ in \eqref{adjointconstr}, for every $t\in[0,T]$ we can write $\bar q=\alpha q_s^\bot$ for some $\alpha\in\RR$. 
On the other hand, due to the relation $q_s\cdot q_{ss}=0$, we can also write $q_{ss}=\beta q_s^\bot$ for some $\beta\in\RR$. Substituting in \eqref{wsss}, we get
$$
\alpha (g\beta^2 +G +\mu)q_s^\bot=0\,.
$$
Since $G>0$ and $g,\mu\ge 0$, we deduce $\alpha=0$ and, consequently 
$$\bar q(0,t)=0\,.$$
Substituting in \eqref{variation}, several terms cancel out and the last integrand reduces to
$$
g(\bar q_s\cdot q_{ss})q_{ss}+G\bar q_s+\mu(\bar q_s\cdot q_s^\bot)q_s^\bot=0\,.
$$
Note that this expression is exactly \eqref{wsss} with $\bar q_s$ in place of $\bar q$. 
Then, using the relation $\bar q_s(0,t)\cdot q_s(0,t)=0$ in \eqref{adjointconstr}, we can repeat the same argument above to conclude
$$\bar q_s(0,t)=0\,.$$

We finally compute the variation of $\bar{\mathcal{L}}$ with respect to the control, assuming for a moment that $u$ is unconstrained, 
and considering smooth admissible tests $v:[0,1]\times[0,T]\to\RR$. We get 
$$
\langle \delta_u \bar{\mathcal L},v\rangle=\int_0^{T}\int_0^1 u v\,ds\,dt+\int_0^{T}\int_0^1\bar q\cdot\{ (\mu\omega v q_{ss}^\bot)_{s} +( \mu\omega v q_{s}^\bot)_{ss}\}\,ds\,dt\,.\\
$$
Integrating by parts and using $\mu(1,t)=\mu_s(1,t)=0$, $\bar q(0,t)=\bar q_s(0,t)=0$, we easily get
\begin{equation}\label{variationu}
\langle \delta_u \bar{\mathcal L},v\rangle=\int_0^{T}\int_0^1 \left( u + \omega\bar H[q,\bar q]\right) v\,ds\,dt\,,\qquad\mbox{i.e.}\qquad \delta_u \bar{\mathcal L}=u + \omega\bar H[q,\bar q] 
\quad\mbox{in \,\,}L^2((0,1)\times(0,T))\,.
\end{equation}
\noindent Summarizing, we end up with the following strong formulation of the adjoint system
\begin{equation}\label{adjoint}
\left\{\begin{array}{ll}
                  \rho \bar q_{tt}=\left(\sigma \bar q_s-H \bar q_{ss}^\bot\right)_s -\left(G \bar q_{ss}+H \bar q_{s}^\bot\right)_{ss} &\mbox{in }(0,1)\times(0,T)\\
                  \qquad +\left(\bar\sigma  q_s-\bar H q_{ss}^\bot\right)_s -\left(\bar G q_{ss}+\bar H q_{s}^\bot\right)_{ss}\\
                  \bar q_s \cdot q_s=0 &\mbox{in }(0,1)\times(0,T)\\
                  \bar q(0,t)=0 & t\in(0,T)\\
                  \bar q_s(0,t)=0& t\in(0,T)\\
                  \bar q_{ss}(1,t)=0& t\in(0,T)\\
                  \bar q_{sss}(1,t)=\frac{1}{\tau\varepsilon}\left((q-q^*)\cdot q_s^\bot\right) q_s^\bot (1,t)& t\in(0,T) \\
                  \bar\sigma(1,t)=-\frac{1}{\tau}(q-q^\ast)\cdot q_s(1,t)& t\in(0,T)\\
                  \bar q(s,T)=-q_t(s,T)& s\in(0,1)\\
                  \bar q_t(s,T)=0& s\in(0,1)
                 \end{array}
\right.
\end{equation}
whereas, using \eqref{variationu} and imposing box constraints, the optimal control $u$ should satisfy, for every $v:[0,1]\times[0,T]\to[-1,1]$, the following variational inequality
\begin{equation}\label{varinequality}
\int_0^{T}\int_0^1 \left( u + \omega\bar H[q,\bar q]\right) (v-u)ds\,dt\ge 0\,.
\end{equation}
It is easy to see that, in case of dissipation, the frictional forces $-\beta q_t$ and $-\gamma q_{sssst}$ in the equations of motion \eqref{tentaclemotionfriction}, 
simply produce in \eqref{adjoint} the analogous (but with opposite signs) terms  $\beta \bar q_t$ and $\gamma \bar q_{sssst}$ respectively.
\subsection{An adjoint-based gradient descent method} From a numerical point of view, solving the optimality system obtained in the previous section is a complicate task,  
since we have to find an optimal $5$-tuple $(q^*,\sigma^*,\bar q^*,\bar \sigma^*,u^*)$ satisfying \eqref{tentaclemotion}, \eqref{adjoint} and 
\eqref{varinequality} in the whole space-time interval $[0,1]\times[0,T]$. 
After discretization, this may result in a huge nonlinear system. Hence we adopt another (and simpler) approach, 
namely we use an iterative method that, at each iteration, splits the solution in three separate steps. More precisely, 
given a guess for the control $u$, we first solve the system of motion \eqref{tentaclemotion} with $u$ fixed, up to the final time. 
Then we solve the adjoint system \eqref{adjoint} backward in time, again with $u$ fixed and using the just computed $(q,\sigma)$ as a datum. 
Finally, we plug $q$ and $\bar q$ in the expression \eqref{variationu} of the gradient $\delta_u \bar{\mathcal L}$, and update the control $u$ 
along the corresponding descent direction. The procedure is iterated until convergence on the control. 
For the discretization of both forward and backward systems, we employ the Velocity Verlet scheme presented in Section \ref{Sec:Discretization}, 
whereas the box constraints on the control are enforced, at each iteration, by the projection on the interval $[-1,1]$, denoted by $\Pi_{[-1,1]}(\cdot)$. Summarizing, we build the following algorithm.\\
 \begin{algorithm}
\begin{algorithmic}[1]
\State Assign $u^{(0)}$, a target point $q^\ast$, a penalty $\tau>0$, a tolerance $tol>0$, a step size $\alpha>0$ and set $i=0$
\Repeat
\State Compute the solution $(q^{(i)},\sigma^{(i)})$ of \eqref{tentaclemotion}
\State Compute the solution $(\bar q^{(i)},\bar \sigma^{(i)})$ of \eqref{adjoint}
\State Compute $\delta_u \bar{\mathcal L}(q^{(i)},\bar q^{(i)})$ using \eqref{variationu}
\State Update the control $u^{(i+1)}=\Pi_{[-1,1]}\left(u^{(i)}-\alpha \delta_u \bar{\mathcal L}(q^{(i)},\bar q^{(i)})\right)$ and set $i=i+1$
\Until {$\left\| u^{(i)}-u^{(i-1)}\right\|_\infty\hspace{-3pt} < tol$}
\end{algorithmic}
\caption{}\label{ALG2}
\end{algorithm}
Let us discuss how to build a suitable initial guess $u^{(0)}$ for the algorithm above. 
Once $q^\ast$ is fixed, we employ Algorithm \ref{ALG1} for the stationary control problem. 
By construction, this provides an equilibrium $\tilde q(s)$ minimizing the distance of the tentacle tip from $q^\ast$, and we can easily synthesize 
the corresponding optimal control $\tilde u(s)$ using the relation 
$\tilde u(s)=\frac{1}{\bar \omega(s)}\tilde q_s(s)\times \tilde q_{ss}(s)$. Then, we define $u^{(0)}$ extending $\tilde u$ to a function on $[0,1]\times[0,T]$ 
which is constant in time, i.e. we set $u^{(0)}(s,t):=\tilde u(s)$ for all $t\in[0,T]$, and we refer to it as to the {\em static optimal control}. 
It is clear that, if we plug $u^{(0)}$ in the controlled dynamical system \eqref{tentaclemotion}, the corresponding evolution will convergence to $\tilde q$ 
only in presence of friction and for $T\to\infty$ (see the simulations in Section \ref{Sec:Discretization}). Here the aim is to show that Algorithm \ref{ALG2} can 
do much better, producing a {\em dynamic optimal control} that changes both in space and time. To this end, we choose the parameters for the simulations as in Section \ref{Sec:Discretization}, with the exception of the friction coefficients, 
that we set as $\beta(s)=2-s$ and $\gamma(s)=10^{-6}(2-s)$. In this way, we accentuate the elastic behavior of the tentacle, 
and we can focus on the reduction of oscillations performed by the optimization. 
Moreover, we set $T=4$, $\tau=10^{-4}$, $tol=10^{-6}$ and $q^\ast=(0.5,-0.25)$. 
We also reduce the number of joints to $N=10$ to have a reasonable computational time, and we use a fixed gradient descent step size $\alpha$ for a simple code implementation.
In this setting, we found that $\alpha=10^{-4}$ produces an almost monotone decrease of the functional $\mathcal{J}$ in \eqref{dynamicfunc} up to convergence of the control. 
Finally, for a qualitative analysis of the results, it is useful to compute the following discretizations of the three terms appearing in $\mathcal{J}$,
$$
\mathcal{J}_u(n)=\frac{\Delta s}{2}\sum_{k=0}^N (u_k^n)^2\,,\quad \mathcal{J}_{q^\ast}(n)=\frac{1}{2\tau}|q_N^n-q^\ast|^2\,,\qquad
\mathcal{J}_{v}(n)=\frac{\Delta s}{2}\sum_{k=0}^N\rho_k|v_k^n|^2\,,
$$
respectively the control energy, the target energy and the kinetic energy at time $t_n=n\Delta t$.

Figure \ref{static-vs-dynamic} shows the resulting evolution at different times. 
For comparison, we represent in each frame the profile associated to the dynamic optimal control (black line), 
the profile associated to the static optimal control (gray line), 
and the equilibrium configuration (dashed line) corresponding to the target point $q^\ast$ (small circle). 
Moreover, Figure \ref{energy-static} and Figure \ref{energy-dynamic} show the behavior in time of the energies 
$\mathcal{J}_{v}$, $\mathcal{J}_{q^\ast}$ and $\mathcal{J}_{u}$, respectively for the static and the dynamic controls.  

While the static control keeps the tentacle oscillating around the target point, waiting for the slow friction dissipation, the optimized dynamics is really rich and deserves some comments. 
We observe that, at the very beginning, the dynamic control strongly activates to compensate the high acceleration given by the target potential, then relaxes to 
reduce the control energy ($0\le t\le 0.25$). 
This produces a first relevant decrease of the tip-target distance, but also some vibrations, clearly visible in the kinetic energy. They 
travel back and forth along the tentacle, and the control reacts again to mitigate them ($0.25\le t\le 0.5$). 
This is the key strategy to stabilize the tip with minimum effort, and it repeats in time, since each increment in the control energy produces new vibrations. Note also that, due to the boundary conditions, the last three joints are aligned, and they cannot be controlled 
by the assumptions \eqref{epsmu} on the elastic constant $\mu$. This suggests that the tip stabilization can be faster as $N\to\infty$, and we can also expect a smoother 
behavior of the kinetic energy. Finally, we observe that, as $t\to T=4$, the dynamic control energy converges to the static one. This is not a mere consequence of 
the choice for the initial guess $u^{(0)}$, since we know that the static control is optimal if the kinetic energy vanishes.
\begin{figure}[!h]
\centering
\begin{tabular}{ccccc}
 \includegraphics[width=.1725\textwidth]{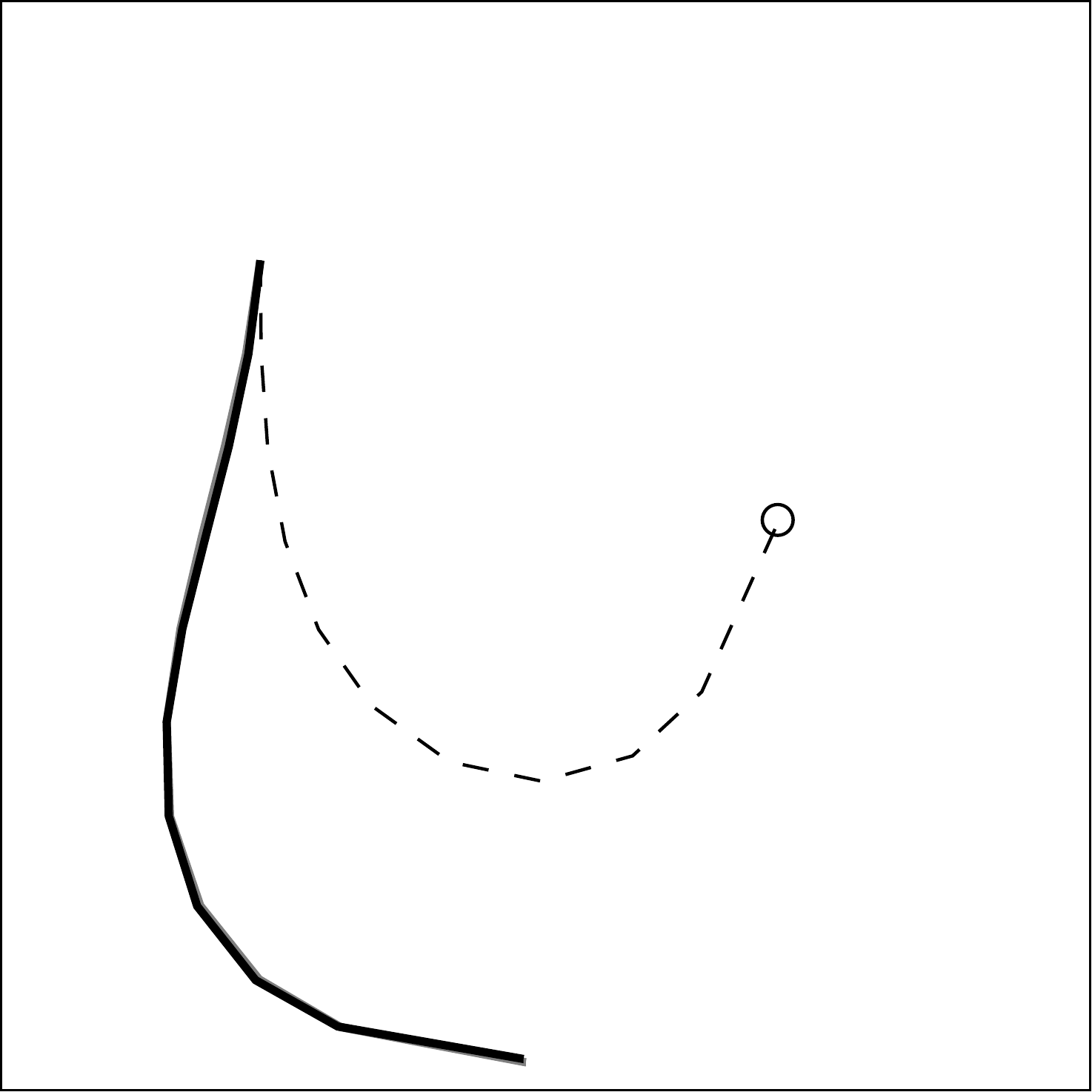}&
 \includegraphics[width=.1725\textwidth]{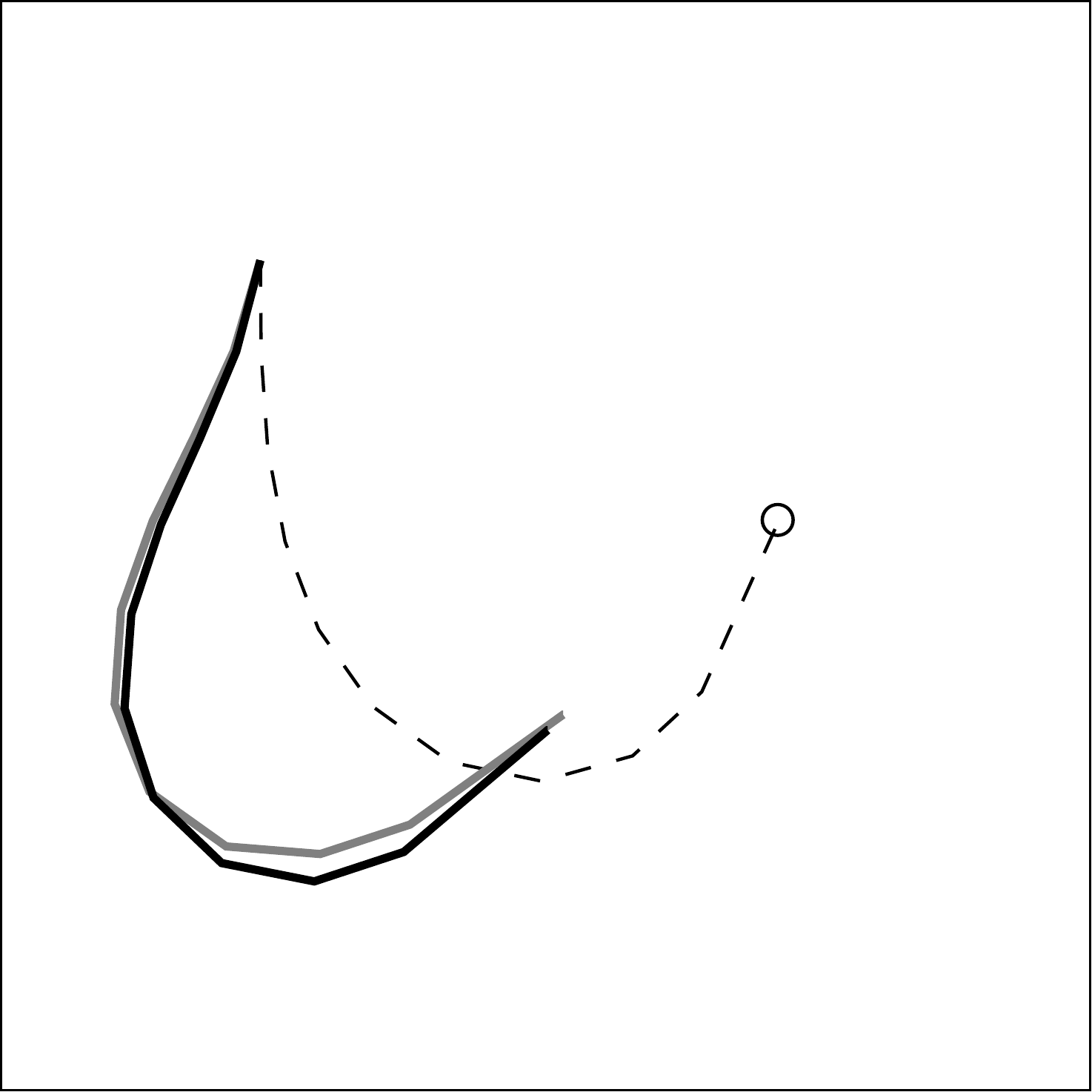}&
 \includegraphics[width=.1725\textwidth]{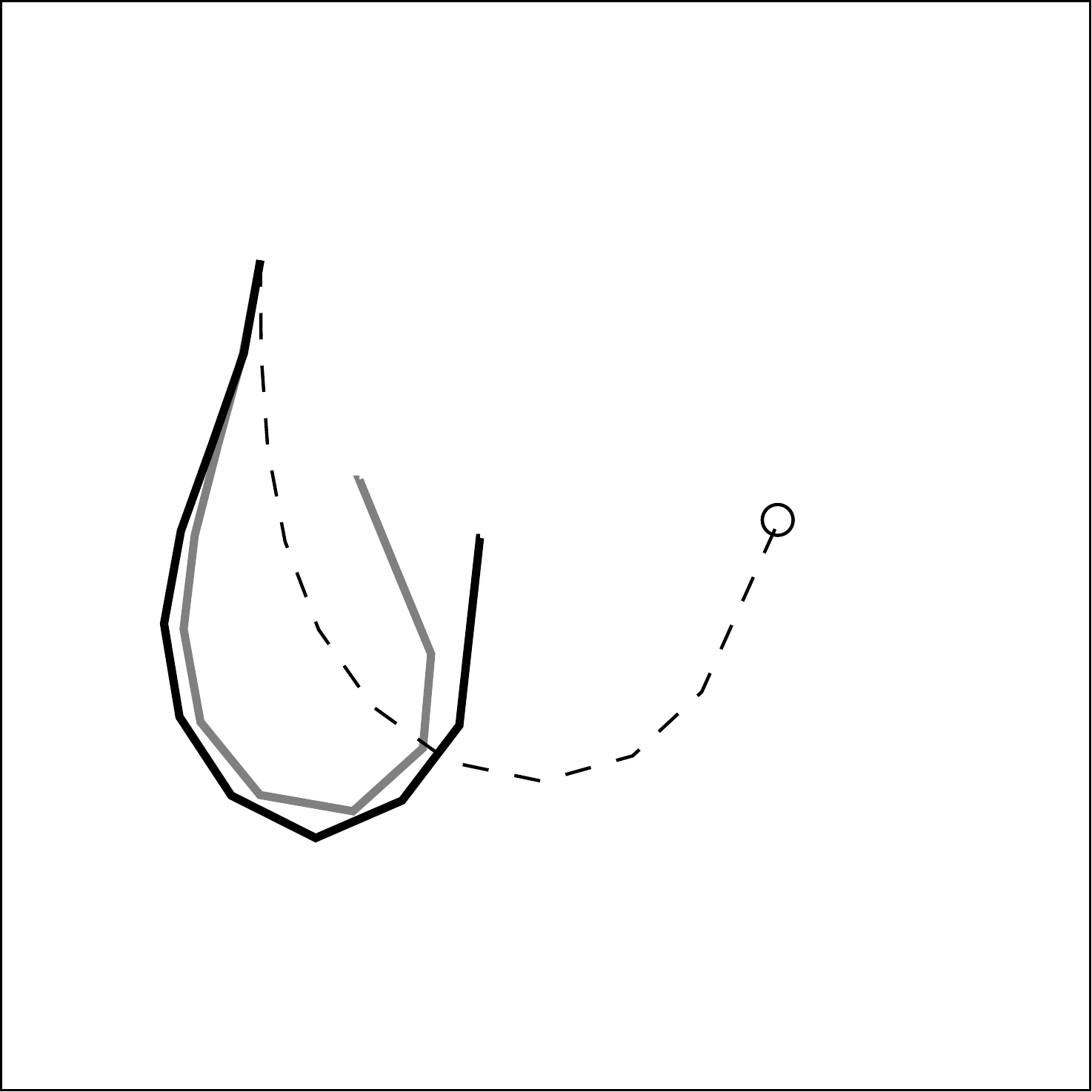}&
 \includegraphics[width=.1725\textwidth]{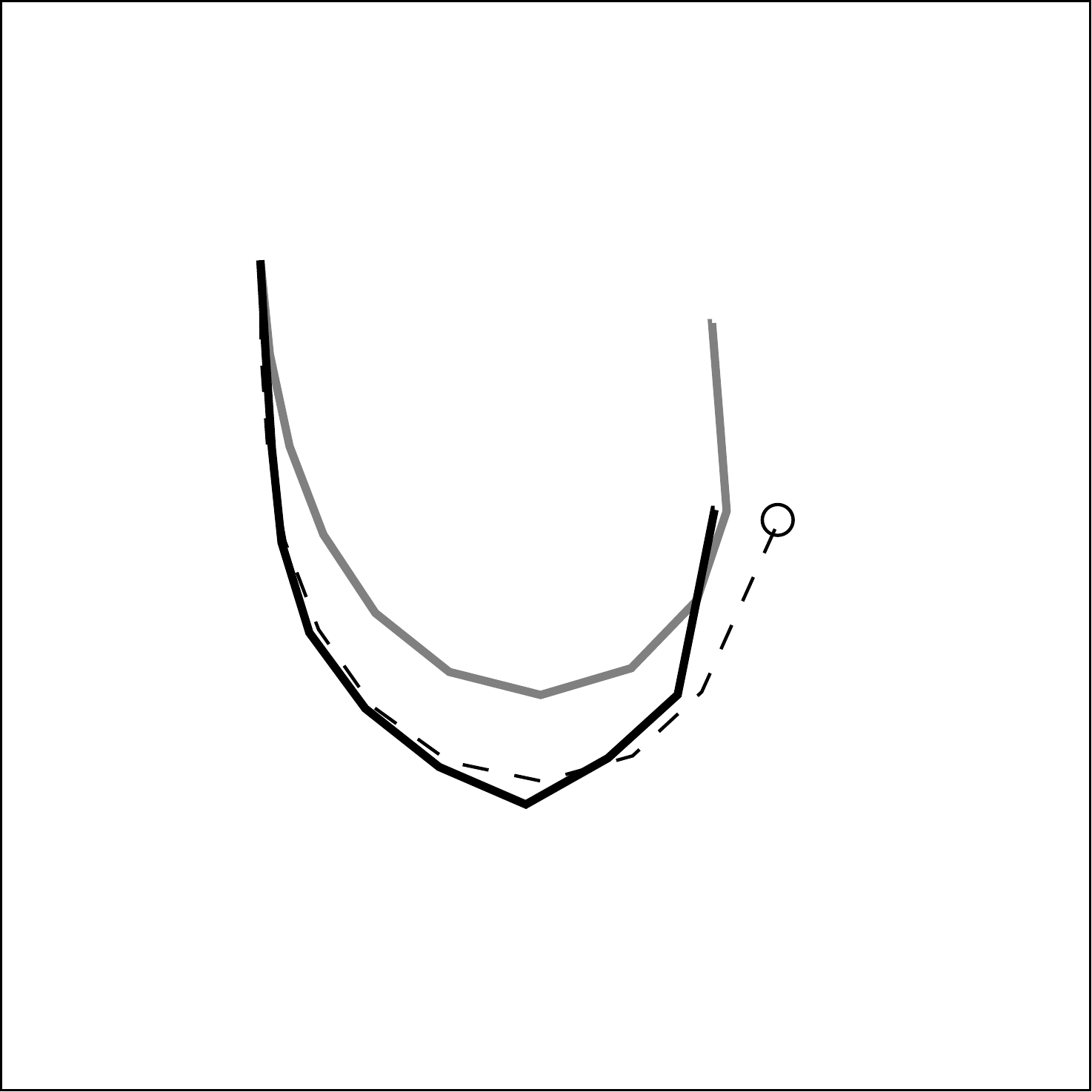}&
 \includegraphics[width=.1725\textwidth]{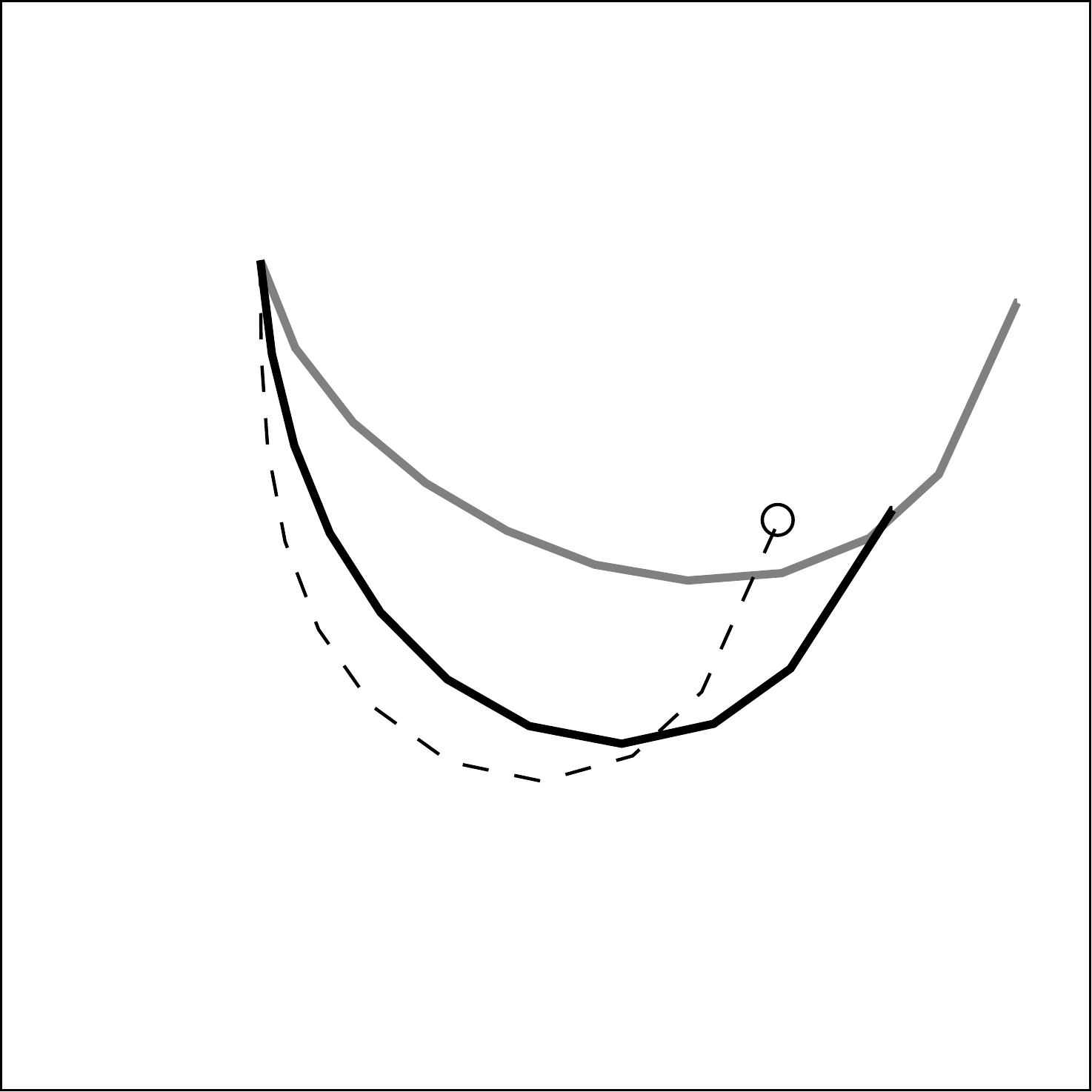}\\
 $t=0.1$ &$t=0.2$ &$t=0.3$ &$t=0.5$ &$t=0.7$\\
 \includegraphics[width=.1725\textwidth]{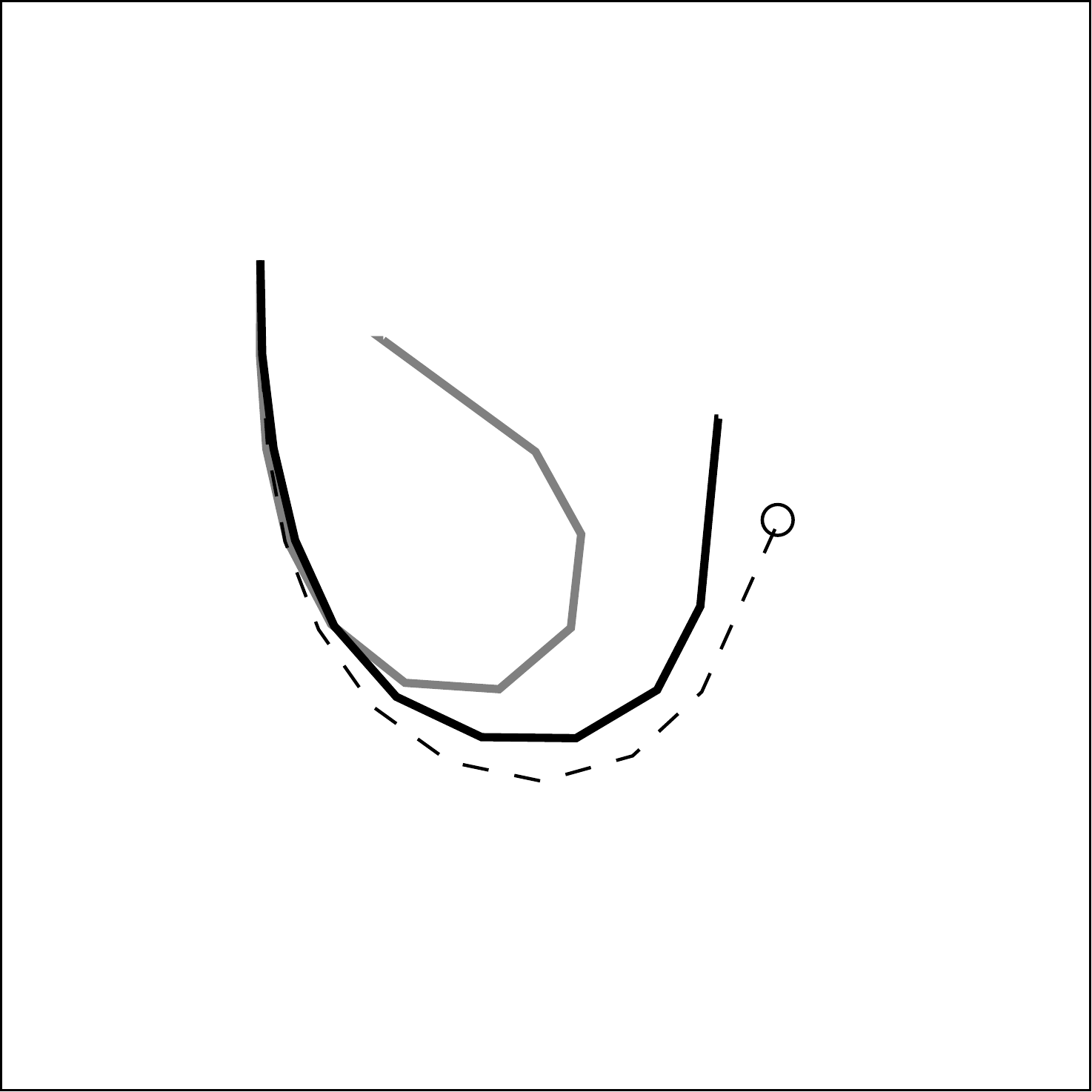}&
 \includegraphics[width=.1725\textwidth]{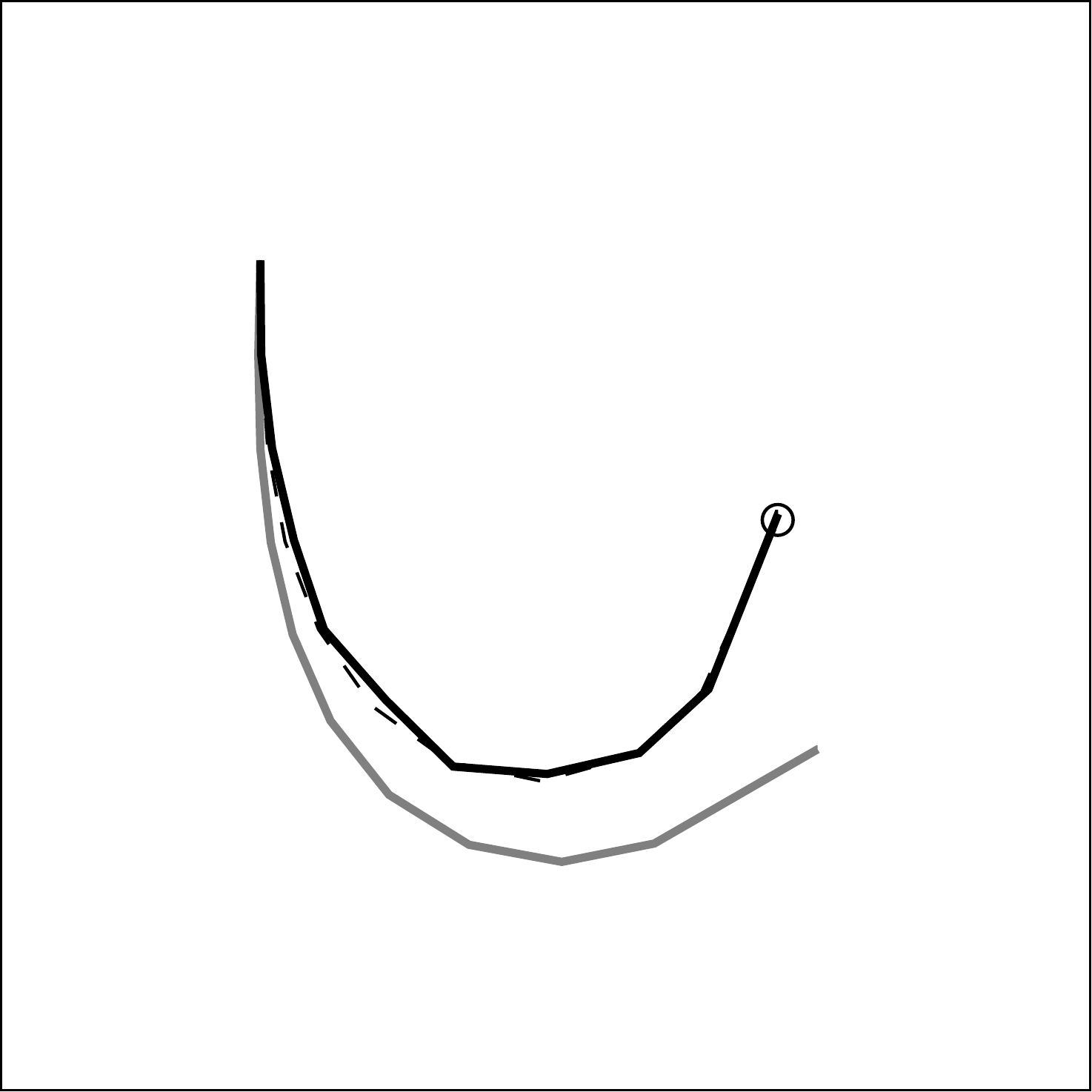}&
 \includegraphics[width=.1725\textwidth]{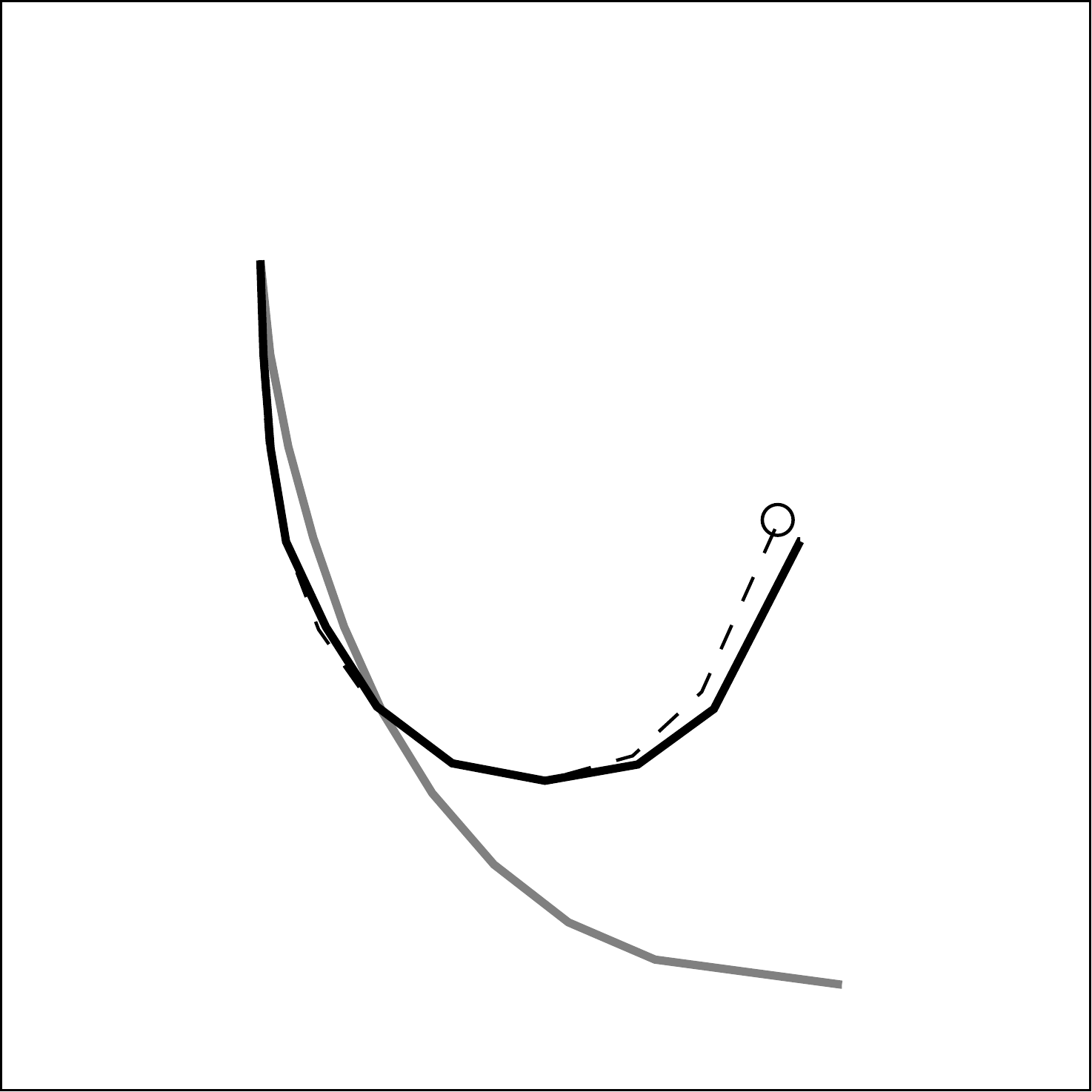}&
 \includegraphics[width=.1725\textwidth]{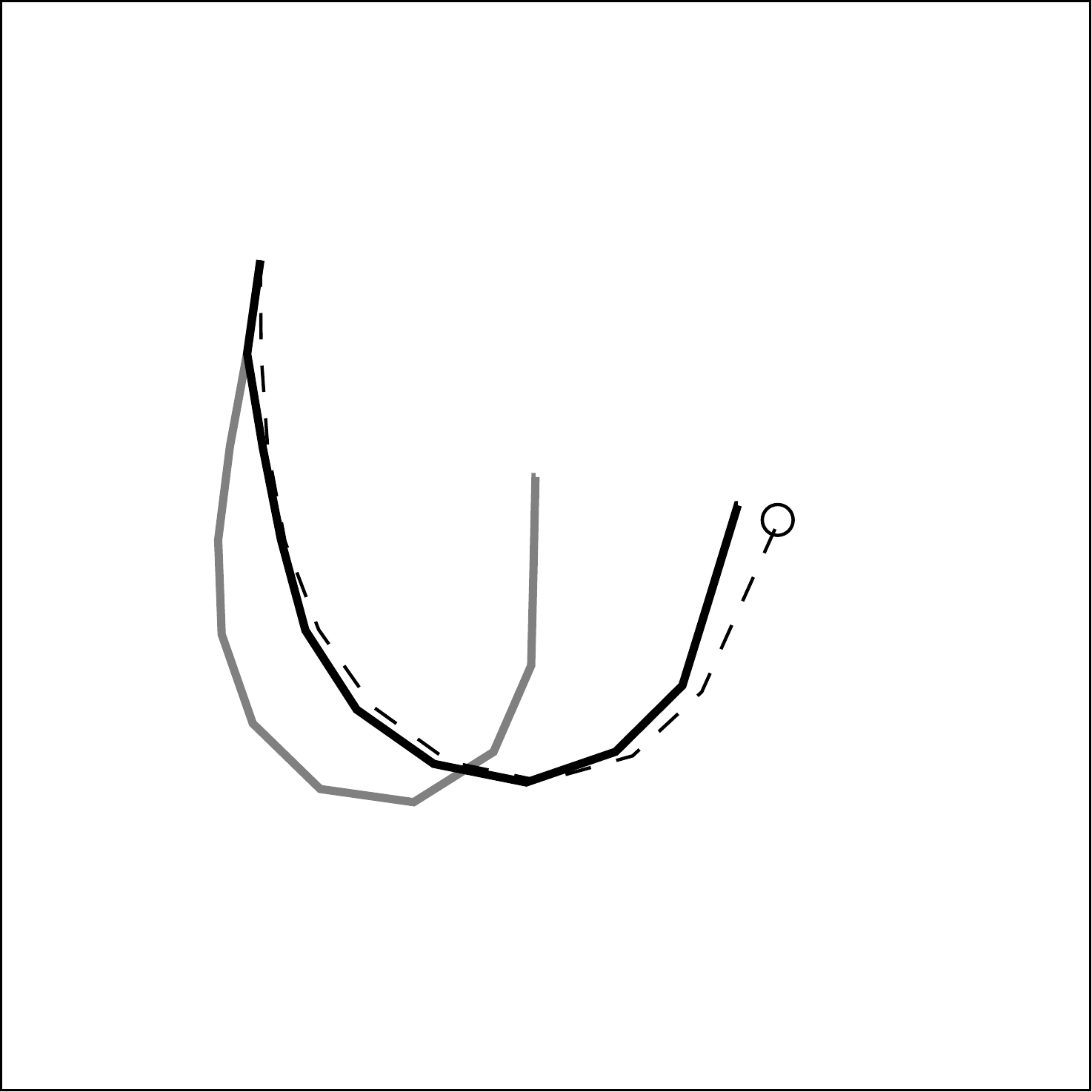}&
 \includegraphics[width=.1725\textwidth]{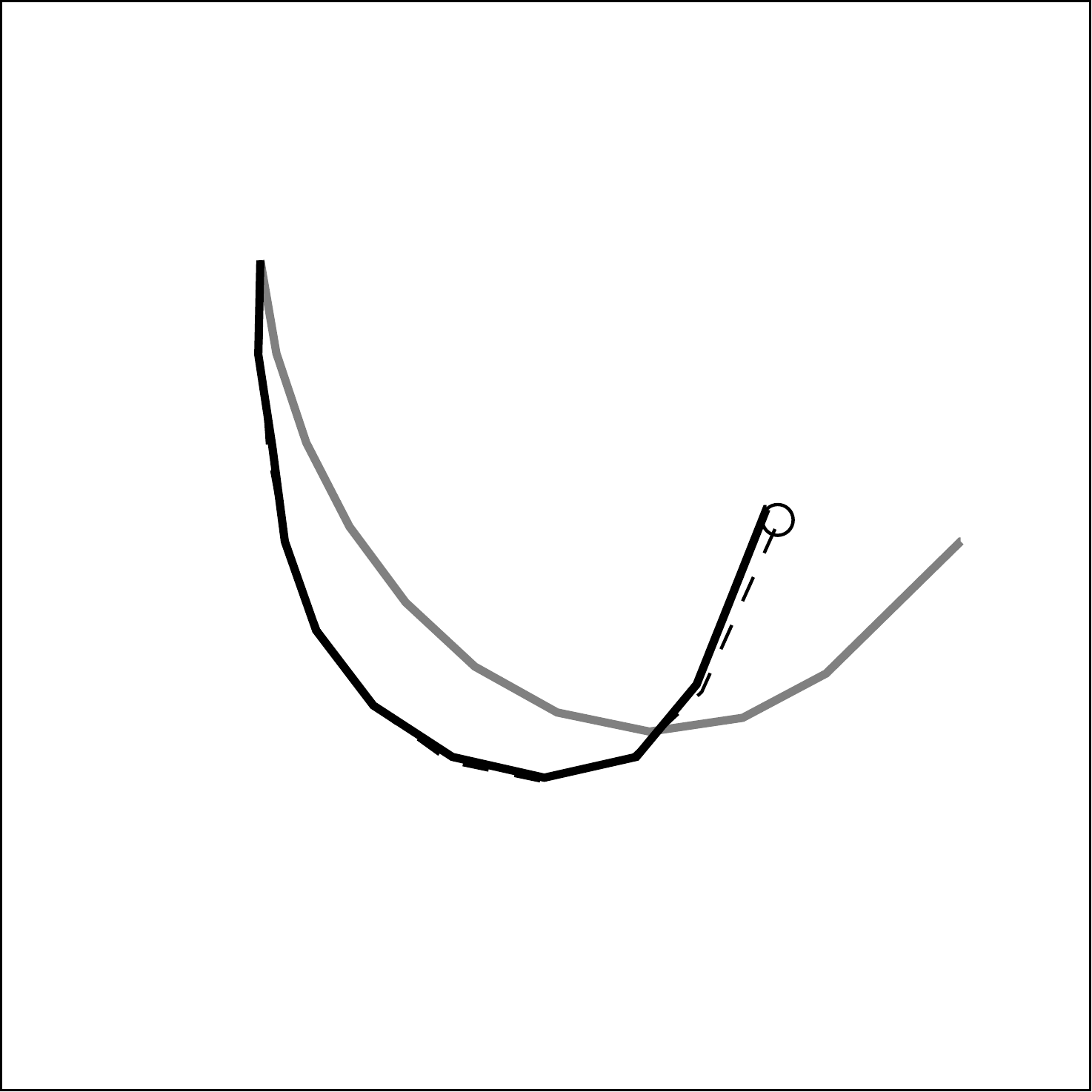}\\
 $t=1.0$ &$t=1.3$ &$t=1.4$ &$t=1.7$ &$t=2.1$\\
 \includegraphics[width=.1725\textwidth]{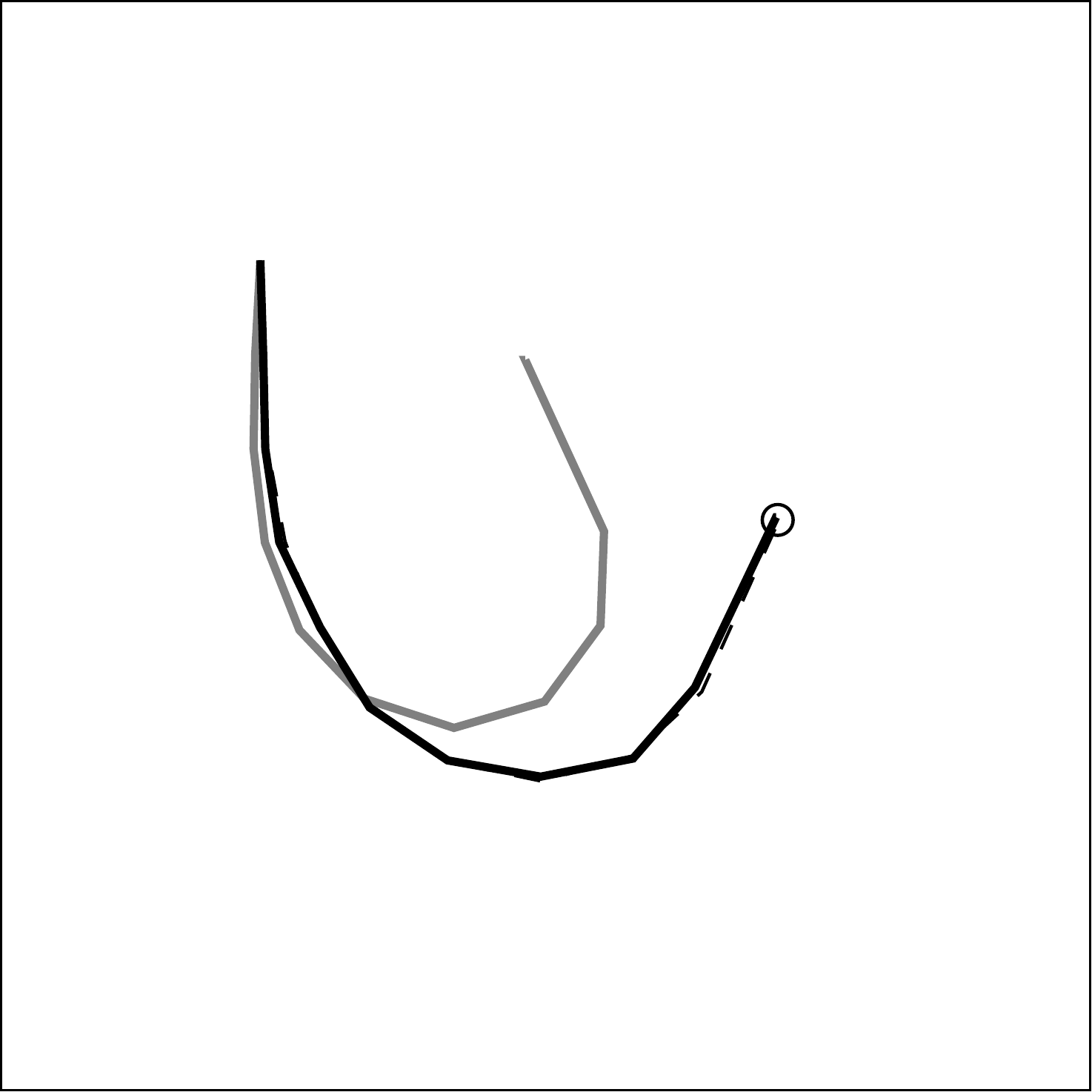}&
 \includegraphics[width=.1725\textwidth]{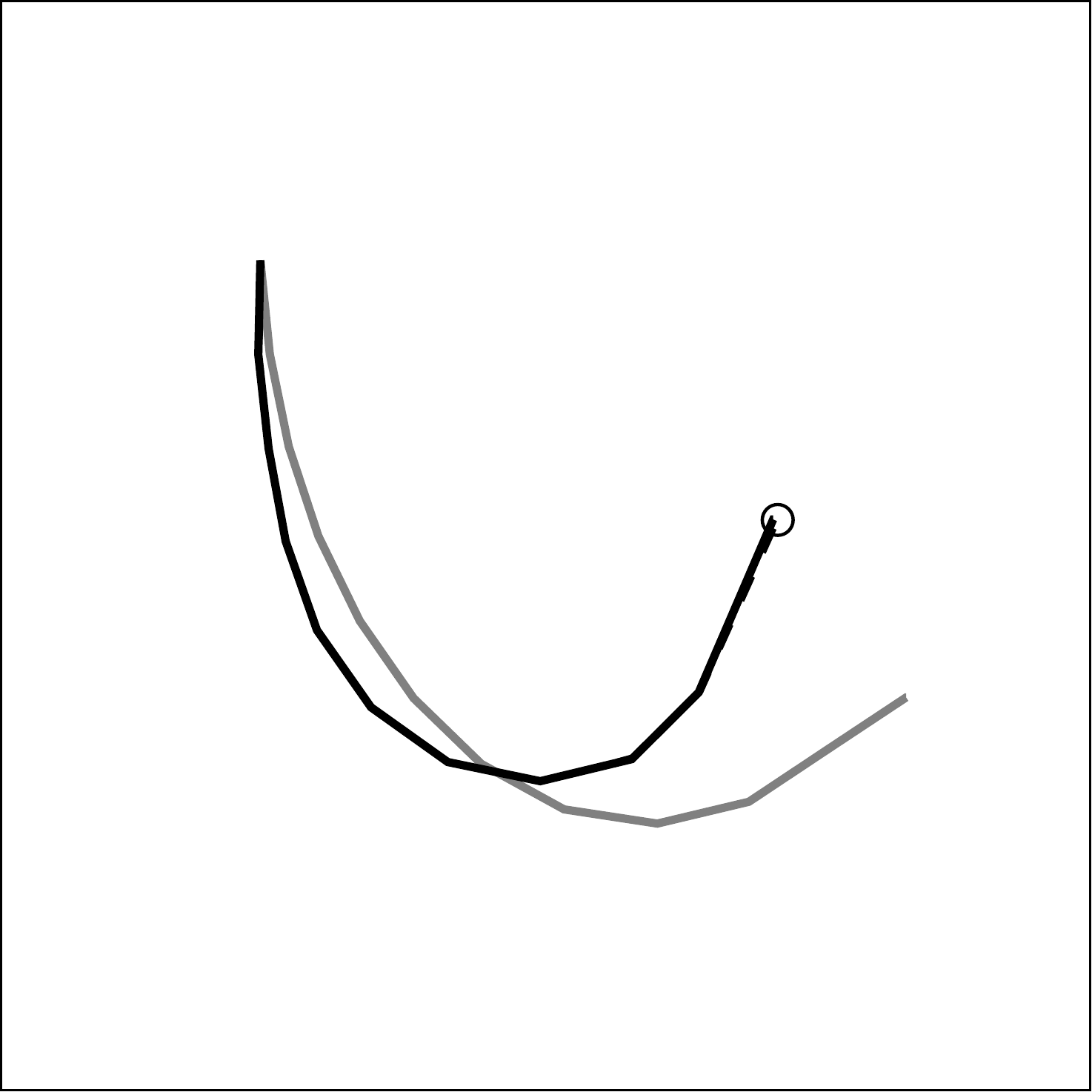}&
 \includegraphics[width=.1725\textwidth]{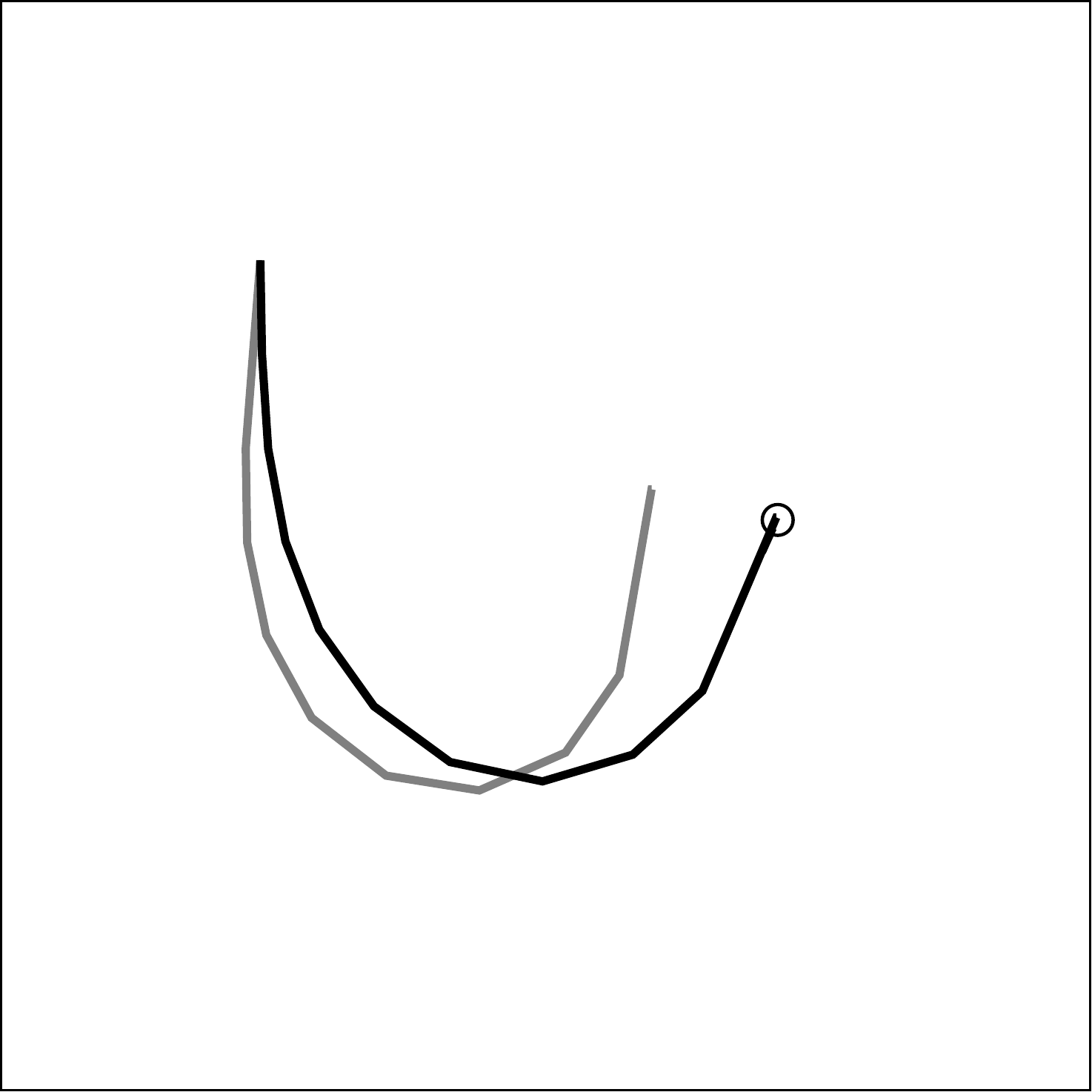}&
 \includegraphics[width=.1725\textwidth]{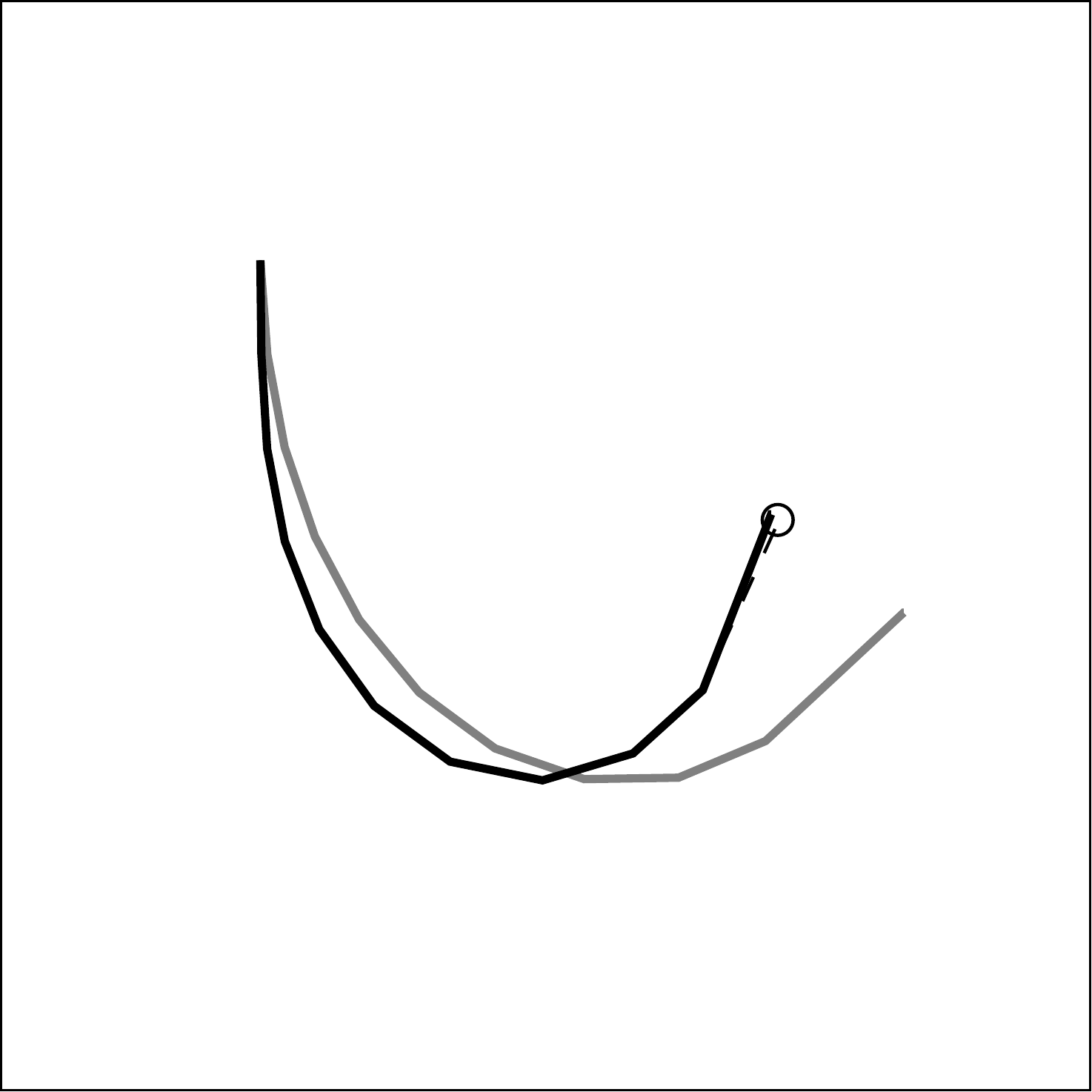}&
 \includegraphics[width=.1725\textwidth]{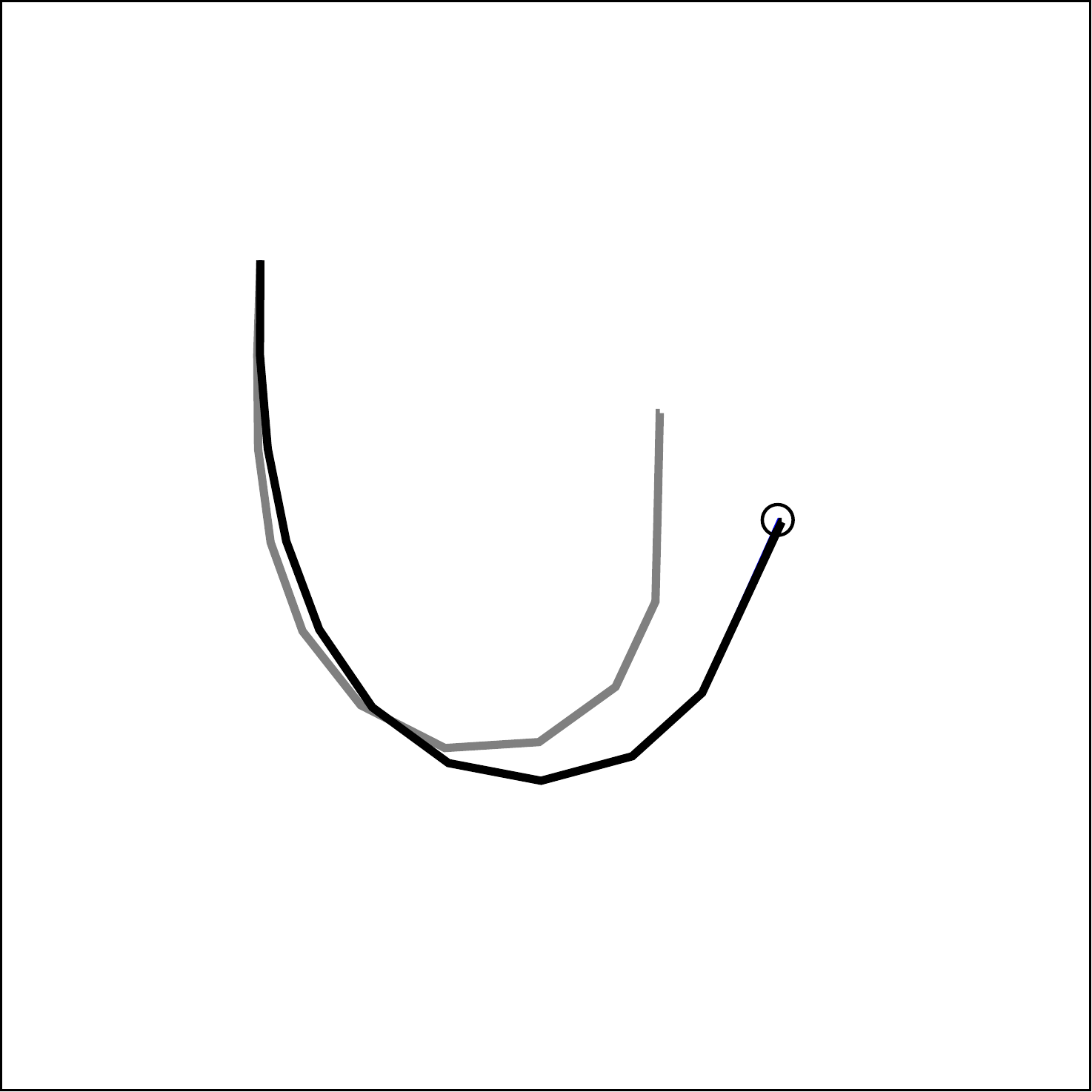}\\
  $t=2.5$ &$t=2.9$ &$t=3.2$ &$t=3.7$ &$t=4.0$
\end{tabular}
\caption{Static optimal control vs dynamic optimal control}\label{static-vs-dynamic}
\end{figure}
\begin{figure}[!h]
\centering
\begin{tabular}{ccc}
 \includegraphics[width=.28\textwidth]{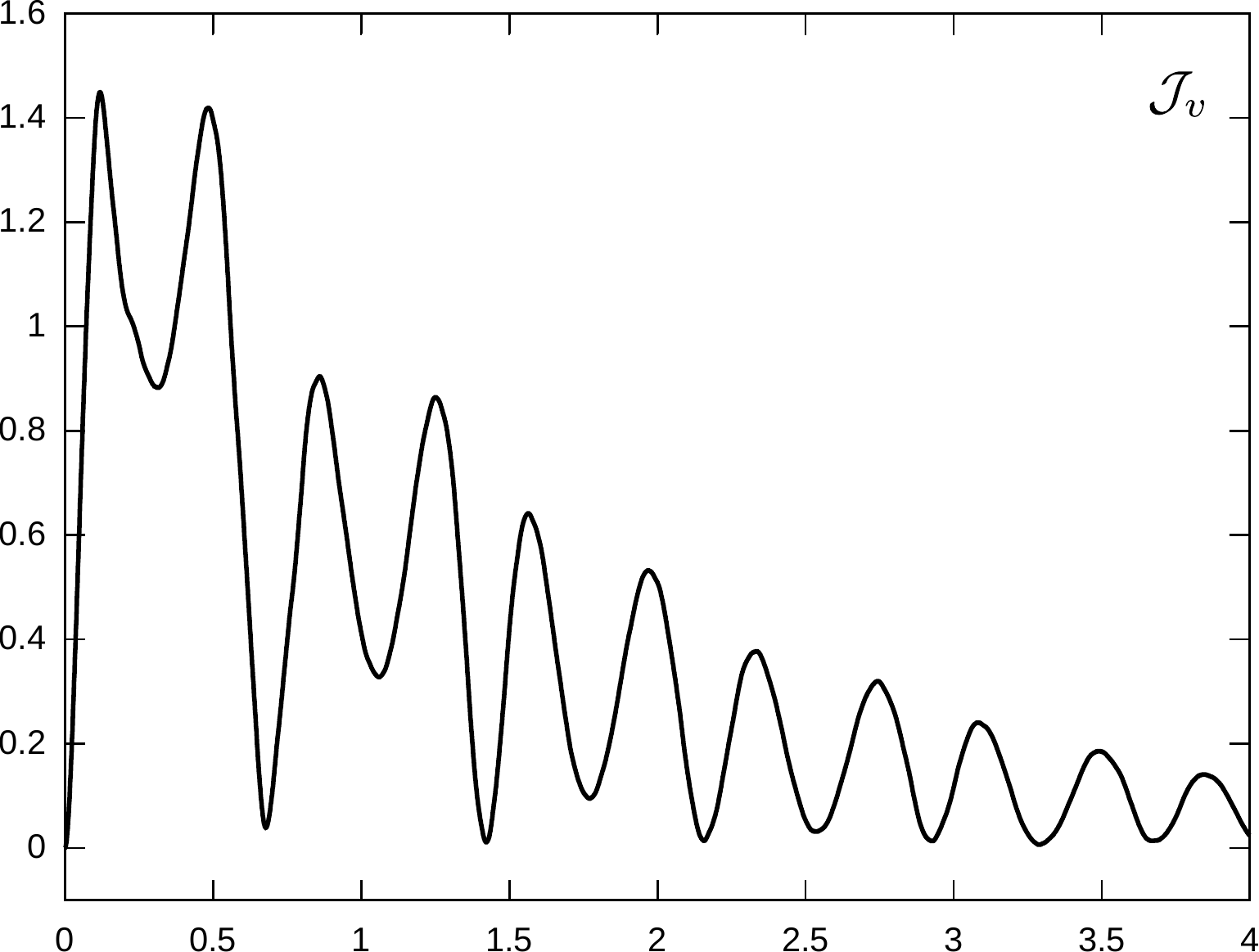}&
 \includegraphics[width=.28\textwidth]{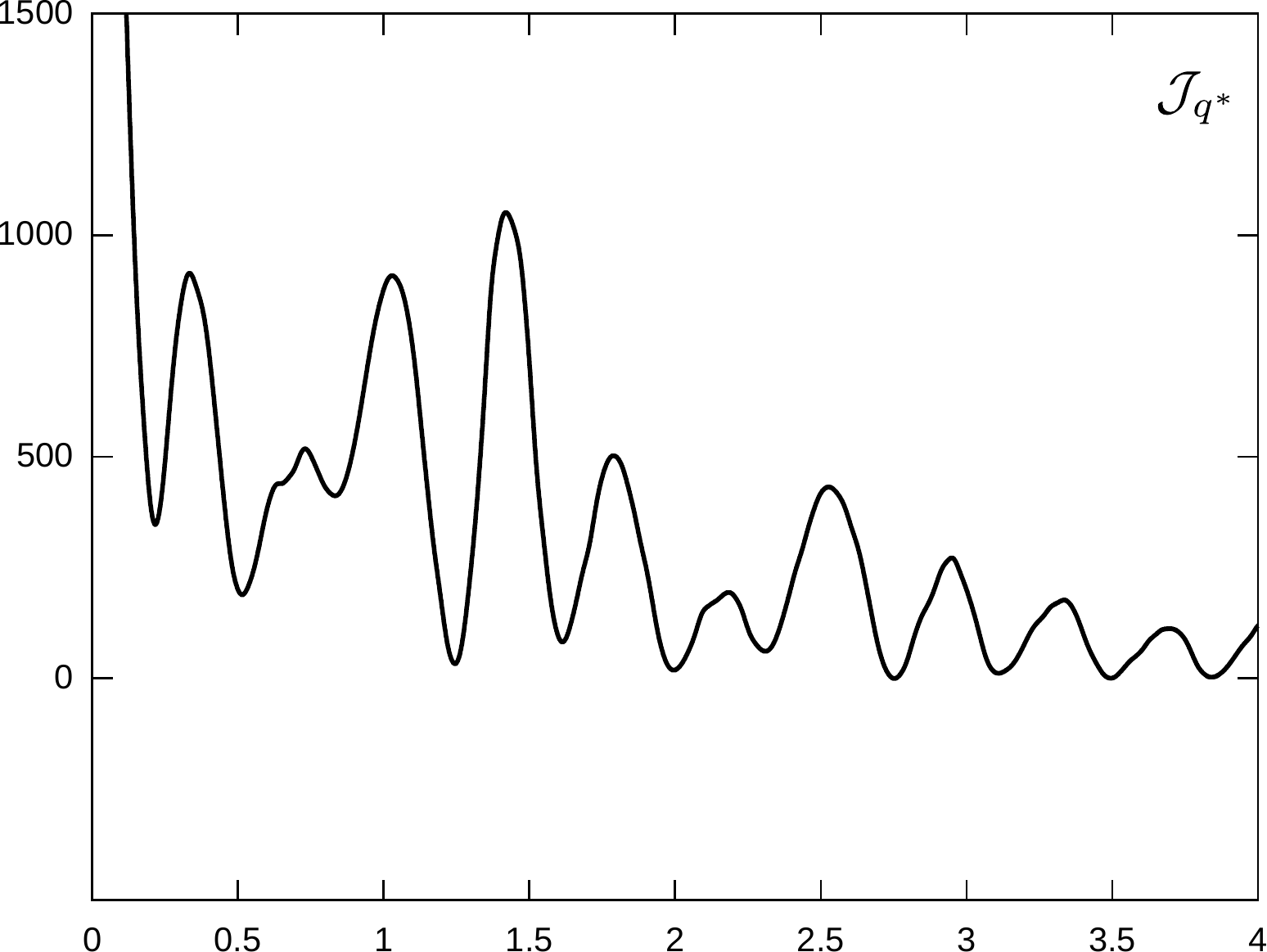}&
 \includegraphics[width=.28\textwidth]{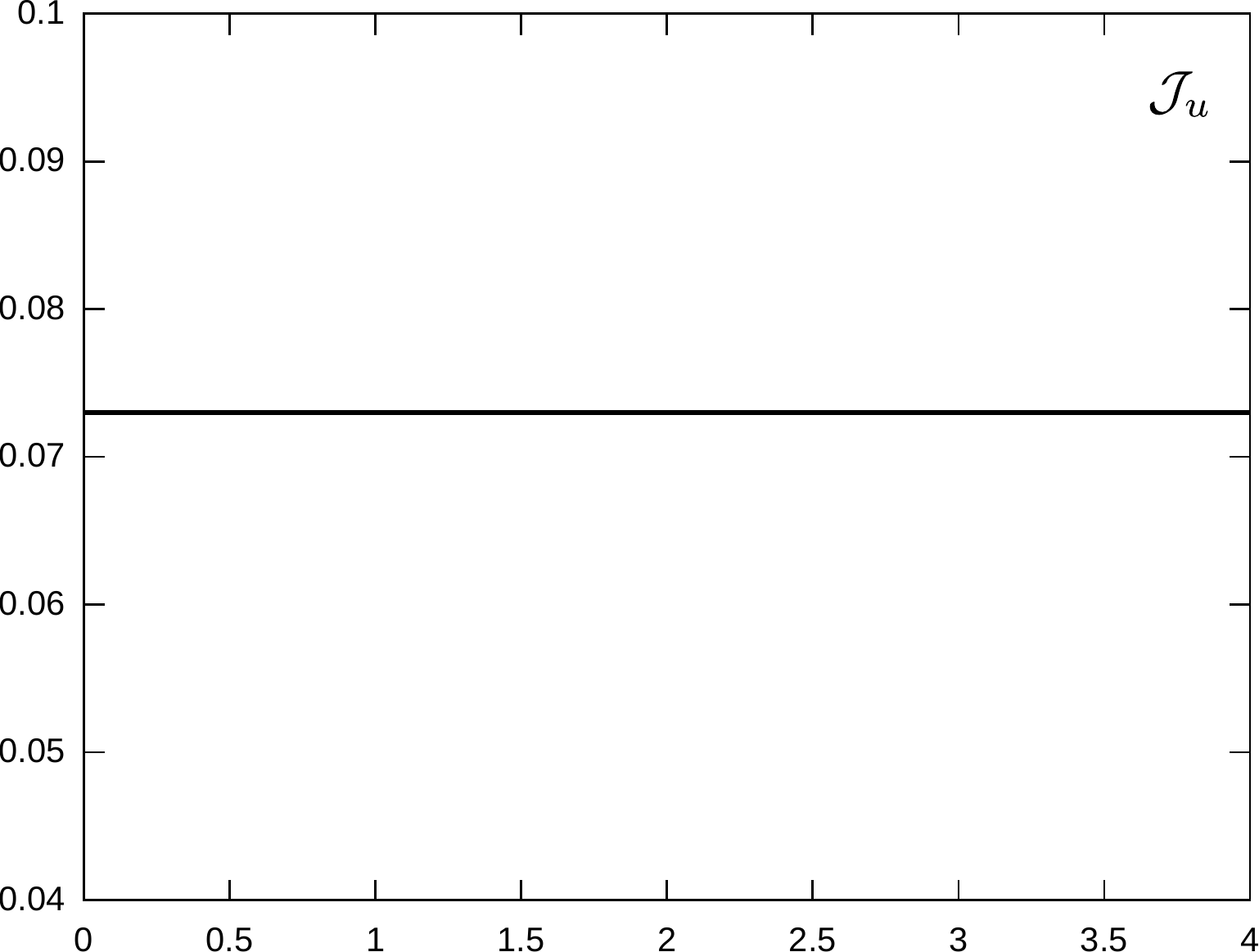}\\
 (a)&(b)&(c)
\end{tabular}
\caption{Time evolution of kinetic energy $\mathcal{J}_{v}$ (a), target energy $\mathcal{J}_{q^\ast}$ (b) and control energy $\mathcal{J}_{u}$ (c) for the static optimal control}\label{energy-static}
\end{figure}
\begin{figure}[!h]
\centering
\begin{tabular}{ccc}
 \includegraphics[width=.28\textwidth]{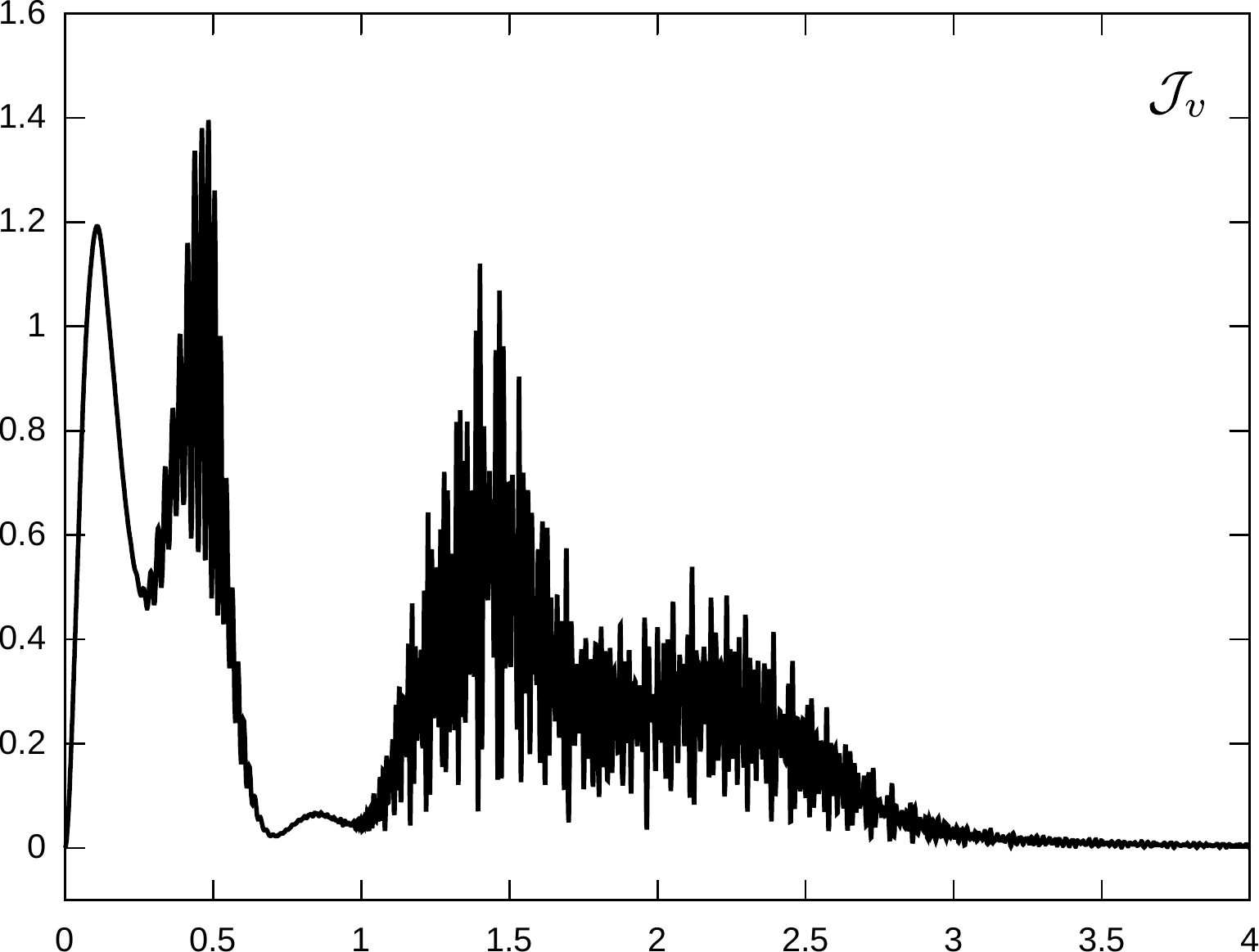}&
 \includegraphics[width=.28\textwidth]{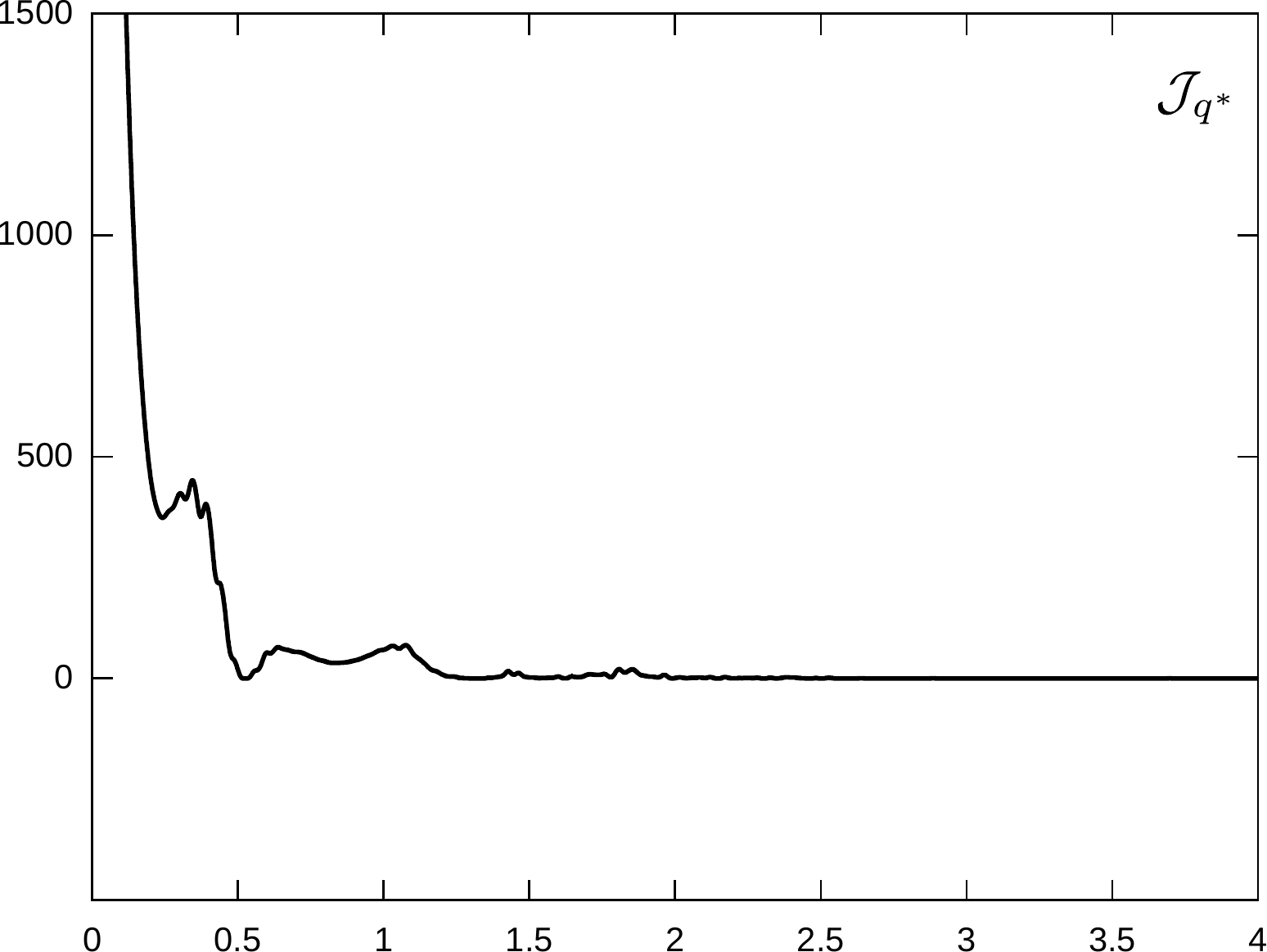}&
 \includegraphics[width=.28\textwidth]{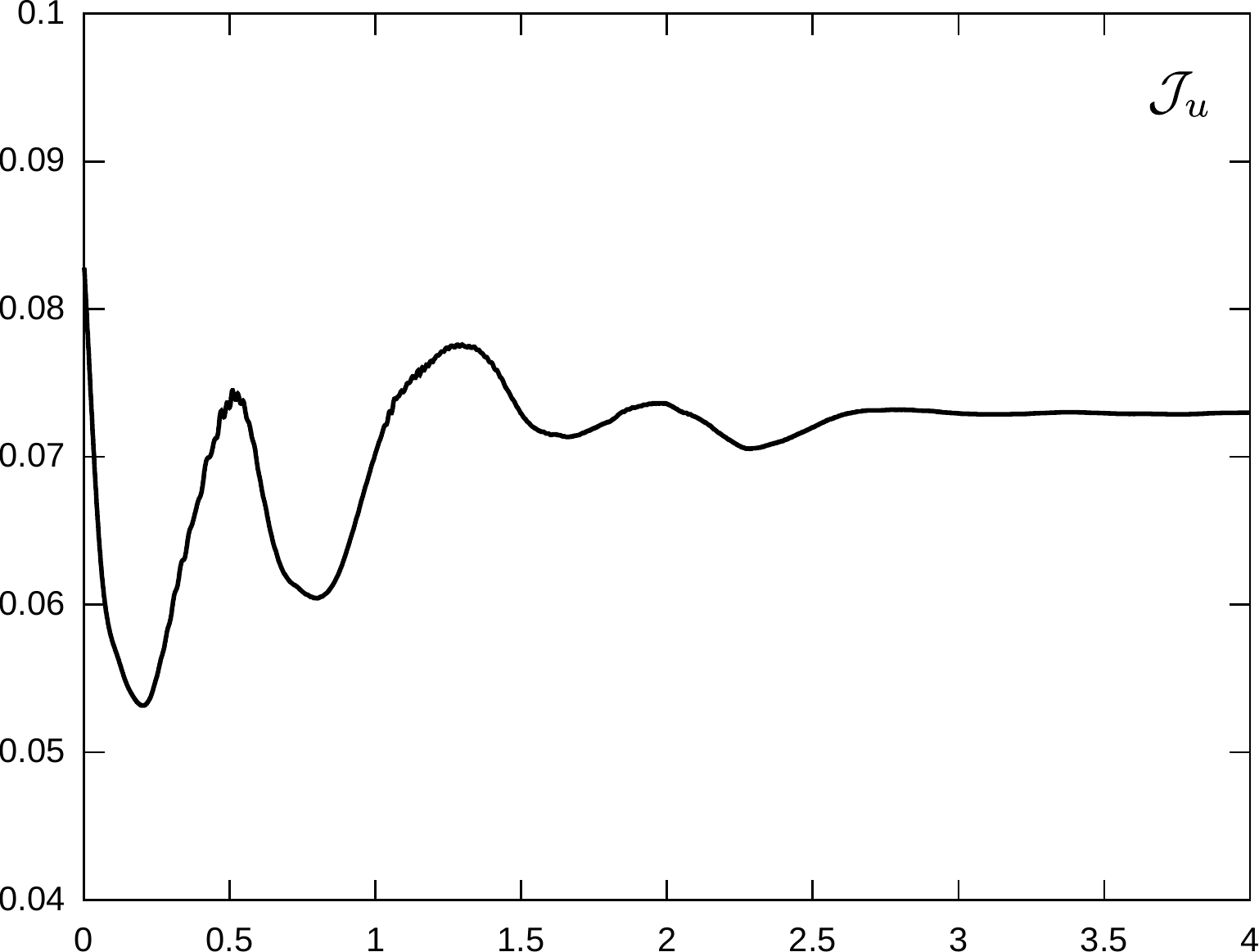}\\
 (a)&(b)&(c)
\end{tabular}
\caption{Time evolution of kinetic energy $\mathcal{J}_{v}$ (a), target energy $\mathcal{J}_{q^\ast}$ (b) and control energy $\mathcal{J}_{u}$ (c) for the dynamic optimal control}\label{energy-dynamic}
\end{figure}\\
\newcommand{\etalchar}[1]{$^{#1}$}

\end{document}